\documentclass{amsart}
\usepackage [latin1]{inputenc}
\usepackage{categorytheory,amssymb,graphicx, comment}
\usepackage{todonotes}
\usepackage{amsaddr}
\usepackage[all]{xy}
\newcommand{\dotp}{{{\raisebox{.2ex}{\scalebox{.5}{$\bullet$}}}}}

\fullpage

\newcommand{\Var}{\mathbf{Var}}
\newcommand{\B}{\mathcal{B}}
\newcommand{\A}{\mathcal{A}}
\DeclareMathOperator{\Ar}{Ar}

\newcommand{\F}{\mathcal{F}}
\newcommand{\mc}{\mathcal}
\newcommand{\W}{\mathcal{W}}
\newcommand{\sH}{\mathscr{H}}

\newcommand{\Cat}{\mathbf{Cat}}
\newcommand{\Sch}{\mathbf{Sch}}
\newcommand{\bs}{\backslash}
\newcommand{\mbf}{\mathbf}
\newcommand{\I}{\mathcal{I}}
\newcommand{\J}{\mathcal{J}}

\DeclareMathOperator{\Mod}{Mod}
\DeclareMathOperator{\SC}{SC}

\renewcommand{\emph}{\textsl}

\usetikzlibrary{arrows.meta}
\usetikzlibrary{cd} 
\usetikzlibrary{decorations.markings, decorations.pathmorphing, shapes}

\tikzset{
  mmor/.style={angle 45 reversed->},
  backmmor/.style={<-angle 45 reversed},
  emor/.style={->, decoration={markings,%
      mark=at position 2pt with {\node[circle, draw=black, fill=white,%
        inner sep=1.2pt] {};}},%
    postaction=decorate},
  backemor/.style={<-,decoration={markings,%
      mark=at position -2pt with {\node[circle, draw=black, fill=white,%
        inner sep=1.2pt] {};}},%
    postaction=decorate},
  Dmor/.style={decoration={markings,%
      mark=at position 0.5 with {\node[circle, draw=black, fill=black,%
        inner sep=1pt] {};}},%
    postaction=decorate},
  smor/.style={densely dotted
  },
}

\def\mrto{\inlineArrow{mmor}}
\def\mlto{\inlineArrow{backmmor}}
\def\erto{\inlineArrow{emor}}
\def\elto{\inlineArrow{backemor}}
\def\eDlto{\inlineArrow{{backemor,Dmor}}}
\def\eDrto{\inlineArrow{{emor,Dmor}}}
\def\mDlto{\inlineArrow{{backmmor,Dmor}}}
\def\msrto{\inlineArrow{{mmor,smor}}}
\def\srto{\inlineArrow{{smor,->}}}
\def\esrto{\inlineArrow{{smor,emor}}}

\def\mDrto{\inlineArrow{{mmor,Dmor}}}
\def\mto{\diagArrow{mmor}}
\def\mDrto{\inlineArrow{{mmor,Dmor}}}
\def\eto{\diagArrow{emor}}
\def\eDto{\diagArrow{emor,Dmor}}
\def\mDto{\diagArrow{mmor,Dmor}}
\newcommand{\labar}[3]{\diagArrow{-,color=white}{#1}{#2}!{\textcolor{black}{#3}}}
\newcommand{\dist}[2]{\labar{#1}{#2}{\square}}
\newcommand{\comm}[2]{\labar{#1}{#2}{\circlearrowleft}}

\newcommand{\defeq}{\stackrel{\mathrm{def}}=}

\shownotes

\newcommand{\dsqinline}[8]{\begin{inline-diagram}{#1 \pgfmatrixnextcell #2 \\ #3%
      \pgfmatrixnextcell #4 \\};%
    \dist{1-1}{2-2} %
    \mto{1-1}{1-2}^{#5} \mto{2-1}{2-2}^{#8} %
    \eto{1-1}{2-1}_{#6} \eto{1-2}{2-2}^{#7} %
  \end{inline-diagram}}
\newcommand{\dsqdisplay}[8]{\begin{tikzpicture}[baseline]%
      \matrix (m) [matrix of math nodes,row sep=2.6em, column sep=2.8em] { %
        #1 \pgfmatrixnextcell #2 \\ #3 \pgfmatrixnextcell #4 \\}; %
      \dist{1-1}{2-2} %
      \mto{1-1}{1-2}^{#5} %
      \mto{2-1}{2-2}^{#8} \eto{1-1}{2-1}_{#6} \eto{1-2}{2-2}^{#7}%
    \end{tikzpicture}}
\newcommand{\dsq}[8]{\mathchoice%
  {\dsqdisplay{#1}{#2}{#3}{#4}{#5}{#6}{#7}{#8}}%
  { \dsqinline{#1}{#2}{#3}{#4}{#5}{#6}{#7}{#8}}%
  { \dsqinline{#1}{#2}{#3}{#4}{#5}{#6}{#7}{#8}}%
  { \dsqinline{#1}{#2}{#3}{#4}{#5}{#6}{#7}{#8}}%
}
\newcommand{\csqinline}[8]{\begin{inline-diagram}{#1 \pgfmatrixnextcell #2 \\ #3%
      \pgfmatrixnextcell #4 \\};%
    \comm{1-1}{2-2}
    \mto{1-1}{1-2}^{#5} \mto{2-1}{2-2}^{#8} %
    \eto{1-1}{2-1}_{#6} \eto{1-2}{2-2}^{#7} %
  \end{inline-diagram}}
\newcommand{\csqdisplay}[8]{\begin{tikzpicture}[baseline]%
      \matrix (m) [matrix of math nodes,row sep=2.6em, column sep=2.8em] { %
        #1 \pgfmatrixnextcell #2 \\ #3 \pgfmatrixnextcell #4 \\}; %
      \comm{1-1}{2-2}
      \mto{1-1}{1-2}^{#5} %
      \mto{2-1}{2-2}^{#8} \eto{1-1}{2-1}_{#6} \eto{1-2}{2-2}^{#7}%
    \end{tikzpicture}}
\newcommand{\csq}[8]{\mathchoice%
  {\csqdisplay{#1}{#2}{#3}{#4}{#5}{#6}{#7}{#8}}%
  { \csqinline{#1}{#2}{#3}{#4}{#5}{#6}{#7}{#8}}%
  { \csqinline{#1}{#2}{#3}{#4}{#5}{#6}{#7}{#8}}%
  { \csqinline{#1}{#2}{#3}{#4}{#5}{#6}{#7}{#8}}%
}

\title{D\'evissage and Localization for the Grothendieck Spectrum of Varieties}
\author{Jonathan A. Campbell}
\address{Center for Communications Research, La Jolla}
\author{Inna Zakharevich}
\address{Cornell University} 

\begin{document}

\begin{abstract}
  We introduce a new perspective on the $K$-theory of exact categories via the
  notion of a \emph{CGW-category}.  CGW-categories are a generalization of
  exact categories that admit a Quillen $Q$-construction, but which also include
  examples such as finite sets and varieties.  By analyzing Quillen's proofs of
  d\'evissage and localization we define \emph{ACGW-categories}, an analogous
  generalization of abelian categories for which we prove theorems akin to
  d\'evissage and localization.  In particular, although the category of varieties is not
  quite ACGW, the category of reduced schemes of finite type is; applying
  d\'evissage and localization allows us to calculate a filtration on the
  $K$-theory of schemes of finite type.  As an application of this theory we
  construct a comparison map showing that the two authors' definitions of the
  Grothendieck spectrum of varieties are equivalent. 
\end{abstract}
\maketitle

\tableofcontents


\section{Introduction}


On August 16, 1964, Grothendieck wrote to Serre of a conjectured category of
motives.  Such a category (called $\mathbf{M}(k)$) would encode schemes up to
decomposition (by cutting out subvarieties), but would itself be an abelian
category capturing the \emph{cohomological} structures involved.
\begin{quotation}
  The sad truth is that for the moment I do not know how to define the abelian
  category of motives, even though I am beginning to have a rather precise yoga
  for this category. For example, for any prime $\ell\neq p$, there is an exact
  functor $T_\ell$ from $\mathbf{M}(k)$ into the category of finite-dimensional
  vector spaces over $\Q$ on which the pro-group $\mathrm{Gal}(\bar k_i/k_i)_i$
  acts, where $k_i$ runs over subextensions of finite type of $k$ and $\bar k_i$
  is the algebraic closure of $k_i$ in $\bar k$; this functor is faithful but
  not, of course, fully faithful\ldots I will not venture to make any general
  conjecture on the above homomorphism; I simply hope to arrive at an actual
  construction of the category of motives via this kind of heuristic
  considerations, and this seems to me to be an essential part of my ``long run
  program.'' \cite[p 174-175]{grothendieckserre}
\end{quotation}
Grothendieck's letter proposes several other properties of this conjectured
category, and discusses his attempts at the construction.  Since then, there
have been many other approximations to construct this category---for an overview
see, for example, \cite{milne_motives}---but all fall short of the ideal.  

Grothendieck's approach begins with the construction of a ``$K$-group'' of
varieties.  These days, this is known as the Grothendieck ring of varieties,
denoted $K_0 (\Var_k)$.  It is generated by isomorphism classes of
$k$-varieties, $[X]$, subject to the relations that
$[X] = [Z] + [X \setminus Z]$ for closed inclusions $Z \hookrightarrow
X$. Kontsevich, following Drinfeld \cite{kontsevich}, calls this the ring of
``poor man's motives.'' He notes that any reasonable abelian category of
motives, $\mc{M}_k$, will have a map $K_0 (\Var_k) \to K_0 (\mc{M}_k)$. For
example, in \cite[Thm. 4]{gillet_soule}, Gillet and Soul\'e show that there is a
group homomorphism $K_0 (\Var_k) \to K_0 (\mbf{M}_{\sim})$ where $\mbf{M}_{~}$
is the category of (pure) motives associated to the equivalence relation
$\sim$. It is thus useful to understand $K_0 (\Var_k)$ in a deep way in order to
learn more about how motives should work. It is even better to understand how it
behaves in relation to abelian categories.

We move toward such an understanding in this paper. Before doing so, we rephrase
the question. The Grothendieck group of an abelian category is a shadow of the
much richer structure of Quillen's higher algebraic $K$-theory
\cite{quillen}. Thus there should in fact exist a map on higher algebraic
$K$-theory spectra $K(\Var_k) \to K(\mc{M}_k)$ provided that one can define the
objects in the map. It is currently far beyond the state of the art to attempt
to understand the right-hand side. However, the authors separately have come up
with models for the left \cite{campbell, zakharevich_assembler}.  Under these
constructions the category of varieties behaves very similarly to an abelian
category, and one may be tempted to conjecture that from some novel perspective
the category of varieties would ``become'' abelian.

Our goal in this paper is to construct such a perspective. This has the added
benefit of putting all objects of interest on the same footing.  Our perspective
begins with thinking of sequences $Z \hookrightarrow X \leftarrow X\setminus Z$
as our ``exact sequences.''  It turns out that with this perspective one can
execute nearly all constructions that one enjoys in abelian categories: kernels,
cokernels, localizations, etc. The main insight is that we should not think of
these constructions algebraically, but in a kind of diagrammatic calculus, where
one of the arrows points the opposite way that one would expect. Such
diagrammatic calculi are, of course, the foundation of Grothendieck's seminal
Tohoku paper \cite{tohoku}.

While we do not develop the general theory of homological algebra of these types
of categories,\footnote{We hope to develop this in future work; it appears to be
  that homological algebra ends up working almost identically to the classical
  theory.}  one can ask which $K$-theoretic theorems hold. Pondering the
fundamental theorems of Quillen's algebraic $K$-theory, we come to the following
desiderata for the construction of $K$-theories of geometric and algebraic
objects:
\begin{enumerate}
\item The categorical machinery should somehow encompass \textit{both} the
  category of varieties with its ``exact sequences'' defined above, and
  Quillen's exact categories \cite[p.92]{quillen}.
\item D\'evissage should hold: Given an inclusion of categories
  $\mc{A} \subset \mc{B}$ such that everything in $\mc{B}$ can be ``broken up''
  into objects in $\mc{A}$, there should be an equivalence
  $K(\mc{A}) \simeq K(\mc{B})$.
\item Localization should hold: given two such categories
  $\mc{A} \subset \mc{B}$, one should be able to produce a localized category
  $\mc{B} / \mc{A}$ as one can with abelian categories. One would also like a
  localization sequence
  \[
    K(\mc{A}) \to K(\mc{B}) \to K(\mc{B}/\mc{A})
  \]
  as in \cite[Thm. 5]{quillen}.
\end{enumerate}

In this paper we show that there \textit{is} such a categorical structure, and
we are able to satisfy the requirements listed above. Moreover, it has the
correct ``yoga'': we are able to not only make the theorems work, but also
Quillen's original proofs.  Although this does not get us much closer to
understanding the conjectural category of motives, it does provide us with a new
perspective and concrete technical tools. The perspective could be summarized as
follows: varieties, together with the exact sequences above, behave
\textit{almost} like abelian categories and one should work with this structure
for as long as possible before passing to abelian categories. As will be shown
below, this perspective is extremely fruitful when discussing algebraic
$K$-theory, and we expect it to be more useful generally.

The fundamental notion introduced in this paper is that of a
\emph{CGW-category}. It is essentially a category equipped with two subclasses
of maps, $\M$ and $\E$ (to be thought of as analogous to admissible
monomorphisms and admissible epimorphisms in exact categories), together with
distinguished squares that tell us how objects are built. In all examples we
know, the horizontal and vertical morphisms need not compose in the category,
and therefore we situate the classes $\M$ and $\E$ in a double category. With
this minimal amount of data we define $K$-theory following the classical constructions due to Quillen (Sect.~\ref{sect:q_construction}) or Waldhausen
(Sect.~\ref{subdivision}). We show that the resulting $K$-theory spaces have the
correct group of components in Thm.~\ref{thm:pi_0}.  CGW-categories satisfy requirement (1) above: they encompass varieties \emph{and} exact categories.  

Of course, as in the case of exact categories, additional structure is required
to prove these theorems. To this end we introduce the definition of an
ACGW-category, which is meant to be a sort of ``abelian'' version of a
CGW-category.  The category of reduced schemes of finite type is such a
category, with the category of varieties sitting inside it as a full
subcategory.  Roughly, an ACGW-category is a category that formally satisfies
all of the properties that open and closed sets do (the complement of a closed
set is open, you can intersect closed sets and union open sets, etc). Using this
definition we prove the first main theorem of the paper:

\begin{theorem}[D\'evissage]
  Let $\mathcal{A}, \mathcal{B}$ be ACGW-categories with
  $\mathcal{A} \subset \mathcal{B}$ satisfying certain technical
  conditions. Suppose every $B \in \mathcal{B}$ has a finite filtration $B_i$
  such that the difference between $B_i$ and $B_{i-1}$ lies in
  $\mathcal{A}$. Then $K(\mathcal{A}) \simeq K(\mathcal{B})$.
\end{theorem}
Here ``difference between'' could mean a quotient or a complement; for the
precise statement see Thm~\ref{thm:devissage}.  The definition of ACGW-category
has a number of requirements, but these requirements are satisfied by the
motivating examples of the category of reduced schemes of finite type, polytopes
\cite{zakharevich_assembler}, finite sets, and abelian categories.

The formal similarities between ACGW-categories and abelian categories suggest
that other theorems in algebraic $K$-theory can be extended to the CGW
case. Quillen's other major tool in algebraic $K$-theory is the localization
theorem, which relates the K-theories of two abelian categories $\mc{A}, \mc{B}$
with the $K$-theory of their quotient category $\mc{B} / \mc{A}$. A very similar
theorem holds for ACGW-categories:
\begin{theorem}[Localization]
  Let $\mc{C}$ be an ACGW category and $\mc{A}$ a sub-ACGW-category of $\mc{C}$
  satisfying certain technical conditions. Then there is a localization
  ACGW-category $\mc{C} \bs \mc{A}$ such that
  \begin{diagram}
    {K(\mc{A}) & K(\mc{C}) & K(\mc{C}\bs \mc{A})\\};
    \to{1-1}{1-2} \to{1-2}{1-3}
  \end{diagram}
  is a homotopy fiber sequence. 
\end{theorem}
For a more precise statement of this theorem, see Theorem~\ref{thm:qLoc}.  

An interesting observation about the proofs of these theorems is how closely
they follow Quillen's original proofs.  The category of varieties really does
``behave like'' an exact category, in the sense that many of the motions that
are necessary to prove theorems have direct analogs in the category of
varieties.  (In fact, the category of varieties lacks only ``pushouts'' to
behave like an abelian category; this is why switching to reduced schemes of
finite type is necessary.  For more detail on this, see Section~\ref{sec:acgw}.)

We expect there to be substantial applications of the d\'evissage and
localization theorems. The main application that we discuss in this paper is a
comparison of models for the K-theory of varieties that both authors have
constructed. Surprisingly, this theorem seems to use every bit of K-theoretic
machinery the authors have developed: assemblers, cofiber sequences in K-theory,
and the d\'evissage and localization theorems. All combine to give the following
theorem.
\begin{theorem}[Comparison]
  Let $K^C (\Var^n)$ denote the $K$-theory of the SW-category $\Var^n$ defined
  in \cite{campbell}, and let $K^Z(\Var^n)$ denote the $K$-theory of the assembler
  $\Var^n$ defined in \cite{zakharevich_assembler}. Then there is a zig-zag of
  weak equivalences
  \begin{diagram}
    {K^C(\Var^n) & \dotp &  K^Z(\Var^n).\\};
    \we{1-1.mid east}{1-2.mid west} \we{1-3.mid west}{1-2.mid east}
  \end{diagram}
\end{theorem}
For a more detailed statement of this theorem, see Theorem~\ref{thm:equiv}.

Each of the models constructed has their own strengths, and this theorem allows
us to pass between models to exploit these. We expect a more general theorem
relating Waldhausen-style $K$-theory to assembler style $K$-theory to be true,
but we leave that for future work.

Whether this new perspective leads to a new theory of motives or not is unclear;
however, the striking behavioral similarities between varieties and abelian
categories was too beautiful to leave unexplored.

\medskip
\noindent\textbf{Acknowledgements.}
The authors would like to thank Pierre Deligne, Andr\'{e} Joyal, Andrew
Blumberg, and Charles Weibel for interesting conversations related to this work.
They also thank Daniel Grayson and Karl Schwede for their patience with our
annoying technical questions.  Lastly, the authors would like to thank the
anonymous referee whose comments greatly improved the readability of this
paper. 

During the writing of this work, Campbell was supported by Vanderbilt University, Duke University, and is now employed at the Center for Communications Research, La Jolla. Zakharevich is supported by Cornell University,  NSF DMS-1654522 and NSF CAREER-1846767.

\section{CGW-Categories} \label{sec:cgw}

This section contains the main definition of the paper: the definition of a
CGW-category.  Because exact categories all embed into abelian categories, the
data of exact sequences is defined using universal properties in this abelian
category: if
\begin{diagram}
  {X & Y & Z\\};
  \to{1-1}{1-2}\to{1-2}{1-3}
\end{diagram}
is an exact sequence, $X$ is the kernel of $Y \rto Z$ and $Z$ is the cokernel of
$X \rto Y$.  However, if we instead discard this ``ambient'' abelian category,
and think of the exact sequences as simply data to be manipulated (as per
Quillen's original definition \cite{quillen}), a simple observation comes to
light: there is no intrinsic reason why admissible monomorphisms and admissible
epimorphisms must compose.  It is simply necessary that we encode their
relationships to one another.

An efficient way to encode this kind of structure is using the formalism of double categories. We thus begin by recalling the definition of a double category, as well as
establishing some notation for working with double categories. The notion of
double categories goes back to \cite{ehresmann63}.  We do not include the
complete definition; for the reader interested in a more in-depth introduction,
see for example \cite[Section II.6]{leinster}.


\begin{definition}
  A \emph{double category} $\C$ is an internal category in $\Cat$.  More
  concretely, a double category consists of a pair of categories, denoted
  $\E_\C$ and $\M_\C$, which have the same objects.  We denote morphisms in
  $\M_\C$ by $\mrto$ and morphisms in $\E_\C$ by $\erto$.  This pair is endowed
  with a collection of squares, called \emph{distinguished squares}.  These are
  denoted
  \[\dsq ABCD{f'}{g'}{g}{f}.\]
  In each distinguished square, $f,f'\in \M_\C$ and $g,g'\in \E_\C$.  The squares
  satisfy compositional axioms, which say in effect that gluing two squares
  horizontally or vertically gives another distinguished square.  In addition,
  if $f$ and $f'$ are both isomorphisms then for any $g,g'$ either both of the
  following squares exist, or neither does: 
  \[\dsq{A}{B}{C}{D}{f}{g}{g'}{f'} \qquad \dsq{B}{A}{D}{C}{f^{-1}}{g'}{g}{{f'}^{-1}}\]


  We  sometimes write $\C = (\E_\C,\M_\C)$.  When $\C$ is clear from context
  we  omit the subscripts from the notation.
\end{definition}

\begin{example} \label{ex:ambient}
  Let $\A$ be any category, and $\E$ and $\M$ two subcategories containing all
  isomorphisms in $\A$.  We can define
  a double category structure $(\E,\M)$ by letting the objects be the objects of
  $\C$, the horizontal morphisms be given by $\M$ and the vertical morphisms by
  $\E$.  We let distinguished squares be any subset of the commutative squares
  in $\A$ which satisfies appropriate closure conditions.
\end{example}

In most cases of interest, the double categories we work with arise as in Example~\ref{ex:ambient}, so it is useful to introduce language for these categories. 

\begin{definition}
  If a double category $(\E,\M)$ arises from a situation as in
  Example~\ref{ex:ambient}, we say that $\A$ is an \emph{ambient} category for
  $(\E,\M)$.  In such a case the identity functor gives a natural isomorphism of
  categories $\iso \E \rto \iso \M$.
\end{definition}

CGW-categories will be double categories equipped with extra data. Most of the data involves the specification of the existence of certain distinguished squares. We define certain categories that come up repeatedly in these specifications. 

\begin{definition}
  Let $\C = (\E, \M)$ be a double category.  We write $\Ar_\square\E$ for the
  category whose objects are morphisms $A \erto B$ in $\E$, and where 
  \[\Hom_{\Ar_\square\E}(\makeshort{A \erto^{g} B, A'\erto^{g'} B'}) =
  \left\{\begin{tabular}{c}
    \hbox{distinguished} \\  \hbox{squares}
    \end{tabular}
    \hbox{$\dsq{A}{A'}{B}{B'}{}{g}{g'}{}$}\right\}\] We have an analogous
category $\Ar_\square\M$.  Every $2$-cell in $\C$ appears uniquely as a morphism
in $\Ar_\square \E$ and $\Ar_\square\M$.

  Now let $\D$ be any ordinary category.  We write $\Ar_\triangle\D$ for the
  category whose objects are morphisms $A \rto B$ in $\D$, and where 
  \[\Hom_{\Ar_\triangle\D}(\makeshort{A \rto^f B, A' \rto^{f'} B'})  =
    \left\{
      \begin{tabular}{c}
        commutative \\ squares
      \end{tabular}
      \begin{inline-diagram}
        {A & A' \\ B & B'\\};
        \arrowsquare{{\cong}}{f}{f'}{}
      \end{inline-diagram}
      \right\}.\]
\end{definition}

We now come to the definition of a CGW-category. 

\begin{definition}
  A \emph{CGW-category} $(\C,\phi,c,k)$ is a double category $\C = (\E,\M)$, an
  isomorphism of categories $\phi:\iso \M \rto \iso \E$ which is the identity on
  objects, and equivalences of categories
  \[k:\Ar_\square \E \rto \Ar_\triangle \M \qqand c:\Ar_\square \M \rto
  \Ar_\triangle \E\] which satisfy:
  \begin{itemize}
  \item[(Z)] $\C$ contains an object $\initial$ which is initial in both $\E$
    and $\M$.
  \item[(I)] If $f: A \mrto B$ is any isomorphism in $\M$ then all four of the
    following squares are distinguished:
    \[\dsq ABBBf{\phi(f)}{1_B}{1_B} \qquad \dsq AAAB{1_A}{1_A}{\phi(f)}f \qquad 
      \dsq ABAAf{1_A}{\phi(f^{-1})}{1_A} \qquad \dsq AABA{1_A}{\phi(f)}{1_A}{f^{-1}}.\]
  \item[(M)] Every morphism in the categories $\E$ and $\M$ is monic.
  \item[(K)] For every $g:A \erto B$ in $\E$, the codomains of $g$ and $k(g)$
    are equal.  We write $k(g)$ as $A^{k/g} \mrto^{g^k} B$.  There exists a
    (unique up to unique isomorphism) distinguished square
    \[\dsq{\initial}{A}{A^{k/g}}{B}{}{}{g}{g^k}.\]
    Dually, for every $f:A \mrto B$ in $\M$ the codomains of $f$ and $c(f)$ are
    equal; we write  $c(\makeshort{A \mrto^f B}) =
    A^{c/f} \erto^{f^c} B$. There exists a (unique up to unique isomorphism)
    distinguished square
    \[\dsq{\initial}{A^{c/f}}{A}{B}{}{}{f^c}{f}.\]
  \item[(A)] For any objects $A$ and $B$ there exist distinguished squares
    \[\dsq{\initial}{A}{B}{X}{}{}{}{} \qqand
      \dsq{\initial}{B}{A}{X}{}{}{}{}.\] 
  \end{itemize}
  As isomorphisms can be considered to be ``both e-morphisms and m-morphisms''
  we will generally draw them as plain arrows.

  When it is clear from context, we write $A^{k/B}$ or $A^k$ instead of
  $A^{k/g}$ (and analogously for $c$).  When $\phi$, $c$ and $k$ are clear from
  context we omit them from the notation.  When $\C$ has an ambient category
  $\A$ and $\phi$ is the identity functor, we omit $\phi$ from the notation.
\end{definition}

The definition of a CGW-category is symmetric with respect to m-morphisms and
e-morphisms.  This duality is highly versatile and allows us to get symmetric
results about e-morphisms and m-morphisms with no extra work. 

\begin{remark}
  Axiom (A) is used only to show that $K_0(\C)$ is an abelian group.  Thus if in
  some case such a property is not necessary this axiom can be dropped and the
  rest of the analysis will still hold.
\end{remark}

Functors of CGW-categories must preserve all structure in sight. 

\begin{definition}
  A \emph{CGW-functor} of CGW-categories is a double functor $F:(\E,\M) \rto (\E',\M')$ which
  commutes with $c$ and $k$.  More concretely, $F$ is a CGW-functor if the
  following two diagrams commute:
  \begin{diagram}
    { \Ar_\square \E & \Ar_\triangle \M  & \qquad\qquad & \Ar_\square \M &
      \Ar_\triangle \E \\
      \Ar_\square \E' & \Ar_\triangle \M' && \Ar_\square \M' & \Ar_\triangle
      \E'\\};
    \to{1-1}{1-2}^k \to{1-4}{1-5}^c 
    \to{2-1}{2-2}^{k'} \to{2-4}{2-5}^{c'}
    \to{1-1}{2-1}_{\Ar_\square F} \to{1-2}{2-2}^{\Ar_\triangle F} 
    \to{1-4}{2-4}_{\Ar_\square F} \to{1-5}{2-5}^{\Ar_\triangle F}
  \end{diagram}
\end{definition}

The fact that $c$ and $k$ take distinguished squares to commutative triangles
means that distinguished squares are \emph{equifibered} (the vertical arrows
have equal ``kernels'' given by $k$) and \emph{equicofibered} (the horizontal
arrows have equal ``cokernels'' given by $c$).  By Axiom (K), $c$ and $k$ are
mutual inverses on objects.

We now prove some technical consequences of the axioms.

\begin{lemma}
  For any $A$, the morphism $f: \initial \mrto A$ has $f^c = 1_A$.
  Dually, the morphism $f:\initial \erto A$ has $f^k = 1_A$. 
\end{lemma}

The following lemma is the most important of the technical results.  It states
that e-morphisms and m-morphisms can be \emph{commuted past} one another using
distinguished squares.  This is what will allow the $Q$-construction in
Section~\ref{sect:q_construction} to work.

\begin{lemma} \label{lem:axiomC}
  For any diagram $A \mrto^f B \erto^g C$ there is a unique (up to unique
  isomorphism) distinguished square
  \[\dsq{A}{B}{D}{C}{f}{}{g}{}.\]
  The analogous statement holds for any diagram $A \erto^f B \mrto^g C$.
\end{lemma}

\begin{proof}
  As the categories $\M$ and $\E$ are symmetric in the definition of a
  CGW-category it suffices to check the first part.  Given a diagram as in the
  statement of the lemma, we can apply $c$ to the first morphism to obtain a
  diagram
  \[A^{c/f} \erto^{f^c} B \erto^g C.\]
  This diagram represents a morphism $(A^{c/f} \erto^{f^c} B) \erto^g (A^{c/f}
  \erto^{gf^c} C)$ in $\Ar_\triangle\E$.  Applying $c^{-1}$ to this morphism produces
  a distinguished square
  \[\dsq{A}{B}{(A^{c/f})^{k/gf^c}}{C}{f}{}{g}{},\]
  where we have used that $c$ and $k$ are inverses on objects.

  To check that this distinguished square is unique, suppose we are given any
  other such square
  \[\dsq{A}{B}{D}{C}{f}{}{g}{f'}.\]
  Applying $c$ to this square produces a morphism
  \[\makeshort{(A^{c/f} \erto^{f^c} B)} \makeshort[4em]{\erto^g}
    \makeshort{(D^{c/f'} \erto^{{f'}^c} C)} \in \Ar_\triangle \E.\]
  Since the square is distinguished, we must have $A^{c/f} \cong D^{c/f'}$; if
  we choose $D^{c/f'} = A^{c/f}$ the codomain of the above morphism becomes
  $A^{c/f} \erto^{gf^c} C$.  Thus any such distinguished square is mapped by $c$
  to the original diagram; since $c$ is an equivalence of categories, the square
  must be canonically isomorphic to the square produced above.
\end{proof}

\begin{lemma} \label{lem:opensquare}
  Given any composition 
  \[C \mrto B \mrto A\]
  there is an induced map $B^{c/A} \erto C^{c/A}$ such that the triangle
  \begin{diagram}
    { B^{c/A} && C^{c/A} \\
      & A \\};
    \eto{1-1}{1-3} \eto{1-1}{2-2} \eto{1-3}{2-2}
  \end{diagram}
  commutes.
\end{lemma}

\begin{proof}
  We begin by applying the equivalence of categories given by $k^{-1}$ from Axiom
  (K).  Since $k^{-1} = c$ on objects, we have the induced diagram
  \begin{diagram}
    {C^{c/B} & C^{c/A} \\
     B & A \\};
   \eto{1-2}{2-2} \eto{1-1}{2-1} \mto{1-1}{1-2}^h \mto{2-1}{2-2}
   \dist{1-1}{2-2}
  \end{diagram}
  We now apply the equivalence given by $c$ to produce the diagram
  \[B^{c/A} = (C^{c/B})^{c/h} \erto C^{c/A} \erto A.\]
\end{proof}

We conclude this section with a pair of definitions that will be useful
in later sections.

\begin{definition}
  Let $\C = (\E,\M,\phi,c,k)$ be a CGW-category.  A \emph{CGW-subcategory} is a
  sub-double category $\A \subseteq \C$ such that $(\A,\phi|_{\A},c|_\A,k|_\A)$
  is also a CGW-category.
\end{definition}

\begin{definition} \label{def:closure}
  We say that a CGW-subcategory $\A$ of a CGW-category $(\C,\phi,c,k)$
  is \emph{closed under subobjects} if for any morphism $B \mrto C\in \M$,
  if $C\in \A$ then $B\in \A$.  We say that $\A$ is \emph{closed under
    quotients} if for any morphism $B \erto C\in \E$, if $C\in \A$ then
  $B\in \A$.  We say that $\A$ is \emph{closed under extensions} if for every
  distinguished square
  \[\dsq ABCD{}{}{}{}\]
  if $A$, $B$ and $C$ are in $\A$ then so is $D$.
\end{definition}

\section{Examples} \label{sec:ex}

In this section we give several motivating examples of CGW-categories.

\begin{example} \label{ex:exact}
  Let $\A$ be an exact category.  Let $(\C,c,k)$ be given by
  \[\E = \{\mathrm{admissible\
      epimorphisms}\}^\op \qqand \M = \{\mathrm{admissible\ monomorphisms}\};\]
  Define $\phi$ to be the identity on objects and \emph{inversion} on morphisms.
  The distinguished squares are stable squares: those squares that are both
  pushouts and pullbacks in $\A$.  The equivalence $k$ is given by mapping every
  admissible epimorphism to its kernel; the equivalence $c$ is given by taking
  every admissible monomorphism to its cokernel.

  We check the axioms explicitly.
  \begin{itemize}
  \item[(Z)] The zero object is initial in $\M$ and terminal in $\E$, so it is
    initial in both $\M$ and $\E$.
  \item[(I)] This follows directly from the definition.
  \item[(M)] This holds by definition.
  \item[(K)] $k$ and $c$ give the correct equivalences, since distinguished
    squares are both equifibered (since they are pullbacks) and equicofibered
    (since they are pushouts).
  \item[(A)] This holds with $X = A\oplus B$.
  \end{itemize}
\end{example}

Thus an exact category gives rise to a CGW-category.
However, there are examples of CGW-categories which are not exact.  

\begin{example} \label{ex:finset*}
  Consider the category $\FinSet_*$ of based finite sets.  We define a
  CGW-category $(\C,c,k)$ by setting
  \[\M = \{\mathrm{injections}\}\qqand 
    \E = \left\{f:\makeshort{A \rto B}\,\Big|\,
      f|_{f^{-1}(B\backslash\{*\})}\hbox{ is a bijection} \right\}^\op.\] The
  distinguished squares are the pushout squares; these are also all
  pullback squares.  The equivalence $\phi$ is defined, as in the previous
  example, by taking inverses.  Define $k$ by taking $f:A \rfib B$ to
  $f^{-1}(*) \rcofib A$.  Define $c$ by taking $g:A \rcofib B$ to
  $B \rfib B/g(A)$, with the elements not in the image of $g$ mapping
  to themselves, and everything else mapping to the basepoint.

  That axioms (Z), (I), (M), and (A) are satisfied is direct from the
  definition.  
  The distinguished squares are pullback squares in the underlying category;
  therefore in a distinguished square the preimages of the basepoint of the two
  vertical maps are isomorphic.  This proves half of (K).  Dually, the
  complements of the two injections horizontally are also isomorphic, since $g$
  is injective away from the basepoint.
\end{example}

One of the advantages of CGW-categories is the observation that the
contravariance in the $\E$-direction is not necessary. All of the following
examples come equipped with an ambient category, so we omit mention of $\phi$.

\begin{example} \label{ex:finset}
  Consider the category $\FinSet$.  We define a CGW-category $(\C,c,k)$ by
  setting
  \[\E = \M = \{\mathrm{injections}\}.\]
  The distinguished squares are the pushout squares; since all morphisms are
  injections, they are also pullback squares.  The equivalences $c$ and $k$ are
  given by taking any injection $A \rcofib B$ to the inclusion
  $B \backslash A \rcofib B$.

  That axioms (Z), (I), (M), and (A) are satisfied is direct from the
  definition.
  Since distinguished squares are pushouts, the complements of the images in the
  horizontal maps are isomorphic; the same holds dually for the vertical maps.
  Thus (K) holds.
\end{example}

We can also improve the intuition from the finite sets example to get a
CGW-category structure on the category of varieties.

\begin{example} \label{ex:var}
  Let $\C = \Var$
  \[\E = \{\mathrm{open\ immersions}\} \qqand\M =
  \{\mathrm{closed\ immersions}\}.\] We let both $c$ and $k$ take a morphism to
  the inclusion of the complement.  The distinguished squares 
  \[\dsq ABCD{}{}{f}{g}\] are the pullback
  squares in which $\im f \cup \im g = D$. Axiom (Z) is satisfied by the empty
  variety. Axiom (I) holds by definition. Axiom (M) is verified by noting that
  open and closed immersions satisfy base change in the category of
  varieties. 
  Axiom (A) holds by setting $X = A \amalg B$.
  To see that Axiom (K) holds, consider a distinguished square 
  \[\dsq ABCD{}{}{}{}.\]
  By definition, $D \smallsetminus C \cong B \smallsetminus A$, since the image
  of $B$ in $D$ contains the complement of the image of $C$.  The dual statement
  for e-morphisms holds as well.  
\end{example}

The CGW-category of varieties includes into the larger category of reduced
schemes of finite type via a CGW-functor:

\begin{example}
  Let $\Sch_{rf}$ be the category of reduced schemes of finite type, with
  morphisms the compositions of open and closed immersions.  We define the
  $\E$-morphisms to be the open immersions and the $\M$-morphisms to be the
  closed immersions.  The distinguished squares are those squares
  \begin{diagram}
    { A & B \\ C & D \\};
    \dist{1-1}{2-2} \mto{1-1}{1-2} \mto{2-1}{2-2}^f \eto{1-1}{2-1} \eto{1-2}{2-2}^g
  \end{diagram}
  for which $D = \im f \cup \im g$ and which are pullbacks in the category of schemes.
\end{example}

We can also restrict attention just to \textsl{smooth} varieties.

\begin{example}
  The category $\Var_{/k}^{sm}$ of smooth varieties can be given a CGW-structure.
  We set the m-morphisms to be closed immersions with smooth complements, and
  the e-morphisms to be open immersions with smooth complements.  Thus
  $\Var_{/k}^{sm}$ is a sub-CGW-category (but not a full sub-CGW-category) of
  $\Var_{/k}$.  
\end{example}

\section{The $K$-theory of a CGW-category}\label{sect:q_construction}

We are now ready to define the $K$-theory of a CGW-category. The construction exactly follows Quillen's $Q$-construction \cite{quillen} for exact categories. After the introduction of the definition, the rest of the section is taken up by noting some useful technical results and providing the standard presentation for the group $K_0(\C)$. 

\begin{definition}
  For a CGW-category $(\C,\phi,c,k)$ we define 
  \[K(\C) = \Omega |Q\C|,\]
  where $Q\C$ is the category with 
  \begin{description}
  \item[objects] the objects of $\C$,
  \item[morphisms] morphisms $A \rto B$ are equivalence classes of diagrams 
    \[A \erto^f X \mrto^g B,\] where $f\in \E$ and $g\in \M$.  Two diagrams
    \[A \erto^f X \mrto^g B \qqand A \erto^{f'} X' \mrto^{g'} B\]
    are considered equivalent if there exists a diagram
    \begin{squisheddiagram}
      { & X \\
        A && B \\
        & X' \\};
      \mto{2-1}{1-2}^f \mto{2-1}{3-2}_{f'}
      \eto{1-2}{2-3}^g \eto{3-2}{2-3}_{g'}
      \to{1-2}{3-2}^{\cong}
    \end{squisheddiagram}
    where the left-hand triangle commutes in $\E$ and the right-hand triangle
    commutes in $\M$.  The functor $\phi$ is implicitly being used to place the
    vertical isomorphism in both $\E$ and $\M$ simultaneously.
  \item[composition] defined using Lemma~\ref{lem:axiomC}.  More concretely, given two
    equivalence classes of diagrams represented by
    \[A \erto^f X \mrto^g B \qqand B \erto^{f'} Y \mrto^{g'} C\]
    there exists a unique (up to unique isomorphism) distinguished square
    \[\dsq XBZY{g}{f''}{f'}{g''}.\]
    The composition of the two diagrams is defined to be the class of diagrams
    represented by
    \[A \erto^{f''f} Z \mrto^{g'g''} C.\]
  \end{description}
  The basepoint is taken to be $\initial$.
\end{definition}

\begin{remark}
  Although we have defined $K$-theory for CGW categories, the $K$-theory of a
  double category is defined for any double category satisfying Lemma~\ref{lem:axiomC}. 
\end{remark}

As with any definition of $K$-theory, the first step is to check that it gives
the desired group on $K_0$.

\begin{theorem}\label{thm:pi_0}
  $K_0(\C)$ is the free abelian group generated by objects of $\C$, modulo the
  relation that for any distinguished square \[\dsq{A}{B}{D}{C}{}{}{g}{f}\] we
  have $[D] + [B] = [A] + [C]$. 
\end{theorem}

\begin{proof}
  There are two ways to proceed. One could prove this by showing that $K(\C)$ is
  equivalent to some variant of the $S_\dotp$ construction, and proceeding
  from there, or one could mimic Quillen's original proof that $\pi_1 (BQ\C) =
  K_0 (\C)$ for exact categories. We opt for the latter, again to emphasize the
  analogy with exact categories.

  We follow a more modern version of the proof (see, e.g. \cite[Proposition
  IV.6.2]{weibel_kbook}).

  The morphisms $\initial \erto^= \initial \mrto A$ form a maximal tree in $BQ\C$.  By
  \cite[Lemma IV.3.4]{weibel_kbook}, the fundamental group $\pi_1 (B Q \C)$ is
  generated by the morphisms of $BQ\C$, modulo the relations
  $[\initial \erto^= \initial\mrto A] = 1$ and $[f] \cdot [g] = [f \circ g]$ for
  composable morphisms in $Q \C$. We proceed by a series of reductions to get
  the set of generators and relations in the theorem. In what follows we let
  $[A \erto X \mrto B]$ denote the equivalence class of a morphism $A \rto B$ in
  $\pi_1 (B Q \C)$. The notation $[A \mrto B]$ corresponds to the morphism
  $[A \erto^= A \mrto B ]$ and similarly $[A \erto B]$ corresponds to
  $[A \erto B \mrto^= B]$.  

  From the definition of $Q$ we have $[B \mrto C][A \mrto B] = [A \mrto C]$.  In
  particular, since $[\initial \mrto X]= 1$ in $\pi_1(BQ\C)$ for all objects $X$,  $[A
  \mrto B] = 1$ for all m-morphisms.

  We begin by noting that by definition
  \[[A \mrto B][D \erto A] = [D \erto A \mrto B].\]
  Now consider $[B \erto C][A \mrto B]$.  By Lemmma~\ref{lem:axiomC} there exists a
  distinguished square
  \[\dsq{A}{B}{D}{C}{}{}{}{}\]
  which implies the relation 
  \[[B \erto C][A\mrto B] = [D \mrto C][A \erto D]\]
  via the composition relation.
  Each distinguished square produces such a relation.  Since all
  morphisms in $\M$ are equal to the identity, this reduces to the equation
  \[[B \erto C] = [A \erto D]\]
  for all distinguished squares.  We have now shown that $\pi_1(BQ\C)$ 
  has as generators the morphisms of $\E$, with relations induced by composition
  and distinguished squares.

  Since
  \begin{equation}\label{main_relation}
  [\initial \erto A_1][A_1 \erto A_2] = [\initial \erto A_2]. 
  \end{equation}
  $\pi_1(BQ\C)$ is generated by the elements $[\initial\erto A]$, which we
  abbreviate to $[A]$.  This expression also eliminates the composition
  relation.  We can substitute for both sides in the relations induced by the
  distinguished squares to get
  \[[B]^{-1}[C] = [A]^{-1}[D].\] 
  This gives the desired presentation of $K_0(\C)$.

  It remains to check that $K_0(\C)$ is abelian; in other words, that $[A][B] =
  [B][A]$.  The relations imposed by the squares in Axiom~(A) state that
  \[[A][B] = [X] = [B][A],\]
  as desired.
\end{proof}

The rest of this section is devoted to some technical lemmas exploring the
properties of this $Q$-construction.  The first identifies the isomorphisms in
$Q\C$ via their components.

\begin{lemma} \label{lem:iso}
  If $\alpha:A \rto B$ is an isomorphism inside $Q\C$ for a CGW-category $\C$
  represented by
  \[A \erto^f X \mrto^g B\]
  then both $f$ and $g$ are isomorphisms in $\C$.
\end{lemma}

\begin{proof}
  Suppose that the inverse of $\alpha$ is represented by
  \[B \erto^{f'} Y \mrto^{g'} A.\]
  Then the composition is represented by a diagram
  \begin{diagram}
    { A & X & B \\
      & Z & Y \\ & & A \\};
    \eto{1-1}{1-2}^f \mto{1-2}{1-3}^g
    \eto{1-2}{2-2}_{f''} \eto{1-3}{2-3}^{f'}
    \mto{2-2}{2-3}^{g''} \mto{2-3}{3-3}^{g'}
    \dist{1-2}{2-3}
  \end{diagram}
  Since this is equivalent to $1_A$, $f''f$ is an isomorphism.  Since $f''$ is
  monic and $f$ is its right inverse, it must be an isomorphism; thus $f$ is an
  isomorphism.  Doing the composition in reverse, we see that $g$ has a right
  inverse and thus must also be an isomorphism. 
\end{proof}

The next lemma illustrates that we can think of a morphism in $Q\C$ as a set of
``layers'' inside $\M$.  This allows us to think about the $Q$-construction in
CGW-categories analogously to the way that Quillen originally thought about
exact categories in \cite{quillen}.

\begin{lemma} \label{lem:preorder}
  For any CGW-category $\B$ and any $B\in \B$, the category $Q\B_{/B}$ is
  equivalent to the category $L_B\B$ with
  \begin{description}
  \item[objects] diagrams $B_1 \mrto B_2 \mrto B$ in $\B$,
  \item[morphisms] commutative diagrams
    \begin{diagram}
      { B_1 & B_2 & B \\
        B_1' & B_2' & B\\};
      \mto{1-1}{1-2} \mto{1-2}{1-3}
      \mto{2-1}{2-2} \mto{2-2}{2-3} 
      \mto{2-1}{1-1} \mto{1-2}{2-2} \eq{1-3}{2-3}
    \end{diagram}
  \end{description}
  In particular, $Q\B_{/B}$ is a preorder for any $B$.
\end{lemma}

\begin{proof}
  It suffices to prove the first part of the lemma; the second follows from the
  definition of $L_B\B$ and axiom (M).

  We define a functor $\kappa: Q\B_{/B} \rto L_B\B$.  
  An object of $Q\B_{/B}$ is a diagram $B_1 \erto^{g} B_2 \mrto B$.  We
  send this to the diagram $B_1^k \mrto^{g^k} B_2 \mrto B$.  Seeing that this
  extends to a functor is a bit more complicated.
  Suppose that 
  \[B_1 \erto^{g} B_2 \mrto^{f} B \qqand B_1' \erto^{g'} B_2'
  \mrto^{f'} B\] are two objects of $Q\B_{/B}$, and suppose that we are given
  a morphism between them.  This morphism consists of an object $C\in \B$ and
  a diagram
  \begin{diagram}
    { B_1 & C & B_1'\\
      & B_2 & B_2' \\ 
      & & B \\};
    \eto{1-1}{2-2}_{g} \eto{1-1}{1-2}^{h'}
    \eto{1-2}{2-2}^{h} \mto{2-2}{2-3} \mto{2-2}{3-3}_{f}
    \mto{1-2}{1-3} \eto{1-3}{2-3}^{g'}
    \mto{2-3}{3-3}^{f'} 
    \dist{1-2}{2-3}
  \end{diagram}

  Applying $c^{-1}$ to the upper-left triangle, this diagram corresponds to a
  unique diagram
  \begin{diagram}
    { B_1^{k/C} & C & B_1'\\
      B_1^{k/B_2} & B_2 & B_2' \\ 
      & & B \\};
    \eto{1-1}{2-1}_{h''} \mto{1-1}{1-2}
    \eto{1-2}{2-2}^{h} \mto{2-2}{2-3} \mto{2-2}{3-3}_{f}
    \mto{1-2}{1-3} \eto{1-3}{2-3}^{g'}
    \mto{2-3}{3-3}^{f'}
    \mto{2-1}{2-2}
    \dist{1-1}{2-2} \dist{1-2}{2-3}
  \end{diagram}
  Applying $k$, this time to the two distinguished squares on the top, gives us a
  unique diagram
  \begin{diagram}
    {  {B_1'}^{k/B_2'} = (B_1^{k/B_2})^{k/h''} & {B_1}^{k/B_2} & B_2 & B_2' \\ 
      & & & B \\};
    \mto{1-1}{1-2} \mto{1-2}{1-3} \mto{1-3}{1-4}
    \mto{1-3}{2-4}_{f} \mto{1-4}{2-4}^{f'} 
  \end{diagram}
  This can be rearranged into a diagram
  \begin{diagram}
    { B_1^{k/B_2} & B_2 & B \\
      {B_1'}^{k/B_2'} & B_2' & B \\};
    \mto{1-1}{1-2} \mto{1-2}{1-3}^{f}
    \mto{2-1}{2-2} \mto{2-2}{2-3}^{f'}
    \eq{1-3}{2-3} \mto{2-1}{1-1} \mto{1-2}{2-2}
  \end{diagram}
  as desired.

  The inverse equivalence is given by sending a diagram $B_1 \mrto B_2 \mrto
  B$ to $B_1^{c/B_2} \erto B_2 \mrto B$. By Axiom~(K) these two functors give
  inverse equivalences.
\end{proof}

We now give several examples of $K$-theories of CGW-categories.

\begin{example}
  We consider the examples from Section~\ref{sec:ex}.

  \begin{description}
  \item[Example~\ref{ex:exact}] When $(\C,c,k)$ arises from an exact category
    $\A$, $BQ\C = BQ\A$, so $K(\C) = K(\A)$.
  \item[Example~\ref{ex:finset*}] The simplicial set $BQ\C$ is an edgewise
    subdivision of the $S_\dotp$-construction for the Waldhausen category
    $\FinSet_*$ with injections as the cofibrations (for a more in-depth
    discussion, see Theorem~\ref{thm:S.=Q}).  Thus
    \[
      K(\C):= \Omega B Q \FinSet_* \simeq K^{\mathrm{Wald}}(\FinSet_+) \simeq
      \Omega^\infty\Sigma^\infty S^0
    \]
    where the last equivalence is by Barrat-Priddy-Quillen
    \cite{barratt_priddy}.
  \item[Example~\ref{ex:finset}] In this case we also have $K(\C) \simeq
    \Omega^\infty \Sigma^\infty S^0$. Indeed, there is an equivalence of CGW-categories between
    $(\FinSet,c,k)$ and $(\FinSet_*,c,k)$ from Example~\ref{ex:finset*} given as
    follows. An injection $[i] \rcofib [j]$ considered as an element of
    $\E \subset \FinSet$ corresponds to an injection $[i]_+ \rcofib [j]_+$ in
    $\FinSet_*$. An injection $u: [i] \rcofib [j]$ considered as an element of
    $\M \subset \FinSet$ corresponds to a surjection $[j]_+ \rfib [i]_+$ by
    taking $m \in [j]$ to $u^{-1}(m)$ and the rest of $[j]$ to the distinguished
    basepoint.
  \item[Example~\ref{ex:var}] $K(\Var)$ is equivalent to the $K$-theory of
    varieties defined in \cite{campbell}; for a more detailed discussion, see
    Section~\ref{subdivision}.
  \end{description}
\end{example}

\section{ACGW-Categories} \label{sec:acgw}

A CGW-category behaves like an exact category.  In order to create categories
that are analogous to abelian categories (with the goal of proving Quillen's
d\'evissage and localization) we need to assume some extra conditions.  The
extra conditions amount to the requirement that certain ``pushout-like'' objects
exist and are compatible with $c$ and $k$; in geometric settings this
corresponds to certain gluings of objects.

\begin{definition}
  An \emph{enhanced} double category is a double category $\C$ with \emph{two}
  notions of $2$-cell, called the \emph{distinguished} and
  \emph{pseudo-commutative} squares.  These are required to satisfy the property
  that forgetting either of the sets of squares produces a double category, and
  all distinguished squares are pseudo-commutative.  We denote distinguished
  squares with $\square$ and pseudo-commutative squares with $\circlearrowleft$.

  We write $\Ar_\circlearrowleft \M$ for the category whose objects are
  morphisms in $\M$ and whose morphisms are pseudo-commutative squares in $\C$.
  We write $\Ar_\times\M$ for the category whose objects are morphisms in $\M$
  and whose morphisms are pullback squares in $\M$.  The category $\Ar_\square\M$
  is a subcategory of $\Ar_\circlearrowleft \M$ and $\Ar_\triangle \M$ is a
  subcategory of $\Ar_\times \M$ (since all morphisms in $\M$ are monic).
\end{definition}

\begin{remark}
  The term ``pseudo-commutative'' is inspired by the role that commutative
  squares play in the case when we are discussing abelian categories.  Consider
  an abelian category $\A$, and the associated CGW-category $\C$.  The
  distinguished squares in $\C$ are the stable squares.  However, the
  commutative squares in $\A$ also play a role in the following sense.  In an
  abelian category, every morphism $f:A \rto B$ can be factored as
  $A \rfib \im f \rcofib B$, an epic followed by a monic.  This means that in
  $\C$, any diagram of the form
  \[X \erto Z \mlto Y,\] which represents a monic followed by an epic, can be
  completed to a square in an essentially unique way.  This square will
  \emph{not} necessarily be distinguished, but it is still important.  This
  completion is the ``mixed pullback'' that we define in the next definition.
\end{remark}

Before we define a pre-ACGW-category we need one extra helper-definition; this
is necessary because, although monomorphisms always behave well with respect to
pullbacks, they do not always behave well with respect to pushouts.

\begin{definition}
  Let $\C$ be a category in which all morphisms are monic, and let
  \[C \lto A \rto B\] be a diagram in $\C$.  The \emph{restricted pushout} of
  this diagram is the initial object (if it exists) in the category of
  commutative squares
  \begin{diagram}
    { A & B \\
      C & X \\};
    \arrowsquare{}{}{}{}
  \end{diagram}
  which are also pullback squares; in other words, it is cones $X$ under the
  diagram such that $A \cong B \times_X C$.  As usual, a morphism between
  diagrams is a natural transformation in which all components are equal to the
  identity except at $X$.  We denote restricted pushouts by $B \star_A C$.
\end{definition}

The important intuition behind this definition lies in the following example:
\begin{example}
  Let $\C$ be the category of sets and injections.  Then $\C$ does not contain
  all pushouts, as for example the diagram
  \[A \lto \emptyset \rto A\]
  does not have a pushout for any nonempty set $A$; this is because the 
  map $A \amalg A \rto A$ is not a monomorphism.  However, the restricted
  pushout of this diagram exists and, as expected, will be isomorphic to $A
  \amalg A$.  
\end{example}

We are now ready to define pre-ACGW-categories:

\begin{definition}\label{defn:pre_acgw}
  A \emph{pre-ACGW-category $(\C,\phi,c,k)$} is an enhanced double category $\C$
  which is a CGW-category when the pseudo-commutative squares are forgotten, and in
  which the following extra axioms are satisfied:
  \begin{itemize}
  \item[(P)] $\M$ and $\E$ are closed under pullbacks.
  \item[(U)] The functors $c$ and $k$ extend to equivalences of categories
    \[c:\Ar_\circlearrowleft \M \rto \Ar_\times\E \qqand 
    k: \Ar_\circlearrowleft \E \rto \Ar_\times \M.\]
    These are compatible in the sense that for any diagram $A \mrto C \elto B$
    there exists a unique isomorphism     \[\varphi: (A\times_C B^k)^{c/A} \rto
    (A^c\times_C B)^{k/B}\]
    such that the square
    \begin{diagram}
      { (A\times_C B^k)^{c/A} & (A^c\times_C B)^{k/B} & B \\
        A && C \\};
      \to{1-1}{1-2}^\varphi \mto{1-2}{1-3} \eto{1-1}{2-1} \mto{2-1}{2-3}
      \eto{1-3}{2-3} 
      \comm{1-1}{2-3}
    \end{diagram}
    is a pseudo-commutative square.

    We write $A\oslash_CB \defeq (A^c\times_C B)^{k/B} \cong (A\times_C
    B^k)^{c/A}$, so that we have a ``mixed pullback square''
    \begin{diagram}
      { A\oslash_C B & B \\ A & C \\};
      \mto{1-1}{1-2} \mto{2-1}{2-2} \eto{1-1}{2-1} \eto{1-2}{2-2}
      \comm{1-1}{2-2}
    \end{diagram}
  \item[(S)]   Suppose that we are given a pullback square
    \begin{diagram}
      { A \times_C B & A \\ B & C \\};
      \mto{1-1}{1-2} \mto{1-2}{2-2} \mto{1-1}{2-1} \mto{2-1}{2-2}
      \diagArrow{densely dotted,mmor}{1-1}{2-2}!\ell
    \end{diagram}
    in $\M$.  Then $X \defeq A \star_{A\times_C B} B$ exists.  The induced
    commuting square
    \begin{diagram}
      {X^{c/C}  & B^{c/C} \\ A^{c/C} & (A\times_CB)^{c/C} \\};
      \eto{1-1}{1-2} \eto{1-1}{2-1} \eto{2-1}{2-2} \eto{1-2}{2-2}
    \end{diagram}
    (constructed using Lemma~\ref{lem:opensquare}) is a restricted pushout.

    The dual of this statement also holds.
  \end{itemize}

  Given a pre-ACGW-category $(\C,\phi,c,k)$, a \emph{pre-ACGW-subcategory} $\D$
  is a sub-double category $\D$ of $\C$ (under \emph{both} double category
  structures in $\C$) such that $(\D,\phi|_\D,c|_\D,k|_\D)$ is also a
  pre-ACGW-category.  We say that $\D$ is \emph{full} if the vertical
  (resp. horizontal) category of $\D$ is a full subcategory of the vertical
  (resp. horizontal) category of $\C$.
\end{definition}

\begin{definition}
  An \emph{ACGW-category} is a pre-ACGW-category $(C,\phi,c,k)$ such that
  the following condition holds: 
  \begin{itemize}
  \item[(PP)] Restricted pushouts exist in $\M$.  These are compatible with
    cokernels, in the sense that a restricted pushout square
    \begin{diagram}
      { A & B \\ C & B \star_A C \\};
      \mto{1-1}{1-2}^f \mto{1-2}{2-2}^{g'} \mto{1-1}{2-1}_g \mto{2-1}{2-2}^{f'}
    \end{diagram}
    induces an isomorphism $A^{c/B} \rto^{\cong} C^{c/(B\star_A C)}$.  In
    addition, restricted pushouts are compatible with distinguished squares in
    the sense that given a diagram
    \begin{diagram}
      { C & A & B \\
        C' & A' & B' \\};
      \mto{1-2}{1-3} \mto{1-2}{1-1} \mto{2-2}{2-3} \mto{2-2}{2-1}
      \eto{1-1}{2-1} \eto{1-2}{2-2} \eto{1-3}{2-3}
      \dist{1-1}{2-2} \dist{1-2}{2-3}
    \end{diagram}
    there is an induced map $B\star_A C \rto B'\star_{A'} C'$ such that the two
    induced squares are distinguished.  These maps are compatible with
    compositions of distinguished squares.

    The dual statement for e-morphisms holds as well.
  \end{itemize}
\end{definition}

The definition of $\star$ implies that it behaves functorially like a pushout,
in the sense that given a diagram
\[C \mlto A \mrto^{f_1} B \mrto^{f_2} B'\] it follows that $(f_2f_1)' = f_2'f_1'$.

\begin{example}
  Let $\A$ be an abelian category.  Then $\A$ defines an ACGW-category for which
  $\M$ is the category of monomorphisms, $\E$ is the opposite category of the
  epimorphisms, distinguished squares are stable squares and pseudo-commutative squares
  are commutative squares.  Here, the ``mixed pullback'' of a diagram 
  \[A \mrto B \elto C\]
  is the factorizarion of the morphism $A \rto C$ into an epic followed by a
  monic.

  Axiom (S) translates to the following observation.  Assuming that we are
  working in $\Mod_R$, let $C$ be an $R$-module, and $A$ and $B$ be submodules
  of $C$.  Then $A\times_C B$ is $A \cap B$.  Then $X = A+B$, and the square
  \begin{diagram}
    { C/(A\cap B) & C/B \\ C/A & C/X \\};
    \fib{1-1}{1-2} \fib{1-2}{2-2} \fib{1-1}{2-1} \fib{2-1}{2-2}
  \end{diagram}
  is a pullback square.

  (PP) corresponds to the fact that an abelian category has all pushouts of
  monics, and such pushouts preserve cokernels.
\end{example}

\begin{example} 
  The category $\Var$ is a pre-ACGW-category.  Here we define the pseudo-commutative
  squares to be the pullback squares.  

  We check the axioms in turn.  Axiom (P) holds because varieties are closed
  under pullbacks.  In order to check Axiom (U) it suffices to check that given
  a variety $X$ and an open subvariety $U$ and a closed subvariety $Z$, we have
  \[Z \bs (Z \cap (X \bs U)) \cong U \cap ((X \bs Z)\cap U).\] This is true
  because it is true in the underlying topological spaces, where each one is
  simply $Z \times_X U$. Axiom (S) holds because it holds in the underlying
  topological spaces.  
\end{example}

\begin{non-example} 
  The CGW-category $\Var_{/k}^{sm}$ is \textsl{not} a pre-ACGW-category, since it
  is possible that the intersection of smooth subvarieties is not smooth.  This
  means that the m-morphisms are not closed under pullbacks.  
\end{non-example}

\begin{example} \label{ex:sch} The category $\Sch_{rf}$ is an ACGW-category,
  with the pseudo-commutative squares being pullback squares.  With this definition we
  can consider $\Var$ a pre-ACGW-subcategory of the ACGW-category $\Sch_{rf}$.
  That Axioms (P), (U), and (S) hold follows identically as for the case of
  varieties.

  Thus it remains to check Axiom~(PP), in particular that $\star$-products
  exist.  The pushout of schemes along open immersions produces a square of open
  immersions by the definition of a scheme; the pushout of schemes along closed
  immersion produces a square of closed immersions of schemes by \cite[Corollary
  3.9]{schwede}.  These are \emph{not} pushouts in the categories of closed/open
  immersions; these are pushouts in the entire category of schemes.  That this
  satisfies the conditions of (PP) follows from the universal property of
  pushouts.
\end{example}

We now consider an example that will be used in Section~\ref{sec:comparison}.

\begin{example} \label{ex:wr}
  Let $G$ be a discrete group, and consider the category $\FinSet \wr G$, with
  \begin{description}
  \item[objects] finite sets, and
  \item[morphisms] $S \rto T$ is a pair of functions $(f:S \rto T, f': S \rto G)$.
  \end{description}
  A composition of morphisms $(f,f'):S \rto T$ and $(g,g'): T \rto U$ is given
  by the pair consisting of $g\circ f$ and the composition
  \[S \rto \Gamma_f \subseteq S \times T \cong T \times S \rto^{g'\times f'} G \times G \rto^\mu
    G,\]
  where $\mu$ is the composition in $G$.

  More informally, we think of a morphism $S \rto T$ in $\FinSet \wr G$ as a map
  of finite sets $S \rto T$ together with a decoration by elements of $G$ on
  each element of $S$.  When we compose two such morphisms, we decorate each
  element by the multiplication of the two elements that it was decorated with
  in the composition: the decoration of the original element in the first
  morphism, and the decoration of its image in the second morphism.  The swap in
  the definition is necessary because composition of morphisms acts on the left,
  rather than the right.

  This can be demonstrated with the following picture:
  \begin{center}
    \begin{tikzpicture}
      \node (m-1-1) at (150:0.4) {$1$};
      \node (m-1-2) at (30:0.4) {$2$};
      \node (m-1-3) at (-90:0.4) {$3$};
      \draw (0,0) circle (0.8);
      \node (m-1-100) at (-90:1.8) {$A$};

      \begin{scope}[xshift=12em]
        \node (m-2-1) at (150:0.4) {$1$};
        \node (m-2-2) at (30:0.4) {$2$};
        \node (m-2-3) at (-90:0.4) {$3$};
        \draw (0,0) circle (0.8);
        \node (m-2-100) at (-90:1.8) {$B$};
      \end{scope}

      \begin{scope}[xshift=24em]
        \node (m-3-1) at (-0.375,0) {$1$};
        \node (m-3-2) at (0.375,0) {$2$};
        \draw (0,0) circle (0.65);
        \node (m-3-100) at (-90:1.8) {$C$};
      \end{scope}

      \diagArrow{->}{1-100}{2-100}^f \diagArrow{->}{2-100}{3-100}^g

      \diagArrow{densely dashed, ->, bend left}{1-1}{2-1}!{g_1}
      \diagArrow{densely dashed, ->}{1-2}{2-1}!{g_2}
      \diagArrow{densely dashed, ->, out=-10, in=-170}{1-3}{2-2}!{g_3}

      \diagArrow{densely dashed, ->, bend left}{2-1}{3-1}!{h_1}
      \diagArrow{densely dashed, ->}{2-2}{3-2}!{h_2}
      \diagArrow{densely dashed, ->, bend right}{2-3}{3-2}!{h_3}

      \diagArrow{densely dotted, ->, out=40,in=140}{1-1}{3-1}!{h_1g_1}
      \diagArrow{densely dotted, ->, bend left}{1-2}{3-1}!{h_1g_2}
      \diagArrow{densely dotted, ->, out=-20, in=-160}{1-3}{3-2}!{h_2g_3}
    \end{tikzpicture}
  \end{center}
  Here, $A,B,C\in \FinSet\wr G$ are sets, illustrated by the elements in the
  circles.  The dashed lines above morphisms $f$ and $g$ illustrate where
  elements map under $f$ and $g$, together with decorations.  The labeled dotted
  lines are the data of $gf$.
  
  When $G$ is trivial $\FinSet \wr G \cong \FinSet$.  By forgetting the
  decoration we get a functor $\FinSet \wr G \rto \FinSet$.

  We define an ACGW-structure on $\FinSet\wr G$ by declaring the e-morphisms and
  m-morphisms to both be all maps which are injective on the underlying sets,
  and declare a square to be distinguished if it commutes in the ambient
  category and if it is distinguished when mapped down to $\FinSet$.  This makes
  $\FinSet\wr G$ an ACGW-category.
\end{example}

We finish this section with a couple of technical lemmas which will be useful
later.

\begin{lemma} \label{lem:(B)}
  Let $\C$ be a pre-ACGW category. Given a diagram 
  \begin{diagram-numbered}{eq:pullingback}
    { C & B & A \\
      C' & B' & A' \\};
    \eto{1-1}{2-1} \eto{1-2}{2-2} \eto{1-3}{2-3}
    \mto{1-2}{1-3} \mto{2-2}{2-3} \eto{1-1}{1-2} \eto{2-1}{2-2}
    \comm{1-2}{2-3}
  \end{diagram-numbered}
  where $C \cong C'\times_{B'} B$ there exists a cube
  \begin{diagram}
    { C&& D\\
      & B&& A\\
      C' && D' \\
      & B' && A' \\};
    \mto{1-1}{1-3}
    \mto{3-1}{3-3} \mto{4-2}{4-4}
    \eto{1-1}{3-1} \eto{1-3}{3-3}  \eto{2-4}{4-4}
    \over{2-2}{2-4}  \mto{2-2}{2-4}
    \over{2-2}{4-2} \eto{2-2}{4-2}
    \eto{1-1}{2-2} \eto{1-3}{2-4} \eto{3-1}{4-2} \eto{3-3}{4-4}
  \end{diagram}
  where the top and bottom squares are distinguished, the left and right squares
  are pullbacks, and the front and back face are pseudo-commutative.

  The statement with the roles of e-morphisms and m-morphisms swapped also holds.
\end{lemma}

\begin{proof}
  Let
  \[D = (C^{k/B})^{c/A} \qqand D' = ((C')^{k/B'})^{c/A'}\]
  Applying $c^{-1}$ to the left-hand square in (\ref{eq:pullingback}) produces a diagram
  \begin{diagram}
    { C^{k/B} & B & A \\
      (C')^{k/B'} & B' & A'\\};
    \mto{1-2}{1-3} \mto{2-2}{2-3} \eto{1-2}{2-2} \eto{1-3}{2-3}
    \eto{1-1}{2-1} \mto{1-1}{1-2} \mto{2-1}{2-2}
    \comm{1-2}{2-3} \comm{1-1}{2-2}
  \end{diagram}
  which corresponds, under $c$, to the pullback square on the right of the cube.
  Lemma~\ref{lem:axiomC} shows that the squares on the top and bottom of the
  cube must be distinguished.  To finish the proof of the lemma it remains to
  check that the back face of the cube is distinguished.  To prove this it
  suffices to check that, after applying $c$ to the m-morphisms in the diagram,
  it corresponds to a pullback square.  This is a straightforward diagram chase
  using the fact that all morphisms are monic.
\end{proof}

\begin{lemma} \label{lem:pbeq} Let $\C$ be a pre-ACGW category.  In any
  pseudo-commutative square
  \[\csq{A}{B}{C}{D}{f}{}{}{f'}\]
  if $f'$ is an isomorphism, so is $f$.
\end{lemma}

\begin{proof}
  Apply $k$ vertically.  This produces a pullback square
  \begin{diagram}
    { A^k & B^k \\ C & D \\};
    \arrowsquare[mmor]{(f')^k}{}{}{f}
  \end{diagram}
  Since $f'$ is an isomorphism, $(f')^k$ must be, as well.  Thus the pseudo-commutative
  square is mapped to an isomorphism inside $\Ar_\times \M$; in particular, both
  horizontal morphisms in the pseudo-commutative square must be isomorphisms.  Thus
  $f$ is an isomorphism, as desired.
\end{proof}

\section{D\'evissage} \label{sec:devissage}

We can now prove a direct analog to Quillen's d\'evissage \cite[Theorem
5.4]{quillen}.  Analogously to the case of exact and abelian categories, the
$K$-theory of an ACGW-category is defined to be the $K$-theory of the underlying
CGW-category.

As the definition of ``creation of colimits'' appears to differ from context to
context we include the definition needed for the next theorem here:
\begin{definition}
  A functor $F: \C \rto \D$ \emph{creates restricted pushouts} if for every
  diagram
  \[B \lto A \rto C\]
  in $\C$, if
  \[F(B) \lto F(A) \rto F(C)\] has a restricted pushout in $\D$, then there
  exists a $D\in\C$ such that $D$ is the pushout of the original diagram, and
  $F(D)$ is the pushout of its image under $F$.
\end{definition}

\begin{theorem} \label{thm:devissage} Let $\A$ be a full pre-ACGW-subcategory of
  the pre-ACGW-category $(\B,\phi,c,k)$, closed under subobjects and quotients
  (see Definition~\ref{def:closure}), such that the inclusion
  $\A \cap \E \rto \E$ creates restricted pushouts.  Suppose that for all
  objects $B \in \B$ there is a sequence
  \[
  \emptyset = B_0 \mrto B_1 \mrto \cdots \mrto
  B_n = B
  \]
  such that $B_{i-1}^{c/B_i}$ is in $\A$ for all $i
  =1,\ldots n$. Then the inclusion functor $\A \rcofib \B$ induces an
  equivalence $K(\A) \rto K(\B)$.
\end{theorem}

\begin{proof}
  The proof proceeds exactly as in \cite{quillen}. Let $\iota:\A \rto \B$ be the
  inclusion of $\A$ into $\B$.  We would like $\iota$ to give a homotopy
  equivalence
  \[B Q \A \rto^{BQ\iota} B Q \B. 
  \] By Quillen's Theorem $A$ it is enough to show that $Q \iota_{/B}$ is
  contractible for any $B \in \B$.  Since $\A$ is closed under subobjects,
  $Q \iota_{/B}$ is the full subcategory of $Q\B_{/B}$ of those objects
  \[A_1 \erto B_2 \mrto B \]
  where $A_1 \in \A$.  By Lemma~\ref{lem:preorder}, $Q\B_{/B}$ is a
  preorder, and thus $Q \iota_{/B}$ is also a preorder.

  By the hypothesis of the theorem, there exists a sequence
  \[
    \emptyset = B_0 \mrto B_1 \mrto \cdots
    \mrto B_n = B
  \]
  with $B_{i-1}^{c/{B_i}} \in \A$ for all $i=1,\ldots,n$.  We prove that $Q\iota_{/B_n}$
  is contractible by induction on $n$.

  We have $B_1 \in \A$; in this case $Q\iota_{/B_1}$ is contractible,
  since it has the terminal object $B_1 \erto^{=} B_1 \mrto^{=} B_1$.

  To prove the inductive step it suffices to show that for any $h:B \mrto B'$ with
  $B^c \in \A$  the map $Q\iota_{/B} \rto Q\iota_{/B'}$ induced by
  postcomposition is a homotopy equivalence.  Let $L_B^\A\B$ be the full
  subcategory of $L_B\B$ containing those objects $B_1 \mrto B_2 \mrto
  B$ where $B_1^{c/B_2} \in\A$.  By Lemma~\ref{lem:preorder} it
  suffices to check that the functor $\iota:L_B^\A\B \rto L_{B'}^\A\B$ induced by
  postcomposition with $h$ is a homotopy equivalence.  

  Let $B_1 \mrto B_2 \mrto B'$ be any object of $L_{B'}^\A\B$.  We have the
  diagram
  \begin{diagram}
    { B_1\times_{B'} B & B_2\times_{B'} B & B \\
      B_1 & B_2 & B'\\};b
    \mto{1-1}{1-2}^{g'} \mto{1-2}{1-3}
    \mto{2-1}{2-2}
    \mto{2-2}{2-3}
    \mto{1-1}{2-1}
    \mto{1-2}{2-2}
    \mto{1-3}{2-3}^{h}
  \end{diagram}
  where both squares are pullback squares.  We define functors
  \begin{align*}
    r:L_{B'}^\A\B\rto L_B^\A\B&\qquad \makeshort{ r(B_1\mrto B_2\mrto B') =
      B_1\times_{B'}B \mrto^{g'} B_2\times_{B'} B\mrto B}. \\
    s:L_{B'}^\A\B \rto L_{B'}^\A\B&\qquad \makeshort{ s(B_1 \mrto B_2 \mrto B')
      = B_1\times_{B'}B \mrto B_2 \mrto B'.}
  \end{align*}
  If $s$ is well-defined (so $(B_1\times_{B'}B)^{c/B_2} \in \A$) then so is $r$,
  because $(B_1\times_{B'}B)^{c/g'}$ is a subobject of
  $(B_1\times_{B'}B)^{c/B_2}$. Thus we just need to check that $s$ is
  well-defined.

  First, by Axiom (U) 
  there exists a map $(B_2\times_{B'}B)^{c/B_2} \mrto B^{c/B'}$; since
  $B^{c/B'}\in \A$, it follows that $(B_2\times_{B'}B)^{c/B_2}$ must be, as
  well.  Now by Axiom (S),
  $(B_1\times_{B'}B)^{c/B_2} \cong B_1^{c/B_2} \star_{Y^{c/B_2}}
  (B_2\times_{B'}B)^{c/B_2}$, where
  $Y = B_1\star_{B_1\times_{B'}B} (B_2\times B' B)$, which exists by Axiom (S);
  since the inclusion $\A\cap \E \rto \E$ creates restricted pushouts, if each
  component of this pushout is in $\A$, then so is $(B_1\times_{B'}B)^{c/B_2}$.
  By assumption $B_1^{c/B_2}\in \A$ and by the above
  $(B_2\times_{B'}B)^{c/B_2}\in \A$, so $Y^{c/B_2}$ is also in $\A$ (as $\A$ is
  closed under subobjects).  Thus $(B_1\times_{B'}B)^{c/B_2}\in \A$, and $s$ is
  well-defined, as desired.

  Redrawing the above diagram, we have the following diagram:
  \begin{diagram}
    { B_1 & B_2 & B' & 1_{L^\A_{B'}\B}\\
      B_1\times_{B'} B & B_2 & B' &s\\
      B_1\times_{B'} B & B_2\times_{B'}B & B'& \iota r\\
    };
    \mto{1-1}{1-2} \mto{1-2}{1-3}
    \mto{2-1}{2-2} \mto{2-2}{2-3}
    \mto{3-1}{3-2} \mto{3-2}{3-3}
    \mto{2-1}{1-1} \eq{1-2}{2-2} \eq{1-3}{2-3}
    \eq{2-1}{3-1} \mto{3-2}{2-2} \eq{2-3}{3-3}

    \To{1-4}{2-4} \To{3-4}{2-4}
  \end{diagram}
  The upper row of squares gives a natural transformation $1_{L^\A_{B'}}
  \B \Rto s$; the lower row gives a natural transformation $\iota r \Rto
  s$.  Since natural transformations realize to homotopies, we see that $\iota r$
  is homotopic to the identity on $L_{B'}^\A\B$.  On the other hand, $r\iota$ is
  equal to the identity on $L_B^\A\B$, so these produce a homotopy equivalence
  of spaces, as desired.
\end{proof}

We can now apply this theorem to compare the $K$-theory of varieties to the
$K$-theory of reduced schemes of finite type.

\begin{example} \label{ex:var->sch} 
  We use the dual of Theorem~\ref{thm:devissage} to prove that $K(\Var) \simeq
  K(\Sch_{rf})$.  

   $\Var$ is a sucategory of $\Sch_{rf}$ closed under subobjects and
  quotients; the inclusion $\Var\cap \M \rto \M$ creates pushouts since the
  pushout of varieties along closed immersions is a variety \cite[Cor.~3.9]{schwede}.  To apply the
  theorem we must show that for every reduced scheme of finite type $X$ there
  exists a filtration
  \[X_0 \erto X_1 \erto \cdots \erto X_n = X\]
  such that $X_i \smallsetminus X_{i-1}$ is a variety for all $i$.  Since $X$ is
  of finite type there exists a finite cover of $X$ by affine opens
  $U_1,\ldots,U_n$; each of these is reduced since $X$ is and separated because
  each is affine.  
  We then define
  \[X_i = \bigcup_{j=1}^i U_i.\]
  This gives a finite open filtration of $X$; it remains to show that $X_i
  \smallsetminus X_{i-1}$ is a variety for all $i$.  Note that $X_i
  \smallsetminus X_{i-1} = U_i \smallsetminus \bigcup_{j=1}^{i-1} (U_j \cap
  U_i)$.  This is reduced, separated and of finite type, and is thus a variety,
  as desired. 
\end{example}

\section{Relationship with the $S_\bullet$-construction}\label{subdivision}

In this section we relate our $Q$-construction to a variation of the
$S_\dotp$-construction of Waldhausen \cite{waldhausen}.  We will show that the
$Q$-construction is equivalent to the construction defined for $\Var_{/k}$ in
\cite{campbell}.  As the $S_\dotp$-construction applied to an abelian category
is not abelian, it is unreasonable to expect that in all cases it will be
possible to iterate the construction.  However, as the $S_\dotp$-construction
for ACGW-categories produces a CGW-category, it is possible to iterate it
twice.  It turns out that this is sufficient to prove a cofiber sequence and the
relationship to the $Q$-construction.

\begin{remark}
  In the interest of keeping this section short and readable, we do not state
  definitions or results in the full generality that would be analogous to
  Waldhausen's exposition.  Instead, we restrict attention to the special cases
  of interest to us.  
\end{remark}

We begin by presenting the definition of the $S_\dotp$ construction for CGW-categories. 

\begin{definition}
  Let $\C$ be a CGW-category. Define $S_\dotp \C$ to be the simplicial set with $n$ simplices $S_n \C$ given by diagrams in the double category $\C$
  \begin{diagram}
    {C_{00} & C_{01} & C_{02} & \cdots & C_{0(n-1)} & C_{0n}\\
       & C_{11} & C_{12} & \cdots & C_{1(n-1)} & C_{1n}\\
       &  &  &  &  & \vdots\\
     &  &  &  &  & C_{nn}\\};
    \mto{1-1}{1-2} \mto{1-2}{1-3} \mto{1-3}{1-4} \mto {1-4}{1-5} \mto{1-5}{1-6}
     \mto{2-2}{2-3} \mto{2-3}{2-4} \mto {2-4}{2-5} \mto{2-5}{2-6}
    \eto{2-2}{1-2} \eto{2-3}{1-3} \eto {2-5}{1-5} \eto{2-6}{1-6}
    \eto{3-6}{2-6}
    \eto{4-6}{3-6}
  \end{diagram}
  such that
  \begin{enumerate}
  \item $C_{ii} = \initial$ for all $i$, and 
  \item Every subdiagram
    \begin{diagram}
      {C_{ki} & C_{kl} \\
        C_{ji} & C_{jl}\\};
      \mto{1-1}{1-2} \mto{2-1}{2-2} \eto{2-1}{1-1} \eto{2-2}{1-2}\dist{1-1}{2-2}
    \end{diagram}
    for  $k < j$ and $i < l$ is distinguished. 
  \end{enumerate}
  The face and degeneracies are defined as in the usual $S_\dotp$-construction:
  the $i$th face map is deleting the $i$th row and $i$th column, and the
  degeneracies are given by repetition.  (For more on the traditional
  $S_\dotp$-construction, see \cite[Section 1.3]{waldhausen}; for a more
  explicit description of how this works in the case of varieties, see the
  $\tilde S_\dotp$-construction in \cite[Definition 3.31]{campbell}.)
\end{definition}

\begin{remark}
The arrow directions in the diagram are chosen to agree with existing examples.
\end{remark}

\begin{example}
  When $\C = \Var$, then $S_\dotp \C$ is exactly the $\widetilde{S}_\dotp$ construction of \cite[Def. 3.31]{campbell}.
\end{example}

\begin{definition}
  Given a CGW category $(\C,\M,\E)$ define
  \[
  K^S (\C) := \Omega |S_\dotp \C|
  \]
\end{definition}

\begin{remark}
When $\C$ is, for example, an exact category this agrees with Waldhausen's $S_\dotp$-construction by Corollary 2 following \cite[Lem. 1.4.1]{waldhausen}. 
\end{remark}

\begin{remark}\label{agreement}
  In \cite{campbell}, the author introduced the $\widetilde{S}_\dotp$
  construction, which is a version of the Waldhausen construction that works on
  $SW$-categories \cite[Defn~3.23]{campbell}. These categories are meant to
  encode cutting and pasting, just as CGW categories do. In fact, in that paper
  there are three notions of such categories that appear: 1. pre-subtractive
  category 2. subtractive categories and 3. SW-categories. Pre-subtractive are
  closely related to CGW-categories; they are categories where one can define a
  higher geometric object that encodes cutting and pasting. Subtractive
  categories correspond to ACGW-categories: certain pushouts and pullbacks are
  required to exist. Finally, $SW$-categories, like Waldhausen categories, are
  allowed to have weak equivalences other than isomorphisms. Subtractive
  categories satisfy the axioms for ACGW-categories, and in this case the
  corresponding $S_\dotp$ constructions are equivalent and, in fact, equal; in
  such situations we will say that the ACGW-category \emph{arises from a
    subtractive category}. An ACGW-category where the distinguished squares are
  \emph{cartesian} in the underlying category $\mathcal{A}$, is an
  $SW$-category, and we may use the full machinery of $SW$-categories. This is
  true, for example, for $\mathbf{Sch}_{rf,/k}$ and $\mathbf{FinSet}_\ast$.
\end{remark}

As expected, this new definition of $K$-theory is equivalent to the original one.

\begin{theorem} \label{thm:S.=Q}
  Let $(\C, \M, \E)$ be a CGW category. Then there is a weak equivalence of topological spaces
  \[
  K^S(\C) \rwe K(\C)
  \]
  induced by a map of simplicial sets $S_\dotp\mc{C} \rto Q \mc{C}$. 
\end{theorem}

The equivalence above is one of topological spaces, not of infinite loop spaces
or spectra. While in many cases the equivalences are equivalences of infinite
loop spaces, that statement is not true in this generality (for example, smooth varieties
cannot be delooped in the way described in \cite[Sec.~5]{campbell} since it relies on
the existence of pushouts). We hope to address deloopings in future work.

In order to make the proof of Thm.~\ref{thm:S.=Q} as formally similar to the classical ``$S_\dotp = Q$'' theorem due to Waldhausen (\cite[Sect~1.9]{waldhausen}) we introduce the following definition.

\begin{definition}
  Let $(\C,\M,\E)$ be a CGW-category. We define $iS_n \C$ to be a category with
  \begin{description}
  \item[objects] Elements of $S_n \C$ 
  \item[morphisms] A collection of isomorphisms $f_{ij}: C_{ij} \to C'_{ij}$ in $\M$ such that the diagrams
    \begin{diagram}
      {C_{ik} & C_{lk} &\qquad& C_{ij} & C_{ik}\\
        C'_{ik} & C'_{lk} && C'_{ij} & C'_{ik} \\};
      \mto{1-1}{1-2} \mto{2-1}{2-2}  \mto{1-1}{2-1}_{f_{ik}} \mto{1-2}{2-2}^{f_{lk}}
      \eto{1-4}{1-5} \eto{2-4}{2-5} \eto{1-4}{2-4}_{\phi(f_{ij})} \eto{1-5}{2-5}^{\phi(f_{ik})} 
    \end{diagram}
commute in $\M$ and $\E$, respectively. 
  \end{description}
\end{definition}

\begin{remark}
We could have also used the isomorphisms in $\E$ in the above definition. The isomorphism $\phi$ guarantees the resulting definition is categorically equivalent to the one above. 
\end{remark}

\begin{proof}[Proof of Theorem~\ref{thm:S.=Q}]
  The definitions are designed to make this statement work exactly as in Waldhausen \cite[Sect~1.9]{waldhausen}. Let $i Q \mathcal{C}$ be the double category where vertical morphisms are isomorphisms in $Q \mathcal{C}$ and horizontal morphisms are morphisms in $Q \mathcal{C}$. Taking the nerve in the horizontal direction, we obtain a simplicial category $i Q_\dotp \mathcal{C}$. There is an equivalence $|Q \mathcal{C}| \xrightarrow{\simeq} |iQ_\dotp \mathcal{C}|$ given by Waldhausen's Swallowing Lemma \cite[Lem.~1.6.5]{waldhausen}.

Similarly, let $\mathbf{sd}\ i S_\dotp \mathcal{C}$ be the simplicial category we obtain from edgewise subdividing the $S_{\dotp}$-construction (for an introduction and proof of the properties of edgewise subdivision see \cite[App. 1]{segal_config}). There is now a functor
\[
\mathbf{sd}\ i S_\dotp \mathcal{C} \rto i Q_\dotp \mathcal{C}
\]
defined as in \cite[Sect.~1.9]{waldhausen}. It is a level-wise categorical equivalence, and thus induces a weak equivalence of bisimplicial sets, by the usual realization lemma (see, e.g. \cite[Lem.~5.1]{waldhausen_free}). 

Altogether we have
\[
|i S_\dotp \mathcal{C}| \lto^\cong |\mathbf{sd}\ i S_\dotp \mathcal{C}| \to^\simeq |i Q_\dotp \mathcal{C}| \inlineArrow{<-}^\simeq |Q \mathcal{C}|
\]
where the first map, a homeomorphism, is given by \cite[Prop.~A.1]{segal_config}.

Finally, we have the commutative diagram
  \begin{diagram}
    {{} |S_\dotp \mc{C}| & {} |\mbf{sd} S_\dotp \mc{C}| & {} |Q \mc{C}| \\
     {} |i S_\dotp \mc{C}| & {} |\mbf{sd} i S_\dotp \mc{C}| & {} |i Q \mc{C}| \\};
    \to{1-2}{1-1}_{\cong} \we{1-1}{2-1} \we{2-2}{2-1} \we{1-2}{2-2} \we{2-2}{2-3} \to{1-2}{1-3}  \we{1-3}{2-3}
  \end{diagram}
where we know that all of the indicated arrows are weak equivalences, and so the
remaining arrow is a weak equivalence. The composite across the top $|S_\dotp
\mc{C}| \rto |Q\mc{C}|$ is thus a weak equivalence. Upon taking loop spaces this gives the statement of the theorem. 
\end{proof}

As a corollary we can now show that D\'evissage works for SW-categories that are
the ambient categories of pre-ACGW-categories.

\begin{corollary} \label{cor:var-sch}
  Let $\A$ and $\C$ be pre-ACGW-categories satisfying the conditions of
  Theorem~\ref{thm:devissage}.  Then the map
  \[K^S(\A) \to K^S(\C)\]
  is an equivalence.  In particular, if $\A$ and $\C$ are constructed from
  SW-categories \cite{campbell} then the induced maps on $K$-theories of the
  SW-categories is also an equivalence.
\end{corollary}

We now use Waldhausen's approach to define relative $K$-theory
(Definition~\ref{defn:relative_K}) and prove a homotopy fiber sequence between
the relative $K$-theory and ordinary $K$-theories (Proposition~\ref{prop:fiber},
analogous to \cite[Prop. 1.5.5]{waldhausen}).  These will be needed in
Section~\ref{sec:comparison} to prove that the previous constructions of the
$K$-theory of varieties are equivalent.

\begin{defn}\label{S._cgw_defn}
  Let $\mc{A}$ be an ACGW-category. We define a CGW-structure on $S_n\mc{A}$. We give $S_n \mc{A}$ distinguished families of $\M$ and $\E$ morphisms as follows. 
  \begin{description}
  \item[$\M$-morphisms] A collection of maps $f_{ij}\colon C_{ij} \mrto D_{ij}$  in $\M$ such that
    \begin{diagram}
      {C_{ij} & D_{ij} & C_{ik}  & D_{ik}\\ 
        C_{ik} & D_{ik} & C_{lk} & D_{lk}  \\};
      \mto{1-1}{1-2} \mto{1-1}{2-1} \mto{1-2}{2-2} \mto{2-1}{2-2}
      \mto{1-3}{1-4} \mto{2-3}{2-4} \eto{1-3}{2-3} \eto{1-4}{2-4} 
    \end{diagram}
    are in $\Ar_\times \M$  and $\Ar_{\circlearrowleft} \mc{M}$, respectively. We visualize these as cubes
    \begin{diagram}
      {C_{ij} & & C_{ik} &\\
      & D_{ij} & & D_{ik} \\
      C_{lj} & & C_{lk} & \\
      & D_{lj} & & D_{lk} \\};
      \mto{1-1}{1-3} \eto{1-3}{3-3} \eto{1-1}{3-1} \mto{3-1}{3-3}
      \mto{1-1}{2-2} \mto{1-3}{2-4} \mto{3-1}{4-2} \mto {3-3}{4-4}
      \over{2-2}{2-4} \mto{2-2}{2-4} \over{2-2}{4-2} \eto{2-2}{4-2} \eto{2-4}{4-4} \mto{4-2}{4-4}
    \end{diagram}
  \item[$\E$-morphisms] A collection of maps $g_{ij}\colon C_{ij} \to D_{ij}$ in $\E$ such that
    \begin{diagram}
       {C_{ij} & D_{ij} & C_{ik}  & D_{ik}\\ 
         C_{ik} & D_{ik} & C_{lk} & D_{lk}  \\};
       \mto{1-1}{2-1} \mto{1-2}{2-2} \eto{1-1}{1-2} \eto{2-1}{2-2}
       \eto{1-3}{2-3} \eto{1-4}{2-4} \eto{1-3}{1-4} \eto{2-3}{2-4}
    \end{diagram}
    are in $\Ar_{\circlearrowleft} \E$ and $\Ar_{\times} \E$, respectively. We visualize these as cubes
        \begin{diagram}
      {C_{ij} & & C_{ik} &\\
      & D_{ij} & & D_{ik} \\
      C_{lj} & & C_{lk} & \\
      & D_{lj} & & D_{lk} \\};
      \mto{1-1}{1-3} \eto{1-3}{3-3} \eto{1-1}{3-1} \mto{3-1}{3-3}
      \eto{1-1}{2-2} \eto{1-3}{2-4} \eto{3-1}{4-2} \eto {3-3}{4-4}
      \over{2-2}{2-4} \mto{2-2}{2-4} \over{2-2}{4-2} \eto{2-2}{4-2} \eto{2-4}{4-4} \mto{4-2}{4-4}
    \end{diagram}
  \item[Distinguished squares] Let $C_{\cdot\cdot}, D_{\cdot\cdot}, E_{\cdot\cdot}, F_{\cdot \cdot}$ denote objects in $S_n \A$. A distinguished square consists of $\M$-morphisms $\C_{\cdot\cdot} \mrto D_{\cdot\cdot}$, $E_{\cdot\cdot} \mrto F_{\cdot\cdot}$ and $\E$-morphisms $\C_{\dot\dot} \erto E_{\dot\dot}$, $\D_{\dot\dot} \erto E_{\dot\dot}$ such that each
    \begin{diagram}
      {C_{ij} & D_{ij} \\
        E_{ij} & F_{ij}\\};
      \mto{1-1}{1-2}\mto{2-1}{2-2} \eto{1-1}{2-1} \eto{1-2}{2-2}
    \end{diagram}
    is distinguished
  \item[The functors $\phi,c, k$] The isomorphism $\phi$ is induced from the isomorphisms on $\A$. The functors $c, k$ are defined pointwise. The fact that the resulting squares are as described is guaranteed by Definition~\ref{defn:pre_acgw}, Axiom (U). 
  \end{description}
\end{defn}

We now describe the enhanced double category structure on $S_n \mc{A}$.

\begin{description}
\item[Enhanced Structure]  The enhanced double category structure on $S_n \mc{A}$, we define pseudo-commutative squares pointwise. That is, let $C_{\cdot\cdot}, D_{\cdot\cdot}, E_{\cdot\cdot}, F_{\cdot \cdot}$ denote objects in $S_n \A$. An element of $\Ar_{\circlearrowleft} S_n \mc{A}$ is given by $C_{\cdot\cdot} \mrto D_{\cdot\cdot}$ and $E_{\cdot\cdot} \erto F_{\cdot\cdot}$ and $C_{\cdot\cdot} \erto E_{\cdot\cdot}$ and $D_{\cdot\cdot} \mrto F_{\cdot\cdot}$ such that each
  \begin{diagram}
      {C_{ij} & D_{ij} \\
        E_{ij} & F_{ij}\\};
      \mto{1-1}{1-2}\mto{2-1}{2-2} \eto{1-1}{2-1} \eto{1-2}{2-2}
  \end{diagram}
    is in $\Ar_{\circlearrowleft} \M$. The 2-cells $\Ar_{\times} \M_{S_n \A}$, $\Ar_{\circlearrowleft} \E_{S_n \A}$ and $\Ar_{\times} \E_{S_n \A}$ are defined similarly. 
\end{description}

With the definitions above, the following is tedious, but straightforward. Indeed, the definitions were chosen to make this lemma true. 

\begin{lemma}\label{S._is_cgw}
  $S_n \mc{A}$, with the structure from Definition~\ref{S._cgw_defn}, satisfies
  all of the axioms of a CGW-category except for Axiom (A).  In particular, the
  $S_\dotp$-construction can be applied to $S_\dotp\mc{A}$. 
\end{lemma}

Using this we can define the relative $S_\dotp$-construction. 

\begin{definition}\label{relative_cgw} \label{defn:relative_K}
  A pair $(\mc{B}, \mc{A})$ of an ACGW-category $\mc{B}$ and a sub-ACGW-category
  $\mc{A}$ is \emph{good} if $\mc{A}$ is full and if for every isomorphism
  $B \rto^{\cong} B'$ in $\mc{B}$, $B$ is in $\mc{A}$ if and only if $B'$ is.
  For a good pair $(\mc{B}, \mc{A})$, define $S_n (\mc{B}, \mc{A})$ via the
  pullback
  \begin{diagram}
    {S_n (\mc{B}, \mc{A}) & S_{n+1} \mc{B}\\
      S_n \mc{A}  & S_n \mc{B}.\\};
    \cofib{1-1}{1-2} \to{1-2}{2-2}^{d_0} \to {1-1}{2-1} \cofib{2-1}{2-2}
  \end{diagram}
  In other words, $S_n(\mc{B},\mc{A})$ is the full subcategory of those objects $C_{\dotp\dotp}$
  in $S_{n+1}\mc{B}$ in which $C_{ij}\in \A$ for all $i > 0$.

  The category $S_n (\mc{B}, \mc{A})$ inherits the structure of a
  CGW-category. 
  The \emph{relative $K$-theory} of $(\mc{B}, \mc{A})$ is defined to be
  \[
    K(\mc{B}, \mc{A}) := \Omega | S_\dotp (\mc{B},\mc{A})|
  \]
\end{definition}

To conclude the section we prove an analog of additivity for the
$Q$-construction and use it to construct a homotopy fiber sequence relating
relative $K$-theory to the $K$-theory of the component categories.

\begin{proposition}[Additivity and Cofiber
  sequence] \label{lem:additivity} \label{prop:fiber} Let $(\mc{B}, \mc{A})$ be
  a good pair which arises from a subtractive category and a full subtractive
  subcategory. Then there exists a weak equivalence
  \begin{diagram}
    {QS_n (\mc{B}, \mc{A}) & Q\mc{B} \times QS_n \mc{A}.\\};
    \we{1-1}{1-2}
  \end{diagram}
  Moreover, the following is a homotopy fiber sequence after geometric realization:
  \begin{diagram}
    {Q\mc{B} & Q S_\dotp (\mc{B}, \mc{A}) & QS_\dotp \mc{A}.\\};
    \to{1-1}{1-2} \to {1-2}{1-3}
  \end{diagram}
\end{proposition}

\begin{proof}
  For any object $C_{\dotp,\dotp}$ in $S_n\mc{B}$, write $C_{\dotp- 1,
    -1}$ for the object in $S_{n-1}\mc{B}$ containing all elements with positive
  indices.  
  When $C_{\dotp,\dotp}\in S_n(\mc{B},\mc{A})$, $C_{\dotp-1,\dotp-1}$ can be
  considered to lie in $S_n\mc{A}$.    
  
  There are functors 
  \begin{general-diagram}{0em}{2.5em}
    { S_n(\B,\A) & \B & \qqand & S_n(\B,\A) & S_n \A \\
      C_{\dotp,\dotp} & C_{0,0} & & C_{\dotp,\dotp} & C_{\dotp-1,\dotp-1} \\};
    \to{1-1}{1-2}^{F'} \to{1-4}{1-5}^{F''}
    \goesto{2-1}{2-2} \goesto{2-4}{2-5}
  \end{general-diagram}
  which induce a map $f:QS_n(\mc{B},\mc{A}) \rto Q\mc{B}\times QS_n\A$.  This
  map is a coretraction, where the reverse map is constructed using the
  subtractive structure of $\mc{B}$.  These fit into a commutative diagram
  \begin{diagram}
    {Q \mc{B} & Q S_n (\mc{B}, \mc{A}) & Q S_n \mc{A}\\
      Q \mc{B} & Q \mc{B} \times Q S_n \mc{A} & Q S_n \mc{A}\\};
    \to{1-1}{1-2} \to{1-2}{1-3} \to{2-1}{2-2} \to{2-2}{2-3} \eq{1-1}{2-1} \to{1-2}{2-2}^f \eq {1-3}{2-3} 
  \end{diagram}
  in which the bottom row is a homotopy fiber sequence (in fact a trivial fiber
  sequence), $Q \mc{B} \rto QS_n \mc{A}$ is constant, and $QS_n \mc{A}$ is
  connected.  Thus, by \cite[Prop.~5.2]{waldhausen_free}, to prove that the
  geometric realization of
  \begin{diagram}
    {Q \mc{B} & Q S_n (\mc{B}, \mc{A}) & QS_n \mc{A}\\};
    \to{1-1}{1-2} \to {1-2}{1-3}
  \end{diagram}
  is a homotopy fiber sequence it suffices to check that $f$ is a weak
  equivalence; thus the second part of the proposition follows from the first.
  
  Consider the following commutative diagram:
  \begin{diagram}
    { Q S_n (\mc{B}, \mc{A}) & Q \mc{B} \times Q S_n \mc{A} \\
      S_\dotp S_n (\mc{B}, \mc{A}) & S_\dotp \mc{B} \times S_\dotp S_n \mc{A}\\};
    \to{1-1}{1-2}^f \we{2-1}{1-1} \we{2-2}{1-2} \to{2-1}{2-2}^{f'}
  \end{diagram}
  The vertical arrows are weak equivalences by Theorem~\ref{thm:S.=Q} and
  \cite[Lemma.~5.1]{waldhausen_free}. Thus $f$ is a weak equivalence if and only
  if $f'$ is.  By assumption, $\mc{A}$ and $\mc{B}$ arise from $SW$-categories,
  and as $S_\dotp$-constructions for ACGW-categories that arise from
  $SW$-categories agree by definition, $f'$ is a weak equivalence by
  \cite[Proposition 5.3]{campbell}.  Thus $f'$ must also be a weak equivalence,
  and the proposition follows.
\end{proof}

\begin{remark}
  In fact, the assumption that $\mc{A}$ and $\mc{B}$ arise from $SW$-categories
  can be significantly weakened; the only assumption necessary is that axioms
  (A) and (PP) hold sufficiently functorially.  In order to check this it is
  necessary to check that all steps in the proofs of \cite[Theorem 4.5,
  Proposition 5.3]{campbell} work analogously in ACGW-categories.  However, as
  this would significantly disrupt the flow of this paper (and not add
  significantly to understanding) we omit this more general result here;
  instead, we restrict solely to the case in which it is needed later in the
  paper. 
\end{remark}

\section{Localization of ACGW-categories} \label{sec:localization}

In this section we state the new definition necessary to state the localization
theorem.  The goal of a localization theorem is to identify the homotopy cofiber
of the map $K(\A) \rto K(\C)$ induced by the inclusion of a sub-CGW-category.
In order to prove the cleanest version of the theorem it is necessary to make
extra assumptions about the structure of $\A$ and $\C$, and thus passage to
ACGW-categories is necessary.  In addition, in order to ensure that objects in
$\A$ can be worked with easily, we assume some nice closure properties on $\A$
(similar to the closure properties assumed by Quillen).

Let $\C=(\E,\M)$ be an ACGW-category, and let $\A$ be a full ACGW-subcategory
closed under subobjects, quotients and extensions, as defined in
Definition~\ref{def:closure}.  The first step towards stating localization is
identifying the CGW-category whose $K$-theory we hope to be the cofiber.

The idea of the localized category is to define morphisms $A \to B$ to be
morphisms defined from a ``dense subset'' to a ``dense subset,'' where ``dense''
is defined to be subobjects/quotients whose cokernel/kernel are in the
subcategory $\A$.  This is motivated by the definition of
monomorphism/epimorphism in an abelian quotient category $\C/\A$, where (for
example), a monomorphism $A \to B$ in the quotient category is a diagram
\[A \lcofib A' \rcofib B' \rfib B\]
where the cokernel of $A' \rcofib A$ and the kernel of $B' \rfib B$ are both in
$\A$.  We can commute the monomorphism and epimorphism past one another (in an
epic-monic factorization) to instead write this as a diagram
\[A \lcofib A' \rfib B'' \rcofib B\] where the cokernel of $A' \rcofib A$ and
the kernel of $A' \rfib B''$ must be in $\A$.  Two such diagrams are equivalent
when they have a ``common refinement'' on which they are identical.  This is
exactly the definition of an m-morphism in the localized category, except that
we are allowed to ``reduce the size of $A$'' by both an m-morphism and an
e-morphism.
 
\begin{definition} \label{def:C-A}
  Let $A \mrto B$ be a morphism in $\M$.  We write $\mDrto$ if $A^c \in \A$.
  We define $\eDrto$ analogously.
  
  Let $\C\bs\A$ be the double category with
  \begin{description}
  \item[objects] the objects of $\C$,
  \item[m-morphisms] A morphism $A \mrto B$ is an equivalence class of
    diagrams in $\C$
    \[A \eDlto A' \mDlto X\eDrto B'\mrto B.\]
    If there exists a diagram in $\C$
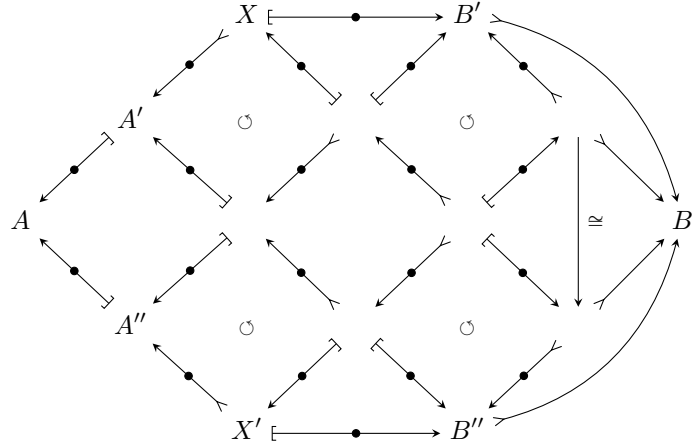
    \begin{diagram}
      {& & X && B' \\
        & A' && \phantom{Z} && \phantom{Z} \\
        A && \phantom{Z} && \phantom{Z} && B \\
        & A'' && \phantom{Z} && \phantom{Z} \\
        & & X' && B'' \\};
      \eDto{2-2}{3-1} \mDto{1-3}{2-2} \eDto{1-3}{1-5} \diagArrow{mmor,bend left}{1-5}{3-7}
      \eDto{4-2}{3-1} \mDto{5-3}{4-2} \eDto{5-3}{5-5} \diagArrow{mmor,bend right}{5-5}{3-7}
      \eDto{3-3}{2-2} \eDto{3-3}{4-2}
      \mDto{3-5}{2-4} \mDto{3-5}{4-4} \mDto{4-4}{3-3} \mDto{2-4}{3-3}
      \eDto{3-5}{2-6} \mto{2-6}{3-7}
      \eDto{3-5}{4-6} \mto{4-6}{3-7}
      \eDto{2-4}{1-3} \comm{2-2}{2-4} \eDto{2-4}{1-5}
      \mDto{2-6}{1-5} \comm{2-4}{2-6}
      \eDto{4-4}{5-3} \comm{4-2}{4-4} \eDto{4-4}{5-5} \mDto{4-6}{5-5}
      \comm{4-4}{4-6}
      \to{2-6}{4-6}^{\cong}
    \end{diagram}
    then the two formal compositions around the outside are considered
    equivalent.  The right-most square with the isomorphism in the middle is the
    same square that determines when two morphisms in $Q\C$ are equivalent.

    Composition is defined via a similar type of diagram, commuting the
    different types of morphisms past one another.
  \item[e-morphisms] A morphism $A \erto B$ is an equivalence class of diagrams
    in $\C$
    \[A \mDlto A' \eDlto X \mDrto B' \erto B.\]
    The equivalence relation between these is defined to be the dual condition
    to the condition on m-morphisms.
  \item[distinguished squares]   The distinguished squares are generated by the
    distinguished squares in $\C$ and axiom (I).  For a more detailed
    description, see Appendix~\ref{app:C-A}.
  \end{description}

  In this section we will often be working with morphisms in $\C\bs\A$ as
  represented by diagrams in $\C$.  As these categories have the same objects
  this can get confusing.  To help with this, we denote morphisms in $\C$ by
  arrows with straight shafts, and morphisms in $\C\bs\A$ by morphisms with wavy
  shafts.  We can thus say that an m-morphism $A \msrto B$ in $\C\bs\A$ is
  represented by a diagram
  \[A \eDlto A' \mDlto X \eDrto B' \mrto B\]
  in $\C$.
  
  We define $c:\Ar_\square \M \rto \Ar_\triangle \E$ on objects by
  $c(A \msrto B) = c_\C(B' \mrto B)$, and
  $k:\Ar_\square\E \rto \Ar_\triangle \M$ by $k(A \esrto B) = k_C(B' \erto B)$.
\end{definition}

There is a functor of double categories $s: \C\rto \C\bs\A$ which takes each
object to itself and takes every morphism to itself.

\begin{remark} \label{rm:cgw} As currently defined, $\C\bs\A$ does not have the
  structure of a CGW-category, as we cannot prove that the definitions of $c$
  and $k$ give equivalences of categories.  Proving that such a structure exists
  appears to require a development of a theory of a left calculus of fractions
  for a double category.  As this is far beyond the scope of this paper, we
  state as a condition of the localization theorem that $\C\bs\A$ extends to a
  CGW-category in a fashion compatible with the CGW-structure on $\C$ and the
  functor $s:\C\rto \C\bs\A$ and show that this works for our relevant examples.
  In Proposition~\ref{prop:C-A-well-def} we show that as long as $c$ and $k$
  give equivalences of categories, $\C\bs\A$ is a well-defined CGW-category.  In
  future work we hope to simplify this condition.

  If $\C\bs\A$ is a CGW-category then by definition the functor $s$ is a
  CGW-functor.
\end{remark}

Before turning to the main theorem we revisit the example of the localization of
an abelian category in detail, as the above definition is by no means easy to
understand.

\begin{example}\label{ex:abloc}
  Let $\C$ be an abelian category and $\A$ a Serre subcategory, considered as
  ACGW-categories. Then we claim that $\C\bs\A$ is exactly the abelian category
  $\C/\A$, considered as an ACGW-category.  First, consider the monics.  A
  morphism in $\C$ is monic in $\C/\A$ exactly when it can be represented by a
  zigzag
  \[X \lto^s Z \rto^f Y\]
  where the kernel and cokernel of $s$ are in $\A$, and when the kernel of $f$
  is in $\A$.  Writing both $s$ and $f$ in an epic-monic factorization and
  switching to the notation of CGW-categories, such a monic can be represented
  by a zigzag
  \[X \mDlto X' \eDrto Z \eDlto Y' \mrto Y.\] As $\C$ is abelian, e-morphisms
  are closed under pullbacks (i.e., epimorphisms are closed under pushouts in
  $\C$), and thus this representation is equivalent to the representation
  \[X \mDlto X' \eDlto X'\times_Z Y' \eDrto Y' \mrto Y.\]
  Using Lemma~\ref{lem:opensquare} we can swap the order of the two arrows on
  the left half, to produce a representation
  \[X \eDlto X'' \mDlto X'\times_Z Y' \eDrto Y' \mrto Y,\]
  as desired.  Given that we can also reverse this construction, we see that the
  monics (and, analogously, the epics) are as represented.

  Since $\C/\A$ is abelian it immediately follows that $\C\bs\A$ must be a
  CGW-category.
\end{example}

Before we state the main theorem, we need some auxillary definitions.

\begin{definition} \label{def:IV}
  Let $V$ be an object in $\C\bs\A$. The category $\I^m_V$ has as its objects
  pairs $(N,\phi)$, where $N\in \C$ and $\phi: sN \inlineArrow{smor,->} V$ is an
  isomorphism in $\C\bs\A$. A morphism $(N,\phi) \rto (N', \phi')$ is an
  equivalence class of diagrams $g:N \eDlto^{g_e} Y \mDrto^{g_m} N'$ (where
  diagrams are allowed to differ by an isomorphic choice of $Y$) such that
  $\phi's(g) = \phi$.  Here, $s(g)$ is considered as an isomorphism in
  $\C\bs\A$.  Composition is defined using mixed pullbacks.

  The category $\I_V^e$ is defined analogously with the roles of m-morphisms and
  e-morphisms swapped.

  If $\I_V^m$ is filtered for all $V$ we say that $\A$ is \emph{m-well-represented} in
  $\C$.  Dually, if $\I_V^e$ is filtered for all $V$ we say that $\A$ is
  \emph{e-well-represented} in $\C$.
\end{definition}

We think of $\I_V^m$ as the category of representatives inside $\C$ of an
isomorphism class of objects in $\C\bs\A$.  When this category is filtered it
means that representatives of $V$ can always be chosen compatibly, at least in
the m-morphism direction.  

\begin{definition}
  Suppose that for every diagram
  \[A \mDrto B \eDrto C\]
  in $\C$ there exists a pseudo-commutative square
  \begin{diagram}
    {A' & B \\ A' & C \\};
    \eq{1-1}{2-1} \eDto{1-2}{2-2}
    \mDto{1-1}{1-2} \mDto{2-1}{2-2} \comm{1-1}{2-2}
  \end{diagram}
  such that $A' \mDrto B$ factors through $A \mDrto B$.  Then we say that $\A$
  is \emph{m-negligible} in $\C$.  If the same statement holds with the
  m-morphisms and e-morphisms swapped, we say that $\A$ is \emph{e-negligible}
  in $\C$.
\end{definition}

Negligibility is a ``dual'' notion to well-representability.  Whereas
well-representability states that representatives can always be compatibly
combined, negligibility says that certain representatives can be ignored.  If
$\A$ is m-negligible in $\C$ this means that we never have to think about
e-components of morphisms inside $Q\C$; all such morphisms can be represented
(up to pseudo-commutative square) purely as an m-morphism.

We are now ready to state the CGW version of localization.

\begin{theorem} \label{thm:qLoc}
  Suppose that $\C$ is an ACGW-category and $\A$ is a sub-ACGW-category
  satisfying the following conditions:
  \begin{itemize}
  \item[(W)] $\A$ is m-well-represented or m-negligible in $\C$ and $\A$ is
    e-well-represented or e-negligible in $\C$.
  \item[(CGW)] $\C\bs\A$ is a CGW-category.
  \item[(E)] For two diagrams $A \eDlto X \mDrto B$ and $A \eDlto X' \mDrto B$
    which represent the same morphism in $\C\bs\A$ there exists an e-morphism
    $C \eDrto B$ and an isomorphism $\alpha: X \oslash_B C \rto X' \oslash_B C$
    such that the induced diagram
    \begin{diagram}
      { A & X \oslash_B C \\
        X' \oslash_B C & C \\};
      \mDto{1-2}{2-2} \mDto{2-1}{2-2}
      \eDto{1-2}{1-1} \eDto{2-1}{1-1}
      \to{1-2}{2-1}^\alpha
    \end{diagram}
    commutes.  The same statement holds with e-morphisms and m-morphisms
    swapped.
  \end{itemize}
  Then the sequence
  \[K(\A) \rto K(\C) \rto K(\C\bs\A)\]
  is a homotopy fiber sequence.
\end{theorem}

We postpone the proof of Theorem~\ref{thm:qLoc} until
Section~\ref{sec:proofloc}.  As mentioned in Remark~\ref{rm:cgw}, in order for
condition (CGW) to hold it suffices to check that $c$ and $k$ (as defined on
objects) extend to equivalences of categories.  In this section we focus on two
applications of the theorem.

The first application is a sanity check, showing that in the case of an abelian
category the theorem is the same as Quillen's localization \cite[Theorem
5.5]{quillen}.

\begin{example}
  Continuing Example~\ref{ex:abloc}, we show that Theorem~\ref{thm:qLoc} applies
  in this example; thus Theorem~\ref{thm:qLoc} is truly a generalization of
  Quillen's localization theorem.
  
  Consider condition (W); we will show that $\A$ is both m- and
  e-well-represented in $\C$.  By symmetry it suffices to check that $\I^m_V$ is
  filtered.  An object $(N,\phi) \in \I^m_V$ is an object $N\in \C$ together
  with a mod-$\A$-isomorphism $N \rto V$; a morphism $(N,\phi) \rto (N', \phi')$
  is a morphism $g:N \rto N'$ in $\C$ such that $\phi's(g) = \phi$.  Suppose
  that we are given two morphisms $g,g': (N,\phi) \rto (N',\phi')$.  Then the
  morphism $N' \rto N'/\mathrm{im}(g-g')$ is a mod-$\A$-isomorphism which
  equalizes $g$ and $g'$; thus $\I^m_V$ has coequalizers.  Now suppose that we
  are given two objects $(N,\phi)$ and $(N',\phi')$ in $\I^m_V$.  Choosing
  representatives appropriately, these give a diagram in $\C$
  \begin{diagram}
    {\widetilde N && V' && \widetilde{N'} \\
      N & N\oplus_{\tilde N} V'& V& N' \oplus_{\tilde{N'}}V' & N' \\
      & & (N \oplus_{\tilde N} V')\oplus_{V'} (N' \oplus_{\tilde N'} V') \\};
    \to{1-1}{1-3} \to{1-5}{1-3}
    \diagArrow{Dmor,right hook->}{1-1}{2-1}
    \diagArrow{Dmor,right hook->}{1-5}{2-5}
    \diagArrow{Dmor,->>}{2-3}{1-3}
    \diagArrow{Dmor,right hook->}{1-3}{2-2}
    \diagArrow{Dmor,right hook->}{1-3}{2-4}
    \to{2-1}{2-2} \to{2-5}{2-4}
    \diagArrow{Dmor,right hook->}{2-2}{3-3}
    \diagArrow{Dmor,right hook->}{2-4}{3-3}
    \diagArrow{densely dashed,->}{3-3}{2-3}^\psi
  \end{diagram}
  where the bulleted arrows represented mod-$\A$-isomorphisms.  The object
  $((N\oplus_{\tilde N} V')\oplus_{V'} (N'\oplus_{\tilde N'} V'), \psi)$ then
  represents an object under both $(N,\phi)$ and $(N', \phi')$.  Thus $\I_V^m$
  is filtered, as desired.

  It remains to check (E).  This is simply the fact that for any two morphisms
  $A \rto B$ in $\C$ which map to the same isomorphism in $\C/\A$ there is a
  quotient of $B$ (by an object in $\A$) on which they are equal---in other
  words, this is the observation that if $g$ and $g'$ represent the same
  isomorphism in $\C/\A$ then $\im (g-g')$ is in $\A$. 
\end{example}

The second example is the case of reduced schemes of finite type of bounded
dimension; we will be using this example in Section~\ref{sec:comparison} to
compare different models of the $K$-theory of varieties.

\begin{example}
  Let $\Sch_{rf}^d$ be the category of reduced schemes of finite type over $k$
  which are at most $d$-dimensional.  As mentioned in Example~\ref{ex:sch},
  $\Sch_{rf}$ is an ACGW-category; since morphisms can only increase the
  dimension of a scheme it follows directly that $\Sch_{rf}^d$ is also an
  ACGW-category.  

  We claim that Theorem~\ref{thm:qLoc} applies for $\Sch_{rf}^{d-1}
  \subseteq \Sch_{rf}^d$. We check the conditions in turn.

  First, consider condition (W).  We claim that $\Sch_{rf}^{d-1}$ is
  m-well-represented and e-negligible in $\Sch_{rd}^d$.  Here, an isomorphism in
  $\Sch_{rf}^d\bs\Sch_{rf}^{d-1}$ is (the germ of) an isomorphism between open subsets whose
  complements are at most $d-1$-dimensional.  Thus when considering an
  isomorphism we can discard all irreducible components of dimension less than
  $d$.  In addition, we can assume that all $d$-dimensional components are
  smooth and consider isomorphisms to be birational isomorphisms.  To check that
  $\Sch_{rf}^d$ is m-well-represented it suffices to check that for any two
  representatives of a birational isomorphism there exists a common dense open
  subset on which they are defined.  This is clearly true.

  To check that $\Sch_{rf}^{d-1}$ is e-negligible in $\Sch_{rf}^d$ we note that
  for any diagram $A \eDrto B \mDrto C$ if we take the nonsingular locus of the
  $d$-dimensional irreducible components of $C$ and intersect it with the image
  of $A$ we get exactly the desired subset, as all that the inclusion $B \mDrto
  C$ can add is either (a) disjoint components of dimension less than $d$ or (b)
  components of dimension less than $d$ that intersect $d$-dimensional
  components.  In case (b) the intersections are singular in $C$, so when we
  remove them we produce exactly the desired morphism.

  We now check condition (CGW).  Proposition~\ref{prop:C-A-well-def} states that
  for $\C\bs\A$ to be a CGW-category we are only required to show that $c$ and
  $k$ are well-defined equivalences of categories; the other axioms follow
  directly from the definitions.  In $\Sch^d_{rf}\bs\Sch^{d-1}_{rf}$ all objects
  are canonically isomorphic to the disjoint union of their $d$-dimensional
  connected components, so it suffices to consider these examples.  By
  definition, both the e-morphisms and m-morphisms in
  $\Sch^d_{rf} \bs \Sch^{d-1}_{rf}$ are birational isomorphisms of the domain
  with a subset of the components of the codomain.  Both $c$ and $k$ simply take
  the components not hit by the morphism.  Consider taking each object to its
  set of connected components; from the definition of the distinguished squares
  (see Appendix~\ref{app:C-A}) a square in $\Sch^d_{rf}\bs \Sch^{d-1}_{rf}$ is
  distinguished if and only if the produced square in the category of finite
  sets is distinguished.  The fact that $c$ and $k$ are equivalences of
  categories thus follows from the fact that they are induced from $c$ and $k$
  on the category $\FinSet$.

  It remains to check condition (E).  Since $\oslash$ in $\Sch_{rf}^d$ is simply
  intersection of schemes the condition as stated follows by the same argument
  as the negligibility condition above.  To check the condition with m-morphisms
  and e-morphisms reversed, let $A_d$ be the $d$-dimensional irreducible
  components of $A$.  Then $A_d \mDrto X \times_A X'$ exists, and the maps
  $A_d \mDrto X \eDrto B$ and $A_d \mDrto X' \eDrto B$ are equal inside the
  (ordinary) category of schemes (since they must be equal on a dense open
  subset, as they are equivalent in $\Sch_{rf}^d\bs\Sch_{rf}^{d-1}$).  Factoring
  this morphism as $A_d \eDrto C \mDrto B$ gives the desired object $C$.

  We now observe that, by the equivariant Barratt--Priddy--Quillen theorem,
  \[K(\Sch_{rf}^d \bs \Sch_{rf}^{d-1}) \simeq \bigoplus_{\alpha \in B_n}
    \Omega^\infty \Sigma^\infty B \Aut(\alpha).\] Here, $B_n$ is the set of
  birational isomorphism classes of schemes of dimension $d$, and $\Aut(\alpha)$
  is the group of birational automorphisms of a representative of the class.
\end{example}

\section{A comparison of models} \label{sec:comparison}

In this section we compare both authors' models for $K(\Var_{/k})$. Write
$K^C(\Var_{/k})$ for the $K$-theory of varieties defined as in \cite{campbell},
and let $K^Z (\Var_{/k})$ denote the model in \cite{zakharevich_assembler}. We
then have the following comparison theorem.

\begin{theorem} \label{thm:equiv}
  $K^C (\Var_{/k})$ is weakly equivalent to $K^Z(\Var_{/k})$.
\end{theorem}

The rest of this section focuses on the proof of the theorem.  For conciseness
we fix the base field $k$ and omit it from the notation.  To prove the theorem
we construct an auxilliary SW-category $\Sch_{rfw}$ and show that there are weak
equivalences
\[K^C(\Var) \rwe K^C(\Sch_{rf}) \rwe K^C(\Sch_{rfw}) \lwe K^Z(\Var).\] Recall
that $\Sch_{rf}$ is the ACGW-category of reduced schemes of finite type
(Example~\ref{ex:sch}). By an abuse of notation, we also write $\Sch_{rf}$ for
the SW-category of reduced schemes of finite type (see Remark~\ref{agreement}).
The left-hand map is an equivalence by Corollary~\ref{cor:var-sch}, so we focus
on the zig-zag on the right.

\begin{remark}
  Constructing the weak equivalence on the right (and checking that it is, in
  fact, a weak equivalence) is a relatively straightforward exercise in
  simplicial objects (see Proposition~\ref{prop:ZtoC}), and has been known to
  the authors for several years.  The most difficult part of this proof is
  actually checking that the middle map (which is induced by an inclusion of
  SW-categories) is a weak equivalence on $K$-theory.  In Waldhausen categories,
  this is analogous to the following question: suppose that $\C$ is a
  Waldhausen category in which the weak equivalences do not satisfy the
  Extension axiom \cite[p. 327]{waldhausen}.  Let $\C'$ be the Waldhausen category
  with the same underlying category and cofibrations as $\C$ together with the
  minimal set of weak equivalences that includes all weak equivalences in $\C$
  and satisfies Extension.  Does the natural functor $\C \to \C'$ induce a weak
  equivalence on $K$-theory?  The authors could not find an answer to this
  question, but the current example on schemes produces an interesting example
  where the answer is ``yes.''  
\end{remark}

\begin{definition} \label{def:Varw} We define a new SW-category $\Sch_{rfw}$.
  Its underlying category is $\Sch_{rf}$, the category of reduced schemes of
  finite type.  We define the structure maps by setting
  \begin{description}
  \item[cofibrations] the open immersions, and
  \item[complement maps] the closed immersions, and
  \item[weak equivalences] those morphisms $f:X \rto Y$ such that
  there exists a stratification
  \[\emptyset = Y_0 \rcofib^{cl} Y_1 \rcofib^{cl} \cdots \rcofib^{cl} Y_n = Y\]
  of $Y$ by closed immersions such that for all $i$, the induced map
  $f_i:X\times_Y (Y_i \setminus Y_{i-1}) \rto Y_i\setminus Y_{i-1}$ is an
  isomorphism.  
  \end{description}
\end{definition}

\begin{remark}
This is equivalent to the statement that there is a corresponding filtration $X_i$ on $X$ such that $f_i \colon X_i \setminus X_{i-1} \rto Y_i \setminus Y_{i-1}$ is an isomorphism. We sometimes use the condition in this form. 
\end{remark}

We state the relevant definitions on the assembler-side of the equivalence.

\begin{definition}
  The assembler $\Var$ (resp. $\Sch_{rf}$) has as objects the varieties
  (resp. reduced schemes of finite type), with morphisms the locally closed
  immersions.  The topology on $\Var$ (resp. $\Sch_{rf}$) is generated by the
  coverage consisting of pairs $\{Y \rcofib X, X \smallsetminus Y \rcofib X\}$,
  where $Y \rcofib X$ is a closed immersion.

  The inclusion of assemblers $\Var \rto \Sch_{rf}$ incudes an equivalence of
  $K$-theories by \cite[Theorem B]{zakharevich_assembler}, as every reduced scheme of
  finite type has a finite disjoint cover by varieties.
\end{definition}

As the proof of Theorem~\ref{thm:equiv} has many parts, we begin by presenting
the basic outline.  This will reduce the proof to showing that certain morphisms
are equivalences on $K$-theory, and the rest of the section will focus on each
of those maps in turn.

\begin{proof}[Outline of proof for Theorem~\ref{thm:equiv}]
  The category of reduced schemes of finite type comes equipped with a
  filtration by dimension.  This filtration is inherited by $\Sch_{rf}$ and $\Sch_{rfw}$,
  and the inclusion $\Sch_{rf} \rto \Sch_{rfw}$ is compatible with this filtration.  Note that 
  \[K^C(\Sch_{rf}) = \hocolim_n K^C(\Sch_{rf}^n),\] and similarly for $K^Z(\Sch_{rf})$ and
  $K^C(\Sch_{rfw})$.  Thus to show the theorem it suffices to show that there exist
  equivalences $K^C(\Sch_{rf}^n) \rto K^C(\Sch^n_{rfw})$ and $K^Z(\Sch^n) \rto
  K^C(\Sch^n_{rfw})$ for all $n$ which are compatible with the inclusions on the
  filtrations.

  Proposition~\ref{prop:ZtoC} constructs a map
  $K^Z(\Sch_{rf}^n) \rto K^C(\Sch_{rfw}^n)$ which is an equivalence for all $n$.  The
  map $K^C(\Sch_{rf}^n) \rto K^C(\Sch^n_{rfw})$ is induced by the identity map
  on the underlying categories (as both $\Sch_{rf}$ and $\Sch_{rfw}$ have the
  same underlying SW-category; they differ only in their choice of weak
  equivalences).

  Our proof proceeds by induction on $n$.  When $n = 0$, $\Sch_{rf}^0 = \Sch^0_{rfw}$, so
  the $K$-theories of these are equal.  We now assume that the natural inclusion
  $K^C(\Sch_{rf}^{n-1}) \rto K^C(\Sch_{rfw}^{n-1})$ is an equivalence.  Consider the following
  diagram: 
  \begin{equation} \label{eq:equivDiag1}
    \begin{tikzcd}
      K^{C} (\Sch_{rf}^{n-1}) \ar{d}{i} \ar{r}{\sim} & K^C(\Sch^{n-1}_{rfw}) \ar{d}{i'} & K^Z (\Sch_{rf}^{n-1}) \ar{l}[swap]{\sim}\ar{d}{i}\\
      K^{C} (\Sch_{rf}^n) \ar{r}{g} \ar{d} & K^C(\Sch^{n}_{rfw})\ar{d} & K^Z (\Sch_{rf}^n)\ar{l}[swap]{\sim}\ar{d} \\
      K^C(\Sch_{rf}^n,\Sch_{rf}^{n-1}) \ar{r}{g'} & K^C(\Sch^n_{rfw},\Sch^{n-1}_{rfw}) & K^Z((\Sch_{rf}^n/i)_\dotp) \ar{l}[swap]{f} 
    \end{tikzcd}
  \end{equation}
  The columns in this diagram are homotopy fiber sequences. The column on the
  right is produced by \cite[Theorem C]{zakharevich_assembler}, the other two
  columns are produced by \cite[Prop.~5.5]{campbell}. The maps between the
  columns are given below. Since the columns are homotopy fiber sequences of
  loop spaces, $f$ must be a weak equivalence by the five lemma.  The map $g$ is
  a weak equivalence if and only if $g'$ is, so we focus on proving that $g'$ is
  a weak equivalence.

  In Definitions~\ref{def:D}, \ref{def:beta}, \ref{def:rho}, and
  \ref{def:lambda} we show that there exists a category $\D$ and morphisms
  \begin{align}
    &\lambda: K^C(\Sch_{rf}^n,\Sch_{rf}^{n-1}) \rto K^C(\D), \label{eq:lambda}\\
    &\beta: K^C(\Sch^n_{rfw},\Sch^{n-1}_{rfw}) \rto K^C(\D), \hbox{ and}\label{eq:beta} \\
    &\rho: K^Z((\Sch_{rf}^n/i)_\dotp) \rto K^C(\D) \label{eq:rho}
  \end{align}
  making the following diagram commute:
  \begin{diagram}
    { K^C(\Sch_{rf}^n,\Sch_{rf}^{n-1})  & K^C(\Sch^n_{rfw},\Sch^{n-1}_{rfw}) & K^Z((\Sch_{rf}^n/i)_\dotp)  \\
      & K^C(\D) \\}; \to{1-1}{1-2}^{g'} \we{1-3}{1-2}_f \to{1-3}{2-2}^\rho
    \to{1-2}{2-2}_\beta \to{1-1}{2-2}_\lambda
  \end{diagram}
  Here, the top row is the bottom row of (\ref{eq:equivDiag1}).  The map $\beta$
  is a weak equivalence by Proposition~\ref{prop:beta}.  Thus we see that $g'$
  is an equivalence if and only if $\lambda$ is; that $\lambda$ is an
  equivalence is exactly the conclusion of Proposition~\ref{prop:lambda}.  Thus
  $g'$ is an equivalence, and the inductive step is complete.
\end{proof} 

We now turn our attention to filling in the details of the proof above.  We
begin by checking that $\Sch_{rfw}$ is well-defined. 

\begin{lemma}
  Let $X,Y,Z\in \Sch_{rfw}$ and suppose $X \rto Y$ and $Y \rto Z$ are weak
  equivalences. Then $X \rto Z$ is a weak equivalence.
\end{lemma}
\begin{proof}
  Recall that $X \rto Y$ being a weak equivalence is the statement that there is a stratification
  \begin{diagram}
    {\emptyset = Y_0 & Y_1 & \cdots & Y_n \\}; \mto{1-1}{1-2} \mto{1-2}{1-3}\mto{1-3}{1-4}
  \end{diagram}
  such that $X \times_{Y}(Y_i \setminus Y_{i-1}) \xrightarrow{\cong} Y_i \setminus Y_{i-1}$. Similarly for $Y \rto Z$. We must produce a new stratification of $Z$, call it $Z'_i$, such that $X \times_Z (Z'_i \setminus Z'_{i-1}) \xrightarrow{\cong} (Z'_i \setminus Z'_{i-1})$. We do this by stratifying each $(Z_i\setminus Z_{i-1})$ in turn, using the stratification of $Y$, and gluing these together. 

  The problem thus reduces to the following. Given $Y_1 \mrto Y_2$ and $Z_1 \mrto Z_2$ with an isomorphism $\varphi: Y_2 \setminus Y_1 \mrto Z_2 \setminus Z_1$, and a further stratification $Y_{1,0} \mrto \cdots \mrto Y_{1,n} = Y_2$, produce a corresponding stratification for $Z_1 \mrto Z_2$.  To do this, define $Z_{1, i} = Z_2 \setminus \varphi(Y_2 \setminus Y_{1,i})$. One checks that
  \[
  Z_{1,i}\setminus Z_{1,i-1} = (Z_2 \setminus \varphi (Y_2 \setminus Y_{1,i})) \setminus (Z_2 \setminus (Y_{1,i-1})) = \varphi(Y_2 \setminus Y_{1,i-1}) \setminus \varphi(Y_2 \setminus Y_{1,i}) \cong \varphi(Y_i \setminus Y_{i-1}) 
  \]
\end{proof}

\begin{lemma}
  $\Sch_{rfw}$ is an SW-category.
\end{lemma}

\begin{proof}
 For this we only need to check the axioms of SW-categories that apply to weak
 equivalences \cite[Defn.~3.24]{campbell}, which are wholly analagous to
 \cite[p.326]{waldhausen}. First, the isomorphisms are certainly contained in
 $w$. Second, we must check that subtraction respects weak equivalences. That
 is, if we have a commutative square with sides as indicated:
 \begin{diagram}
   { X & X' \\ Y & Y' \\};
   \cofib{1-1}{1-2} \cofib{2-1}{2-2} \we{1-1}{2-1} \we{1-2}{2-2}
 \end{diagram}
  then there is a weak equivalence $X' \setminus X \rto Y' \setminus Y$ making
  the induced square commute. Thus, we need a stratification on
  $Y' \setminus Y$. Since we are subtracting off $Y$, the stratification of $Y$
  will not come into play. Define the stratification to be
  \[ \initial = (Y' \setminus Y) \times_{Y'} Y'_0 \rcofib (Y'\setminus Y)
    \times_{Y'} Y'_1 \rcofib \cdots \rcofib Y' \setminus Y.\]
  
  Finally, we must check that in a diagram as below, where all the horizontal
  maps are cofibrations and the squares are pullbacks, the induced map between
  pushouts is a weak equivalence:
  \begin{diagram}
    { X'' & X & X'\\
      Y'' & Y & Y' \\};
    \diagArrow{left hook->}{1-2}{1-1} \cofib{1-2}{1-3}
    \diagArrow{left hook->}{2-2}{2-1} \cofib{2-2}{2-3}
    \we{1-1}{2-1} \we{1-2}{2-2} \we{1-3}{2-3}
  \end{diagram}
  Since $X' \rto Y'$ is a weak equivalence, $X \rto Y$ is trivially so: a
  stratification $Y'_i \mrto Y'$ pulls back to one on $Y$,
  $Y \times_{Y'} Y'_i \to Y$. A similar statement also holds for
  $X'' \rto Y''$ .

    It suffices to consider the case where both $X' \rto Y'$ and $X'' \rto Y''$
    are given by two step stratifications. Let these be $Y'_1 \mrto Y'$ and
    $Y''_1 \mrto Y''$. Denote the two induced stratifications on $Y$ by
    $Y^{(1)}_1 \mrto Y$ and $Y^{(2)}_1 \mrto Y$ so that
    $Y^{(1)}_1 = Y \times_{Y'} Y'_1$ and $Y^{(2)}_1 = Y \times_{Y''} Y''_1$. We
    now consider the three-step stratification
    \begin{diagram}
    {Y''_1 \amalg_{Y^{(2)}_1 \times_Y Y^{(1)}_1} Y'_1  & Y'' \amalg_{Y^{(1)}_1} Y'_1 & Y'' \amalg_Y Y'\\}; \mto{1-1}{1-2} \mto {1-2}{1-3}
    \end{diagram}
    One verifies that 
    \begin{align*}
      (Y'' \amalg_{Y^{(1)}_1} Y'_1)\setminus (Y''_1 \amalg_{Y^{(2)}_1 \times Y^{(1)}_1} Y'_1) & \cong   (Y'' \setminus Y''_1) \\
     (Y'' \amalg_Y Y')\setminus ( Y'' \amalg_{Y^{(1)}_1} Y'_1)  &\cong (Y' \setminus Y'_1)
    \end{align*}
\end{proof}

We now define our second helper-category, $\D$.

\begin{definition} \label{def:D}
  Let $\D$ be the category with
  \begin{description}
  \item[objects] finite disjoint unions of smooth $n$-dimensional varieties, written
    $\coprod_{i\in I} X_i$, where each $X_i$ is irreducible,
  \item[morphisms] $\coprod_{s\in S} X_s \rto \coprod_{t\in T} Y_t$ are maps of
    sets $f:S \rto T$ together with birational isomorphisms $X_s \rto
    Y_{f(s)}$.\footnote{Here, by ``birational isomorphism'' we mean an
      equivalence class of maps, rather than a specific map which is a
      birational isomorphism}
  \item[composition] induced by composition of set maps together with the
    composition of birational isomorphisms.
  \end{description}
  The category $\D$ has a forgetful functor to $\FinSet$ induced by mapping
  $\coprod_{s\in S} X_s$ to $S$.

  We put an ACGW-structure on $\D$ by declaring all morphisms with injective
  underlying maps of sets to be both e-morphisms and m-morphisms, and by setting
  the distinguished (resp. commutative) squares to be the squares that become
  distinguished (resp. commutative) in the ACGW-structure on $\FinSet$; the
  forgetful functor then becomes a functor of ACGW-categories. 

  The SW-structure on $\D$ is given by
  \begin{description}
  \item[cofibrations] morphisms whose underlying set map is injective,
  \item[complement maps] the same as the cofibrations, and
  \item[weak equivalences] isomorphisms.
  \end{description}

  With these definitions, the $S_\dotp$-construction gives equal structures for
  the $K$-theory of $\D$ considered as a CGW- or an SW-category.

  The ACGW-category $\D$ is equivalent to a disjoint union of categories of the
  form $\FinSet\wr G$, for $G$ a group of birational automorphisms (see Example~\ref{ex:wr}).
\end{definition}

The main work of this section goes into proving Propositions \ref{prop:ZtoC} and
\ref{prop:lambda} which together immediately imply Theorem~\ref{thm:equiv}.

\begin{proposition} \label{prop:ZtoC}
  For $n \leq \infty$,
  \[K^Z(\Sch_{rf}^n) \simeq K^C(\Sch_{rfw}^n),\]
  induced by taking each tuple of varieties in $\Sch_{rf}^n$ to their disjoint union.
\end{proposition}

\begin{proof}
  For conciseness of notation, we give the proof for the case $n = \infty$ and
  omit the $n$ from the notation.  The proof works identically for all finite
  $n$. Throughout this proof we freely use the notation and definitions of \cite{zakharevich_assembler}. 
  
  We construct a functor of simplicial categories
  $F_\dotp:\W(\Sch_{rf}^{\vee \dotp}) \rto wS_\dotp\Sch_{rfw}$ which has a
  levelwise right adjoint.  Thus the functor is levelwise a homotopy
  equivalence, and we get an equivalence on the geometric realizations of the
  simplicial categories.  This equivalence produces an equivalence
  $K^Z(\Sch_{rf})_1 \rto K^C(\Sch_{rfw})_1$, and (since these are both $\Omega$-spectra
  above level $1$) an equivalence of $K$-theories.

  The functor is defined in the following manner.  $\W(\Sch_{rf}^{\vee m})$ is the
  full subcategory of $\W(\Sch_{rf})^m$ consisting of those objects with disjoint
  indexing sets.  We will thus refer to objects of $\W(\Sch_{rf}^{\vee m})$ as tuples
  $(\{A_{1i}\}_{i\in I_1},\ldots,\{A_{mi}\}_{i\in I_m})$ in $\W(\Sch_{rf})^m$ and
  simply ensure that at all stages the indexing sets are disjoint.  Let
  $F_m(A_1,\ldots,A_m)$ be the functor $X : \tilde \Ar[m] \rto \Sch_{rfw}$ given by
  \[X_{i,j} = \coprod_{k=i+1}^j \coprod_{\ell\in I_k} A_{k\ell},\] with
  morphisms given by the natural inclusions into the coproduct.  A morphism of
  tuples gives a natural transformation of functors, each component of which is
  a weak equivalence in $\Sch_{rfw}$, so $F_m$ is well-defined.  The simplicial maps
  in $\W(\Sch_{rf}^{\vee \dotp})$ are induced by maps on the indexing sets, so
  these commute with the simplicial structure maps in $wS_\dotp\Sch_{rfw}$.  Thus
  $F_\dotp$ is a simplicial functor.

  It remains to check that $F_m$ has a right adjoint.  Given a diagram $X:\tilde
  \Ar[m] \rto \Sch_{rf}$, we define $G_m(X)$ to have as its $i$-th component
  $\{X_{0i}\bs X_{0(i-1)}\}_{\{i\}}$.

  We define the unit of the adjunction by taking each $\{A_{ji}\}_{i\in I_j}$ to
  $\{\coprod_{i\in I_j} A_{ji}\}_{\{j\}}$; this is a valid morphism in
  $\W(\Sch_{rf})$, so gives a valid morphism in $\W(\Sch_{rf})^m$, with the
  indexing sets disjoint by definition.

  Now consider $F_m \circ G_m$.  This takes a functor $X: \tilde \Ar[m] \rto
  \Sch_{rf}$ to the functor $X': \tilde \Ar[m] \rto \Sch_{rf}$, where
  \[X'_{ij} = \coprod_{k=i+1}^j X_{ij} \bs X_{i(j-1)}.\]
  There is a natural weak equivalence $X' \rto X$ by simply mapping each
  component to itself.  This gives the counit of the adjunction and completes
  the proof of the proposition.
\end{proof}

We can now define the map $\beta$ \eqnref{eq:beta}. 

\begin{definition} \label{def:beta} To define a map
  $K^C(\Sch^n_{rfw},\Sch^{n-1}_{rfw}) \rto K^C(\D)$ it suffices to define for
  all $r$, a map
  $|wS_\dotp^{(r)} S_\dotp(\Sch_{rfw}^n, \Sch_{rfw}^{n-1})| \rto |iS_\dotp^{(r)}
  \D|$.  In order to construct such a map it suffices to construct a partial
  functor $b_\dotp: S_\dotp(\Sch_{rfw}^n,\Sch_{rfw}^{n-1}) \rto \D$ defined on
  the subcategories of closed immersions, open immersions, and weak
  equivalences, as long as this functor is compatible with the simplicial
  structure maps and takes objects in the $S_\dotp$-construction to objects in
  the $S_\dotp$-construction.  An object of
  $S_m(\Sch_{rfw}^n, \Sch_{rfw}^{n-1})$ is a diagram
  \begin{diagram}
    { &&&& X_{m,m} \\
      &&&& \vdots \\
      && X_{2,2} & \cdots & X_{2,m} \\
      & X_{1,1} & X_{1,2} & \cdots & X_{1,m} \\
      Y_0 & Y_1 & Y_2 & \cdots & Y_m \\};
    \cofib{1-5}{2-5} \cofib{2-5}{3-5}
    \cofib{3-3}{3-4}^o \cofib{3-4}{3-5}^o
    \cofib{3-3}{4-3} \cofib{3-5}{4-5}
    \cofib{4-2}{4-3}^o \cofib{4-3}{4-4}^o \cofib{4-4}{4-5}^o
    \cofib{4-2}{5-2} \cofib{4-3}{5-3} \cofib{4-5}{5-5}
    \cofib{5-1}{5-2}^o \cofib{5-2}{5-3}^o \cofib{5-3}{5-4}^o \cofib{5-4}{5-5}^o
  \end{diagram}
  in which each $X_{i,j}\in \Sch^{n-1}_{rfw}$.  We define $b_m$ to take this
  diagram to the tuple containing the irreducible $n$-dimensional components of
  the nonsingular points of $Y_m$ (indexed over the set of irreducible
  $n$-dimensional components of $Y_m$).
\end{definition}

\begin{lemma}
  The partial functor $b_\dotp$ is well-defined and induces a map on $K$-theory.
\end{lemma}

\begin{proof}
  First, suppose that $X$ and $Y$ are irreducible and $n$-dimensional.  Then a
  weak equivalence $X \rwe Y$ is, by definition, a birational isomorphism.  In
  particular, this means that under $b_\dotp$, all weak equivalences in
  $S_\dotp(\Sch^n_{rfw}, \Sch^{n-1}_{rfw})$ are taken to isomorphisms in $\D$.
  An open embedding $X \rcofib^o Y$ is also a birational isomorphism; a closed
  embedding is an honest isomorphism, unless we allow $X$ to have dimension less
  than $n$; in that case, $X$ is taken to the empty tuple in $\D$.  Thus a
  diagram in the $S_\dotp$-construction of
  $S_\dotp(\Sch^n_{rfw}, \Sch^{n-1}_{rfw})$ is taken to a diagram with injective
  underlying maps of sets both vertically and horizontally (decorated with
  birational isomorphisms); the pushout condition translates to the analogous
  pushout condition on the underlying diagram of sets.  The weak equivalence
  direction is mapped to morphisms which are isomorphisms of the underlying maps
  of sets, decorated with birational isomorphisms.  This is exactly the
  $S_\dotp$-construction applied to $\D$, and thus each $b_m$ is well-defined.
  The simplicial structure maps never fully get rid of one of the $Y$'s in the
  bottom row of the diagram; since all of the horizontal maps in the diagram
  above are birational isomorphisms (as the complements have dimension strictly
  less than $n$) the partial functor is well-defined, and induces a map on
  $K$-theory.
\end{proof}

The map induced by $b_\dotp$ is $\beta$. 

We now consider the map $\rho$.

\begin{definition} \label{def:rho} The map
  $\rho: K^Z((\Sch_{rf}^n/i)_\dotp) \rto K^C(\D)$ is defined by a composition of
  two maps.  The first map is the map $K^Z((\Sch_{rf}^n/i)_\dotp) \rto K^Z(\D)$,
  defined by taking each irreducible scheme of dimension $n$ to its birational
  isomorphism type.  The second map $K^Z(\D) \rto K^C(\D)$ is induced by the map
  $K^Z(\D) \rto K(\SC(\D)) \rto K^C(\D)$, where the first map is the natural
  transformation taking $K^Z(\D)$ to the Waldhausen $K$-theory of the Waldhausen
  category $\SC(\D)$ (defined in \cite[Theorem 2.1, Proposition
  2.6]{Z-ass-pi1}), and the second map is induced by an equivalence of
  $K$-theories taking an object in $\SC(\D)$ (which is a tuple of tuples of
  birational isomorphism classes) to the ``flattened'' tuple (indexed by the
  disjoint union of the indexing sets of the tuples).  In this map, the key
  observation is that cofibrations in $\SC(\D)$ are cofibrations in the
  SW-category $\D$, while cofiber maps in $\SC(\D)$ are \emph{opposites} of
  complement maps in $\D$: a cofiber map in $\SC(\D)$ is induced by a map
  selecting a subset of the indexing set, and this is exactly a description of
  the opposite of a map in $\D$.
\end{definition}

\begin{proposition} \label{prop:beta}
  The map $\rho$ is a weak equivalence, and the diagram
  \begin{diagram}
    { K^Z((\Sch_{rf}^n/i)_\dotp) & K^C(\Sch^n_{rfw},\Sch^{n-1}_{rfw}) \\
      & K^C(\D)\\};
    \to{1-1}{1-2}^f \to{1-2}{2-2}^{\beta} \to{1-1}{2-2}_{\rho} 
  \end{diagram}
  commutes.  The map $\beta$ is therefore also a weak equivalence.
\end{proposition}

\begin{proof}
  The map $f$ is a weak equivalence because it is a map induced on homotopy
  cofibers by a pair of weak equivalences.  The map $\rho$ is a weak equivalence
  by \cite[Proposition 7.1]{zakharevich_assembler} (where it is the map $p$).

  We now check the commutativity of the diagram.  The map $f$ is defined
  analogously to the natural transformation in \cite[Proposition
  2.6]{Z-ass-pi1}, designed to make this triangle commute.  The map $\beta$ is
  defined analogously to $p$ in \cite[Proposition 7.1]{zakharevich_assembler},
  again designed to make this diagram commute.  In particular, both of these
  compositions take all objects in simplicial levels higher than $0$ (in the
  original categories) to the empty set, and take the objects in simplicial
  level $0$ to a tuple of birational isomorphism classes of varieties (with
  morphisms given by permutations of these decorated by birational
  isomorphisms).  The indexing set of each tuple is the set of irreducible
  components of the varieties, so this diagram does, indeed, commute. 

  Thus, by 2-of-3, $\beta$ is also a weak equivalence.
\end{proof}

Since $\beta$ is a weak equivalence, $g$ is a weak equivalence if and only if
$\lambda$ is.  Thus it remains to consider the map $\lambda$.

\begin{definition} \label{def:lambda}
  The map $\lambda: K^C(\Sch_{rf}^n,\Sch_{rf}^{n-1}) \rto K^C(\D)$ is defined to
  be the composition
  \[K^C(\Sch_{rf}^n, \Sch_{rf}^{n-1}) \rto K^C(\Sch_{rfw}^n, \Sch_{rfw}^{n-1})
    \rto^\beta K^C(\D).\]
\end{definition}

\begin{proposition} \label{prop:lambda} The map $\lambda$ is a weak equivalence.
\end{proposition}

\begin{proof}
  Let $\lambda': \Sch_{rf}^n \bs \Sch_{rf}^{n-1} \rto \D$ be the CGW-functor
  taking each variety of dimension $n$ to the set of birational isomorphism
  classes of its irreducible components.  This is actually an equivalence of
  categories (on the level of CGW-categories) with the inverse equivalence given
  by choosing a representative in each birational isomorphism class and taking
  an object in $\D$ to the disjoint union of its representatives.  Thus
  $\lambda'$ is a weak equivalence.

  Consider the following diagram:
  \begin{diagram}
    { \Omega|i S_\dotp S_\dotp (\Sch^n_{rf},\Sch^{n-1}_{rf})| &
      \Omega|QS_\dotp(\Sch^n_{rf},\Sch^{n-1}_{rf}) & \Omega |Q(\Sch^n_{rf}\bs\Sch^{n-1}_{rf})| \\
      \Omega|iS_\dotp \D| && \Omega|Q\D| \\};
    \we{1-1}{1-2} \we{1-2}{1-3} \we{2-1}{2-3} 
    \to{1-1}{2-1}_\lambda \to{1-3}{2-3}^{\lambda'}
  \end{diagram}
  The leftmost two horizontal maps are given by the natural transformation
  described in the proof of Theorem~\ref{thm:S.=Q} for the comparison between
  the $Q$-construction and the $S_\dotp$-construction.  The right-hand map in
  the top row is a weak equivalence by Theorem~\ref{thm:qLoc}.  Thus, by 2-of-3,
  $\lambda$ is a weak equivalence.
\end{proof}

\section{Proof of Theorem~\ref{thm:qLoc}} \label{sec:proofloc}

The goal of this section is to prove Theorem~\ref{thm:qLoc}.  The idea of the
proof is to use Quillen's Theorem B \cite[Theorem B]{quillen} applied to the
functor $Qs$.  There are therefore two steps to the proof: proving that the
theorem applies to $Qs$, and proving that the fiber agrees with $K(\A)$.

Let $i:\A \rto \C$ be the inclusion functor.  Then $Qi$ factors as
\begin{diagram}[6em]
  {Q\A & Qs_{\initial/} & Q\C \\};
  \to{1-1}{1-2}^{M \mapsto (M, 1_\initial)}
  \to{1-2}{1-3}^{(N,u) \mapsto N}
\end{diagram}
If Theorem B applies to $Qs$ then its fiber is $Qs_{\initial/}$.  In this case,
to show that the fiber agrees with $K(\A)$ it suffices to check that the
left-hand map in this factorization is a weak equivalence.  We see that the
theorem is thus a direct consequence of the following two propositions:

\begin{proposition} \label{prop:locinc}
  The inclusion $Q\A \rto Qs_{\initial/}$ is a homotopy equivalence.
\end{proposition}

\begin{proposition} \label{prop:locThmB}
  Quillen's Theorem B applies to the functor $Qs$. More concretely, for any $u:V
  \rto V'$ in $Q(\C\bs\A)$, the induced functor $u^*:Qs_{V'/} \rto Qs_{V/}$ is a
  homotopy equivalence.
\end{proposition}

The rest of this section is taken up with the proof of these two propositions.
We write $\C = (\E, \M)$ and $\A = (\E_\A, \M_\A)$.  
We begin by analyzing how morphisms in $\C\bs\A$ and $Q(\C\bs\A)$ work.  

\begin{lemma} \label{lem:1of3}
  $\M_\A$ and $\E_\A$ satisfies 1-of-3, in the sense that $\M_\A$ and $\E_\A$
  are subcategories of $\M$ and $\E$, respecively, and given any composable
  morphisms $f,g\in \M$ (resp. $\E$), if $gf\in \M_\A$ (resp. $\E_\A$) then so
  are $f$ and $g$.
\end{lemma}

\begin{proof}
  We prove this for $\M_\A$; the result for $\E_\A$ follows by duality.

  Suppose that we are given $f:A \mrto B$ and $g: B \mrto C$ in $\C$.  This
  corresponds to a diagram
  \begin{diagram}
    { & \initial & B^{c/g} \\      
      \initial & A^{c/f} & A^{c/gf} \\
      A & B & C \\ };
    \mto{1-2}{1-3} \mto{2-1}{2-2} \mto{2-2}{2-3} \mto{3-2}{3-3}
    \mto{3-1}{3-2} 
    \eto{1-3}{2-3} \eto{2-3}{3-3} \eto{1-2}{2-2} \eto{2-2}{3-2}
    \eto{2-1}{3-1}
    \dist{2-1}{3-2} \dist{2-2}{3-3} \dist{1-2}{2-3}
  \end{diagram}
  The lower-left square exists by the definition of $A^{c/f}$; the lower-right
  square exists by applying $k^{-1}$ to the bottom row; the upper square exists
  because $(A^{c/f})^{c/A^{c/gf}} \cong B^{c/g}$ by the definition of $c$.
  Consider the upper square; since $\A$ is closed under subobjects,
  quotients and extensions, $A^{c/gf}$ is in $\A$ if and only if $A^{c/f}$ and
  $B^{c/g}$ are.  Thus, if $f$ and $g$ are in $\M_\A$ so is $gf$ (showing that
  $\M_\A$ is a subcategory) and if $gf$ is in $\M_\A$ then $f$ and $g$ must be,
  as well. 
\end{proof}

\begin{lemma}  \label{lem:presMA}
  The categories $\E_\A$ and $\M_\A$ satisfy the following properties:
  \begin{itemize}
  \item[(a)] The subcategories $\E_\A$ and $\M_\A$ are preserved under pullbacks
    and mixed pullbacks along morphisms in $\E$ and $\M$.
  \item[(b)] Pullback squares and mixed pullback satisfy 3-of-4: if
    three of the morphisms in a square are in $\M_\A$ or $\E_\A$, the fourth
    must be as well.
  \end{itemize}
\end{lemma}

\begin{proof}
  We first prove (a). Suppose that we have a square
  \[\csq{A}{B}{C}{D}{f'}{}{}{f}.\]
  We want to show that if $f$ is in $\M_\A$, so is $f'$.  Applying $c$ to this
  diagram produces a pullback square
  \begin{diagram}
    { A^{c/f' }& B \\ C^{c/f} & D \\};
    \eto{1-1}{1-2} \eto{2-1}{2-2} \eto{1-1}{2-1} \eto{1-2}{2-2}
  \end{diagram}
  By definition, $C^{c/f'}\in \A$; thus, since $\A$ is closed under quotients,
  $A^{c/f}\in \A$, as desired.  The other proofs of closure under pullbacks
  follow analogously.  

  We turn our attention to (b). To check 3-of-4, consider a square as above
  where we know that $A \mrto B$ is in $\M_\A$ and $B \erto D$ is in $\E_\A$.
  Because $\E_\A$ is closed under pullbacks, it follows that
  $A^{c/f} \erto C^{c/f'}$ is also in $\E_\A$.  Thus we have a distinguished
  square
  \[\dsq{\initial}{A^{c/f}}{(A^{c/f})^{k/c}}{C}{}{}{}{}\]
  in which we know everything but $C$ is in $\A$.  Since $\A$ is closed under
  extensions, $C\in \A$ as well.  The other forms of 3-of-4 follow analogously.
\end{proof}

This proposition implies that we can identify the isomorphisms in $\C\bs\A$ in
the following manner:
\begin{lemma} \label{lem:morrep'}
  An m-morphism in $\C\bs\A$ represented by a diagram
  \[A \eDlto A' \mDlto X \eDrto B' \mrto B\] is an isomorphism if and only if
  $B' \mrto B$ is in $\M_\A$; the dual statement holds for e-morphisms.

  Any morphism $u:A \rto B$ in $Q(\C\bs\A)$ can be represented by a diagram
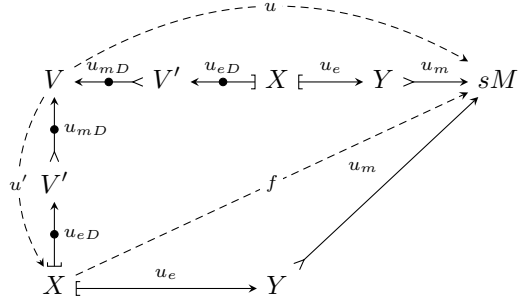
  \[A  \mDlto A'  \eDlto X \erto B' \mrto B \]
  in $\C$.
  Such a diagram represents an isomorphism if and only if $X \erto B'$ is in
  $\E_\A$ and $B' \mrto B$ is in $\M_\A$.  
\end{lemma}

\begin{proof}
  If $\B'\mrto B$ is in $\M_\A$ then the given diagram represents an isomorphism
  by definition (by reversing the composition for the inverse).  Conversely, if
  an m-morphism has an inverse then tracing through the definition of
  composition and using Lemmas~\ref{lem:1of3} and \ref{lem:presMA} gives
  that $B'\mrto B$ must be in $\M_\A$.

  A morphism $A \rto B$ in $Q(\C\bs\A)$ is represented by a composition of an
  e-morphism $A \esrto C$ and an m-morphism $C \msrto B$.  We can represent
  these by the top and right side of the following diagram:
  \begin{diagram}
    { A & A' & X & C' & C \\
      & \phantom{Z} & \phantom{Z}& \phantom{Z} & C'' \\
      A'' & & Z & \phantom{Z} & Y  \\
      & & & \phantom{Z} & B' \\
      & & B'' & & B \\};
    \mDto{1-2}{1-1} \eDto{1-3}{1-2} \mDto{1-3}{1-4} \eto{1-4}{1-5}
    \eDto{2-5}{1-5} \mDto{3-5}{2-5} \eDto{3-5}{4-5} \mto{4-5}{5-5}
    \eDto{2-4}{1-4} \eto{2-4}{2-5}
    \eDto{2-3}{1-3} \mDto{2-3}{2-4} \comm{1-3}{2-4}
    \comm{2-4}{3-5}
    \mDto{3-4}{2-4} \eto{3-4}{3-5} \eto{3-4}{4-5} \mto{3-4}{4-4} \eto{4-4}{5-5}
    \dist{4-4}{4-5}
    \eDto{2-3}{1-2} \mDto{3-3}{2-3} \mDto{3-3}{3-4}
    \mto{3-3}{4-4}
    \eDto{2-2}{1-1} \mDto{2-3}{2-2} \dist{1-2}{2-2} \mDto{3-3}{2-2}
    \eDto{3-3}{3-1} \mDto{3-1}{1-1} \dist{3-1}{2-2}
    \eto{3-3}{5-3} \mto{5-3}{5-5} \dist{5-3}{4-4}
  \end{diagram}
  The rest of the diagram shows that the composition around the bottom is an
  equivalent representation of this morphism; its construction liberally uses
  the previous lemmas and results about CGW-categories in Section~\ref{sec:cgw}.

  Since morphisms in $Q(\C\bs\A)$ are isomorphisms exactly when both components
  are isomorphisms (by Lemma~\ref{lem:iso}), the composition is an isomorphism
  if and only if the morphisms $C \erto C'$ and $B' \mrto B$ are isomorphisms,
  meaning that they are in $\E_\A$ and $\M_\A$, respectively.  If this is the
  case then $Z \erto B''$ and $B'' \mrto B$ are in $\E_\A$ and $\M_A$,
  respectively, and this represents an isomorphism.  Conversely, if this is an
  isomorphism then we must have $Z \erto B''$ and $B'' \mrto B$ in $\E_\A$ and
  $\M_\A$; tracing through and using that $\E_\A$ and $\M_\A$ satisfy 1-of-3 we
  obtain the converse.
\end{proof} 

We turn our attention to proving Proposition~\ref{prop:locinc}.

\begin{definition}
  Let $V\in Q(\C\bs\A)$, and let $\F_V$ be the full subcategory of $Qs_{V/}$ of
  those objects $(M, u:V \rto sM)$ in which $u$ is an isomorphism.
\end{definition}

Proposition~\ref{prop:locinc} is the $V = \initial$ case of the following:
\begin{proposition} \label{prop:10.7}
  The inclusion $\iota_V: \F_V \rto Qs_{V/}$ is a homotopy equivalence for all
  $V\in Q(\C\bs\A)$.
\end{proposition}

\begin{proof}
  By \cite[Theorem A]{quillen}, it suffices to check that for all
  $(M,u) \in Qs_{V/}$, the category $\iota_V/(M,u)$ is contractible for all
  $(M,u)$.  By the dual of \cite[Proposition 3, Corollary 2]{quillen} it
  suffices to check that it is a cofiltered category.  By
  Lemma~\ref{lem:morrep'}, $u$ can be represented by a diagram
  \[V \mDlto^{u_{mD}} V' \eDlto^{u_{eD}} X \erto^{u_e} Y \mrto^{u_m} sM.\]
  
  An object of $\iota_V/(M,u)$ is a triple $(u', M', f)$ of an isomorphism
  $u':V \rto sM'$ in $\F_V$ together with a morphism $f:M' \rto M$ in $Q\C$ such
  that $s(f)u' = u$.  A morphism $(u', M', f) \rto (u'', M'', f')$ is a morphism
  $g:M' \rto M''$ in $Q\C$ such that $f'g= f$.  In particular, there is a
  faithful forgetful functor to $Q\C_{/M}$; since by Lemma~\ref{lem:preorder}
  this is a preorder, so is $\iota_V/(M,u)$.  All it remains to check is that it
  is nonempty and that any two objects have a common object above them.

  To see that $\iota_V/(M,u)$ is nonempty, consider the following diagram in
  $\C$:
  \begin{diagram}
    { V & V' & X & Y & sM \\
      V' \\
      X & & Y \\};
    \mDto{1-2}{1-1}_{u_{mD}} \eDto{1-3}{1-2}_{u_{eD}} \eto{1-3}{1-4}^{u_e}
    \mto{1-4}{1-5}^{u_m}
    \mDto{2-1}{1-1}_{u_{mD}}
    \eDto{3-1}{2-1}_{u_{eD}}
    \eto{3-1}{3-3}^{u_e} \mto{3-3}{1-5}^{u_m}
    \node (m-6-7) at (m-1-5) {$\phantom{sM}$};
    \diagArrow{densely dashed,->,bend left}{1-1}{6-7}!u
    \diagArrow{densely dashed,->,bend right}{1-1}{3-1}!{u'}
    \diagArrow{densely dashed,->}{3-1}{1-5}!f
  \end{diagram}
  This represents an object of $\iota_V/(M,u)$ as desired.

  Now suppose that we are given two different objects of $\iota_V/(M,u)$; we
  want to show that there is an object mapping to both of them.  Suppose that
  the two objects are given by $(u':V \rto sM', f:M' \rto M)$ and $(u'': V \rto
  sM'', f':M'' \rto M)$.  Writing these in terms of their representations we get
  the outside of the following diagram; it is possible to complete the outside
  to the diagam on the inside because $s(f)u' =
  s(f')u''$.  
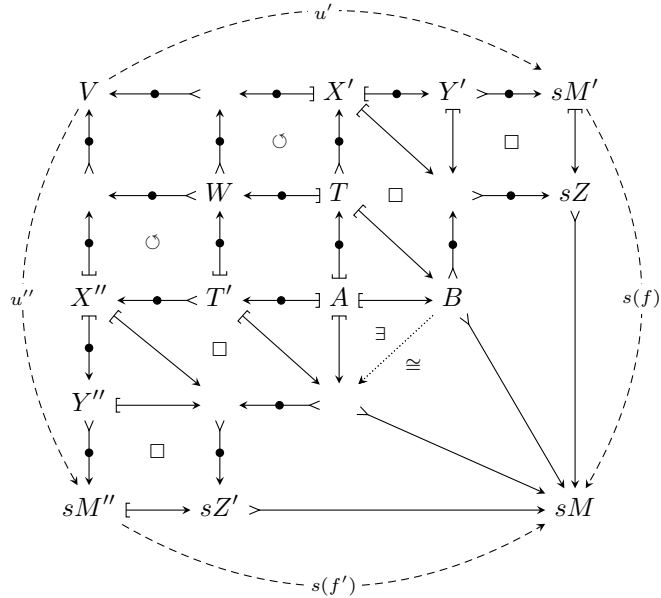
  \begin{diagram}
    { V & \phantom{Z} & X' & Y' & sM' \\
      \phantom{Z} & W & T & \phantom{Z} & sZ\\
      X'' & T' & A & B \\
      Y'' & \phantom{Z} & \phantom{Z} & \phantom{Z}\\
      sM'' & sZ' & && sM \\};
    \mDto{1-2}{1-1} \eDto{1-3}{1-2} \eDto{1-3}{1-4} \mDto{1-4}{1-5}
    \mDto{2-1}{1-1} \eDto{3-1}{2-1} \eDto{3-1}{4-1} \mDto{4-1}{5-1}
    \eto{5-1}{5-2} \mto{5-2}{5-5}
    \eto{1-5}{2-5} \mto{2-5}{5-5}
    \node (m-100-1) at (m-1-1) {$\phantom{V}$};
    \node (m-100-2) at (m-5-1) {$\phantom{sM''}$};
    \diagArrow{densely dashed,->,bend left}{100-1}{1-5}!{u'}
    \diagArrow{densely dashed,->,bend right}{1-1}{5-1}!{u''}
    \diagArrow{densely dashed,->,bend right}{100-2}{5-5}!{s(f')}
    \diagArrow{densely dashed,->,bend left}{1-5}{5-5}!{s(f)}
    \eto{1-4}{2-4} \mDto{2-4}{2-5} \dist{2-4}{1-5}
    \eto{4-1}{4-2} \mDto{4-2}{5-2} \dist{5-1}{4-2}
    \mDto{2-2}{2-1} \mDto{2-2}{1-2} \mDto{2-3}{1-3}
    \mDto{3-2}{3-1} \eDto{3-2}{2-2} \eDto{2-3}{2-2}
    \comm{1-2}{2-3} \comm{2-1}{3-2}
    \eDto{3-3}{3-2} \eDto{3-3}{2-3}
    \eto{3-1}{4-2} \eto{1-3}{2-4}
    \mDto{4-3}{4-2} \eto{3-2}{4-3} \dist{3-2}{4-2}
    \mDto{3-4}{2-4} \eto{2-3}{3-4} \dist{2-3}{2-4}
    \eto{3-3}{3-4} \eto{3-3}{4-3} \mto{3-4}{5-5} \mto{4-3}{5-5}
    \diagArrow{densely dotted,->}{3-4}{4-3}^\cong
    \diagArrow{densely dotted,->}{3-4}{4-3}_\exists
  \end{diagram}
  Consider the object represented by
  \[\makeshort{(A \eDrto W \mDrto V, A \erto B \mrto M)}.\]
  This is a well-defined morphism of $\iota_V/(M,u)$.  This comes with a
  morphism to $(u',f)$ given by the formal composition
  \[A \eDrto T \mDrto X' \eDrto Y' \mDrto M'\]
  and an analogous morphism to $(u'',f')$.  Thus $\iota_V/(M,u)$ is cofiltered,
  as desired.
\end{proof}

We now turn our attention to Proposition~\ref{prop:locThmB}; this proof is quite
complicated and will take the rest of this section.  In order to prove that
$u^*$ is a homotopy equivalence for all $u$ it suffices to show that it is true
for the morphisms $\initial \erto V$ and $\initial \mrto V$.  Since all of the
conditions of the theorem are symmetric in m-morphisms and e-morphisms, it
suffices to prove this for $\initial \mrto V$; we focus on this case for the
rest of this proof.  The key idea of the proof is to construct a category
$\sH_N$ and functors $P_{(N,\phi)}:\sH_N \rto \F_V$ and $k_N:\sH_N \rto Q\A$
such that the diagram
\begin{diagram-numbered}[4em]{diag:loc}
  { \sH_N & \F_V & Qs_{V/} \\
    Q\A & \F_\initial & Qs_{\initial/} \\};
  \to{1-1}{1-2}^{P_{(N,\phi)}} \cofib{1-2}{1-3}
  \to{2-1}{2-2}^{\cong} \cofib{2-2}{2-3}
  \to{1-1}{2-1}_{k_N} \to{1-3}{2-3}^{u^*}
\end{diagram-numbered}
commutes up to homotopy.  We will then show that $k_N$ and $P_{(N,\phi)}$ are
both homotopy equivalence.  From this Proposition~\ref{prop:locThmB} follows by
2-of-3 and Proposition~\ref{prop:10.7}.

We thus turn our attention to constructing $\sH_N$, $k_N$ and $P_{(N,\phi)}$.

\begin{definition} \label{def:HN}
  The category $\sH_N$ has as objects equivalence classes of diagrams
  \[M \eDlto^{h_e} X \mDrto^{h_m} N,\] where two diagrams are allowed to differ
  by an isomorphic choice of $X$. A morphism
  \[(M \eDlto^{h_e} X \mDrto^{h_m} N) \rto (M' \eDlto^{h'_e} X' \mDrto^{h'_m} N)\]
  is a diagram $M \erto^j M_1 \mrto^i M'$ such that there exists a map
  $\tilde h_m:X \mrto X'$ such that 
  the triangle on the left commutes and the square on the right
  \[
    \begin{inline-diagram}
      { X && X' \\
        & N \\};
      \mto{1-1}{1-3}^{\tilde h_m} \mto{1-1}{2-2}_{h_m} \mto{1-3}{2-2}^{h_m'}
    \end{inline-diagram}
    \qquad
    \csq{X}{X'}{M_1}{M'}{\tilde h_m}{jh_e}{h_e'}{i}
  \]
  is a pseudo-commutative square.  Composition works via composition in $Q\C$; using
  the following diagram we see that it is well-defined:
  \begin{diagram}
    { X & X' & X'' \\
      M_1 & M' \\
      \dotp & M_2 & M'' \\};
    \mto{1-1}{1-2}^{\tilde h_m} \mto{1-2}{1-3}^{\tilde h_m'}
    \mto{2-1}{2-2}^{i} \mto{3-1}{3-2} \mto{3-2}{3-3}^{i'}
    \eto{1-1}{2-1}_{jh_e} \eto{1-2}{2-2}^{h_e'} \eto{1-3}{3-3}^{h''_e}
    \eto{2-1}{3-1} \eto{2-2}{3-2}^{j'} 
    \comm{1-1}{2-2} \comm{1-2}{3-3} \dist{2-1}{3-2}
  \end{diagram}

  The functor $k_N: \sH_N \rto Q\A$ takes $M \eDlto^{h_e} X \mDrto N$ to
  $X^{k/h_e}$. A morphism is taken to the representation
  \[X^{k/h_e} \erto X^{k/jh_e} \mrto^{\tilde h_m} X',\]
  where the first map is obtained by applying $c^{-1}$.
\end{definition}

\begin{definition}
  Let $(N,\phi)$ be an object of $\I^m_V$ (Definition~\ref{def:IV}).  We define
  $P_{(N,\phi)}: \sH_N \rto \F_V$ by letting it take every object
  $M \eDlto X \mDrto N$ to the composition
  \[ V \inlineArrow{smor, ->}^{\phi^{-1}} sN \mDlto sX \eDrto sM\] in
  $\F_V\subseteq Qs_{V/}$.  As both $\sH_N$ and $\F_V$ have as morphisms
  morphisms of $Q\C$, the functor is defined to take a morphism to the morphism
  represented by the same data.
\end{definition}

\begin{lemma}
  $P_{(N,\phi)}$ is a well-defined functor.
\end{lemma}

\begin{proof}
  Checking that $P_{(N,\phi)}$ is well-defined on objects is straightforward
  from the definition.  Suppose that we are given a morphism in $\sH_N$ as
  defined in Definition~\ref{def:HN}.  We must show that this produces a
  well-defined morphism in $\F_V$; from the definition the produced morphism in
  $Qs_{V/}$ is an isomorphism, so it suffices to show that a morphism in $\sH_N$
  gives a well-defined morphism in $Qs_{V/}$.     For this to be true it suffices
  to check that the morphisms represented by
  \[N \mDlto X \eDrto M \eDrto M_1 \mDrto M'\]
  and
  \[N \mDlto X' \eDrto M'\] are equivalent in $Q(\C\bs\A)$.  This is true
  because the are equivalent isomorphisms inside the m-morphisms of $\C\bs\A$
  via the following diagram:
  \begin{diagram}
    { N & X & M_1 & M' \\
      & X' & M' \\};
    \mDto{1-2}{1-1} \eDto{1-2}{1-3} \mDto{1-3}{1-4}
    \mDto{2-2}{1-1} \mDto{1-2}{2-2} \mDto{1-3}{2-3} \eDto{2-2}{2-3}
    \comm{1-2}{2-3}
    \eq{2-3}{1-4}
  \end{diagram}
  where the marked square is pseudo-commutative from the definition of a morphism in
  $\sH_N$.  
  That $P_{(N,\phi)}$ respects composition follows directly from the definition,
  since composition in both $Qs_{V/}$ and $\sH_N$ is defined using composition
  in $Q\C$.
\end{proof}

We begin our analysis by showing that (\ref{diag:loc}) commutes up to homotopy.

\begin{lemma}
  In (\ref{diag:loc}) the composition around the top and the composition around
  the bottom are homotopic.
\end{lemma}

\begin{proof} 
  Consider an object $M \eDlto^{h_e} X \mDrto N$ in $\sH_N$.  Under the composition
  around the top it is mapped to
  \[\initial \mrto V \inlineArrow{smor,->}^{\phi^{-1}} sN \mDlto sX \eDrto sM;\]
  this is equivalent to the representation
  \[\initial \mrto sM.\] Around the bottom this is mapped to
  $\initial \mrto X^{k/h_e}$.  There is a natural map $h_e^k: X^{k/h_e} \mrto M$
  which induces a morphism between these in $Qs_{\initial/}$, so we just need to
  check that this gives a natural transformation.  To see that this
  transformation is natural, suppose that we are given a morphism
  \[\makeshort{(M \eDlto^{h_e} X \mDrto^{h_m} N)} \rto \makeshort{(M'
      \eDlto^{h_e'} X' \mDrto^{h'_m} N)}\] represented by
  $M \erto^j M_1 \mrto^i M'$.  Consider the following diagram in $\C$:
  \begin{diagram}
    { X^{k/h_e} & X^{k/jh_e} & (X')^{k/h_e'} \\
      M & M_1 & M' \\};
    \eto{1-1}{1-2} \eto{2-1}{2-2} \mto{1-2}{1-3} \mto{2-2}{2-3}
    \mto{1-1}{2-1}_{h_e^k} \mto{1-2}{2-2}^{(jh_e)^k} \mto{1-3}{2-3}^{{h'}_e^k}
    \dist{1-1}{2-2}
  \end{diagram}
  The left-hand square exists and is distinguished by the definition of $k$.
  The right-hand square exists and commutes by the condition on morphisms in
  $\sH_N$; this is exactly $k$ applied to the pseudo-commutative square.  After
  applying $s$ to the diagram and considering the outer corners as objects under
  $\initial$, we see that this diagram exactly corresponds to a naturality
  square for functors $\sH_N \rto Qs_{\initial/}$, as desired.
\end{proof}

It remains to show that $k_N$ and $P_{(N,\phi)}$ are homotopy equivalences.  We
begin with $k_N$; however, before we can prove that $k_N$ is a homotopy
equivalence we must develop some theory.

\begin{definition}
  Let $\J_N$ be a skeleton of the full subcategory of $\M_{/N}$ containing those
  morphisms $A \mrto N$ such that $A^c \in \A$.  The category $\J_N$ has a
  terminal object: $1_N$.
\end{definition}

\begin{definition}
  Let $\sH'_N$ be the full subcategory of $\sH_N$ containing those objects where
  $h_m$ is an isomorphism; in particular, each object in $\sH_N'$ can be
  uniquely represented by an e-morphism $N \erto X$.  For any m-morphism
  $i: M \mrto N$ we define the functor $\rho_i:\sH'_N \rto \sH'_M$ by sending
  the e-morphism $N \erto^f X$ to the e-morphism $M \erto \tilde X$, where
  $M \erto \tilde X$ is determined by the following distinguished square:
  \[\dsq MN{\tilde X}Xi{}f{}.\]
  Given a morphism represented by $X \erto X_1 \mrto X'$ in $\sH'_N$, this is
  mapped to the morphism represented by $\tilde X \erto \tilde X_1 \mrto \tilde
  X'$, where $\tilde X_1$ is defined by the distinguished square
  \[\dsq {\tilde X}{X}{\tilde X_1}{X_1}{}{}{}{}.\]
\end{definition}

\begin{lemma} \label{lem:k'andrho}
  Let $i:(J \mDrto N) \rto (I\mDrto N)$ be a morphism in $\J_N$.  Then the
  diagram
  \begin{diagram}
    { \sH'_I & & \sH'_J \\
      & Q\A \\};
    \to{1-1}{1-3}^{\rho_i} \to{1-1}{2-2}_{k'_I} \to{1-3}{2-2}^{k'_J}
  \end{diagram}
  commutes up to natural isomorphism.
\end{lemma}

\begin{proof}
  Consider the object $I \eDrto^h M$ in $\sH'_I$.  Its image under $k'_I$ is
  $I^{k/h}$.  For the other composition, we consider the distinguished square
  \[\dsq{J}{I}{M'}{M}{i}{h'}{h}{}.\]
  $h$ is mapped to $h'$, and then to $J^{k/h'}$.  Since $I^{k/h}$ and $J^{k/h'}$
  are the kernels in a natural distinguished square, they are naturally
  isomorphic, as desired.
\end{proof}

Consider the functor $F:\sH_N \rto \J_N$ defined by sending each class
$[M \eDlto X \mDrto N]$ to $X \mDrto N$, assuming that this representative is
chosen so that this morphism is in $\J_N$.  Note that for each class this
representative is unique.

\begin{lemma} \label{lem:Ffiber}
  $\sH_N$ is fibered over $\J_N$.
\end{lemma}

\begin{proof}
  For any $i:I\mrto N\in \J_N$, $F^{-1}(i)$, the fiber over $i$, is
  isomorphic to $\sH'_I$. The category $F_{i/}$ has as its objects the solid
  part of the diagram
  \begin{diagram}
    { M' & I & N \\
      M & X \\};
    \mto{1-2}{1-3}^i \mto{1-2}{2-2} \mto{2-2}{1-3}
    \eto{2-2}{2-1}
    \diagArrow{mmor,densely dashed}{1-1}{2-1}
    \diagArrow{emor,densely dashed}{1-2}{1-1}
    \dist{1-1}{2-2}
  \end{diagram}  
  The functor taking such a diagram to $I \erto M'$ is the right adjoint to the
  inclusion $\sH'_I = F^{-1}(i) \rto F_{i/}$.

  Thus $\sH_N$ is prefibered over $\J_N$.  To check that it is fibered it
  suffices to check that this right adjoint is compatible with composition in
  the following sense.  For any $j:(I \mrto^i N) \rto (I' \mrto^{i'} N)$ in
  $\J_N$ we get an induced functor $j^*:F^{-1}(i') \rto F^{-1}(i)$ defined by
  the composition
  \[
    \left(\begin{inline-diagram}
      { M & I' & N \\};
      \eto{1-2}{1-1} \mto{1-2}{1-3}^{i'} 
    \end{inline-diagram}\right)
    \rgoesto
    \left(\begin{inline-diagram}
      { M & I & N \\
        M' & I' &  \\};
      \diagArrow{densely dashed,emor}{1-2}{1-1}
      \diagArrow{densely dashed,mmor}{1-1}{2-1}
      \dist{1-1}{2-2}
      \mto{1-2}{2-2}_j \eto{2-2}{2-1}
      \mto{1-2}{1-3}^i \mto{2-2}{1-3}_{i'}
    \end{inline-diagram}\right)
    \rgoesto
    \left(\begin{inline-diagram}
      { M' & I & N \\};
      \eto{1-2}{1-1} \mto{1-2}{1-3}^i
    \end{inline-diagram}\right).\]
  We must show that for any composable $j$ and $k$, $(kj)^*$ is naturally
  isomorphic to $j^*k^*$.  This is true because completing a formal composition
  to a distinguished square is unique up to unique isomorphism.  As both
  $j^*k^*$ and $(kj)^*$ are obtained by completing a formal composition
  \[M \elto I'' \mlto^k I' \mlto^j I\]
  to a distinguished square, they are naturally isomorphic.
\end{proof}

We are now ready to prove that $k_N$ is a homotopy equivalence.

\begin{lemma} \label{lem:kN}
  $k_N$ is a homotopy equivalence.
\end{lemma}

\begin{proof}
  We begin by checking that $k'_N \defeq k_N|_{\sH'_N}$ is a homotopy
  equivalence.  Let $T$ be an object in $Q\A$; it suffices to check that
  $k'_N/T$ is contractible for all $T$.  An object of $k'_N/T$ is a triple
  $(M,h_e:N\erto M,u:N^k\rto T)$ with $u\in Q\A$.  Let $\C'$ be the full
  subcategory of $k'_N/T$ consisting of those morphisms $u$ which can be
  represented purely by an e-morphism.

  Represent $u$ as $X^k \mrto^i Y \erto^j T$, and consider the following
  diagram:
  \begin{diagram}[5em]
    { N^k & M & N \\
      Y & M\star_{N^k} Y & N \\
      T \\}; \mto{1-1}{1-2}^{h_e^k} \eto{1-3}{1-2}_{h_e} \mto{1-1}{2-1}^i
    \mto{1-2}{2-2} \mto{2-1}{2-2}^{(h_e^k)'} \eto{2-3}{2-2}_{k^{-1}((h_e^k)')}
    \eq{1-3}{2-3} \comm{1-2}{2-3} \eto{2-1}{3-1}^j
    \diagArrow{densely dashed, bend right,->}{1-1}{3-1}_u
  \end{diagram}
  Here, the upper-left square is produced by condition (PP).  We claim that the
  map taking $(M, h_e, u)$ to $(M\star_{N^k} Y, h_e', j)$ is a functor which
  produces a retraction from $k'_N/T$ to $\C'$.  To check that this is
  functorial, consider a morphism in $k'_N/T$.  This is represented by a diagram
  \begin{diagram}
    { T &Y &N^{k/h_e} & M &  \\
      & Y'& N^{k/jh_e} & M_1 & N \\
      & &N^{k/h_e'} &  M' & N\\};
    \mto{1-3}{1-4} \eto{2-5}{1-4}_{h_e}
    \eto{3-5}{3-4}_{h'_e} \mto{2-5}{3-5}^= \eto{2-5}{2-4}
    \comm{2-4}{3-5}
    \eto{1-4}{2-4} \mto{2-4}{3-4}
    \eto{1-3}{2-3} \mto{2-3}{2-4} \dist{1-3}{2-4}
    \mto{2-3}{3-3} \mto{3-3}{3-4}
    \mto{1-3}{1-2}_i \eto{1-2}{1-1}_j
    \mto{3-3}{2-2}^{i'} \eto{2-2}{1-1}^{j'}
    \mto{2-3}{2-2} \eto{1-2}{2-2} \dist{1-2}{2-3}
  \end{diagram}
  where the morphism is considered to go from the object represented by the
  diagram around the top to the object represented by the diagram around the
  bottom.  This diagram produces a map
  $M\star_{N^{k/h_e}} Y \erto M_1 \star_{N^{k/jh_e}} Y' \mrto
  M'\star_{N^{k/h_e'}} Y'$ by the functoriality conditions in (PP).  This is
  compatible with composition by Lemma~\ref{lem:(B)} and the definition of
  morphism composition in $Q\A$.  This functor also comes with a natural
  transformation from the identity produced by the map $M \mrto M\star_{N^k} Y$.
  Thus $k'_N/T$ is homotopy equivalent to $\C'$.  The category $\C'$ has an initial
  object $(N, 1_N, \initial \erto T)$, so it is contractible.  Thus $k'_N/T$ is
  contractible for all $T$, and so $k'_N$ is a homotopy equivalence.


  

  We have now shown that $k'_N$ is a homotopy equivalence.  By 2-of-3, in order
  to show that $k_N$ is a homotopy equivalence it suffices to check that the
  inclusion $\sH'_N \rto \sH_N$ is a homotopy equivalence.

  Since $k_n'$ is a homotopy equivalence, by Lemma~\ref{lem:k'andrho} we see
  that $\rho_i$ is a homotopy equivalence for all $i\in \J_N$.  Thus, since
  $\sH_N$ is fibered over $\J_N$, by \cite[Theorem B, Cor]{quillen}, for all
  $I \mrto N$, $\sH'_I$ is homotopy equivalent to the homotopy fiber of $F$.
  However, since $\J_N$ is contractible it follows that the inclusion
  $\sH'_I \rto \sH_N$ is a homotopy equivalence.  In particular, taking the
  m-morphism to be the identity on $N$ gives the desired result.
\end{proof}

We now turn our attention to $P_{(N,\phi)}$.

We will need two different proofs for this functor, depending on whether $\A$ is
m-negligible or m-well-represented in $\C$.

\begin{lemma}
  If $\A$ is m-negligible in $\C$ then $P_{(N,\phi)}$ is a homotopy equivalence.
\end{lemma}

\begin{proof}
  We prove this using Theorem A.  An object of $\F_V$ is an isomorphism
  $V \srto^\psi sA$.  We will show that $(P_{(N,\phi)})_{A/}$ is contractible.
  We can fix representatives for $\phi$ and $\psi$ such that an object of
  $(P_{(N,\phi)})_{/A}$ is represented by a diagram
  \begin{diagram-numbered}{diag:obj}
    { V & V' & Z & N' & N\\
      && A' && X\\
      && A & M' & M \\};
    \eDto{1-2}{1-1} \mDto{1-3}{1-2} \mDto{1-3}{1-4} \eDto{1-4}{1-5}
    \mDto{1-3}{2-3} \eDto{2-3}{3-3}
    \mDto{2-5}{1-5} \eDto{2-5}{3-5}
    \mto{3-5}{3-4} \eto{3-4}{3-3}
    \node (m-100-100) at (m-1-1) {\phantom{$V$}};
    \node (m-99-99) at (m-3-5) {\phantom{$V$}};
    \diagArrow{densely dashed, ->, bend left}{100-100}{1-5}^{\phi^{-1}}
    \diagArrow{densely dashed, ->, bend left}{1-1}{3-3}_\psi
    \diagArrow{densely dashed,->,bend left}{1-5}{3-5}
    \diagArrow{densely dashed,->,bend left}{99-99}{3-3}
  \end{diagram-numbered}
  where the dashed arrows commute inside $Q(\C\bs\A)$.  $V', Z, N', A'$ are all
  fixed by our choice of representatives; the only part of the diagram that is
  allowed to change are the bottom and rightmost rows.  A representative of an
  object is well-defined up to unique isomorphism, since both the right-hand
  column (an object in $\sH_N$) and the bottom row are well-defined up to unique
  isomorphism. The maps $M \mrto M'$ and $M' \erto A$ must also be in $\M_\A$
  and $\E_\A$, respectively, since $\M_\A$ and $\E_\A$ are closed under 2-of-3
  by Lemma~\ref{lem:1of3}. (This follows by computing a representative of the
  composition and noting that since its components are in $\M_\A$
  (resp. $\E_\A$) the two maps across the bottom are.) 

  A morphism $(M/A) \rto (M'/A)$ of $(P_{(N,\phi)})_{/A}$ is a
  diagram
  \begin{diagram}
    { V & V' & Z & N' & N && N \\
      && A' && X &&  \hat X\\
      && A & M' & M & M_1 & \hat M\\};
    \eDto{1-2}{1-1} \mDto{1-3}{1-2} \mDto{1-3}{1-4} \eDto{1-4}{1-5}
    \mDto{1-3}{2-3} \eDto{2-3}{3-3}
    \mDto{2-5}{1-5} \eDto{2-5}{3-5}
    \mto{3-5}{3-4} \eto{3-4}{3-3}
    \node (m-100-100) at (m-1-1) {\phantom{$V$}};
    \node (m-99-99) at (m-3-5) {\phantom{$V$}};
    \eq{1-5}{1-7}
    \mDto{2-7}{1-7} \eDto{2-7}{3-7}
    \eDto{3-7}{3-6} 
    \mDto{3-6}{3-5} \mDto{2-7}{2-5}
  \end{diagram}
  where the morphism $\hat M \rto A$ in $Q\C$ is given by the composition across the
  bottom.  

  Let $\D$ be the full subcategory of $(P_{(N,\phi)})_{/A}$ of those objects
  which can be represented by a diagram where the morphism $X \eDrto M$ is the
  identity.  Given any object represented by (\ref{diag:obj}) there is a
  well-defined morphism given by
  \begin{diagram}
    { V & V' & Z & N' & N && N \\
      && A' && X &&  X\\
      && A & M' & M & M & X \\};
    \eDto{1-2}{1-1} \mDto{1-3}{1-2} \mDto{1-3}{1-4} \eDto{1-4}{1-5}
    \mDto{1-3}{2-3} \eDto{2-3}{3-3}
    \mDto{2-5}{1-5} \eDto{2-5}{3-5}
    \mDto{3-5}{3-4} \eDto{3-4}{3-3}
    \node (m-100-100) at (m-1-1) {\phantom{$V$}};
    \node (m-99-99) at (m-3-5) {\phantom{$V$}};
    \eq{1-5}{1-7}
    \mDto{2-7}{1-7} \eq{2-7}{3-7}
    \eDto{3-7}{3-6} 
    \eq{3-6}{3-5} \eq{2-7}{2-5}
  \end{diagram}
  which is natural in our object (since the choice of $X$ is unique up to unique
  isomorphism).  This show that $\D$ is a retractive subcategory of
  $(P_{(N,\phi)})_{/A}$, and is thus homotopy equivalent to it.  

  A morphism inside $\D$ is represented by a diagram
  \begin{diagram}
    { V & V' & Z & N' & N & N \\
      & A' & A & M' & M & \hat M \\};
    \eDto{1-2}{1-1} \mDto{1-3}{1-2} \mDto{1-3}{1-4} \eDto{1-4}{1-5}
    \eq{1-5}{1-6}
    \mDto{1-3}{2-2} \eDto{2-2}{2-3} \eDto{2-4}{2-3} \mDto{2-5}{2-4}
    \mDto{2-6}{2-5}
    \mDto{2-6}{1-6} \mDto{2-5}{1-5}
  \end{diagram}
  The only important information here is the lower-right-hand side.  Thus we
  will think of morphisms in $\D$ as represented by diagrams
  \[N \mDlto M \mrto M' \erto A\] which are equivalent inside $\C\bs\A$.  Since
  all morphisms in $\M$ are monic, such a morphism (if it exists) is unique;
  thus $\D$ is a preorder.  To show that $\D$ is contractible we will show that
  it is nonempty and cofiltered.

  Given two objects
  \[N \mDlto M \mrto M' \erto A \qqand N \mDlto \tilde M \mrto \tilde M' \erto
    A\]
  we know that they are equivalent inside $\C\bs\A$ if there exists a diagram $X
  \eDrto Y \mDrto N$ such that precomposition by this diagram sends these to
  diagrams which are equivalent in $\C$.  However, since $\A$ is m-negligible in
  $\C$ we see that such a diagram exists if and only if such a diagram exists
  with the e-component equal to the identity.  Picking such a morphism $Y \mDrto
  N$ we see that the object represented by
  \[N \mDlto Y\times_N M \mrto M' \erto A\]
  maps to both of these objects.  Thus $\D$ is cofiltered.

  To see that $\D$ is nonempty, consdier the diagram
  \[Z \mDrto N' \eDrto N\]
  given by the chosen representative for $\phi$.  Since $\A$ is m-negligible in
  $\C$ there exists an m-morphism $M \mDrto N$ such that $M\oslash_N N' \cong M$
  and $M \mDrto N'$ factors through $Z \mDrto N'$.  Then the diagram
  \[N \mDlto M \mDrto A' \eDrto A\]
  gives a well-defined object of $\D$.  Thus $\D$ is nonempty and cofiltered,
  and therefore contractible.

\end{proof}

If $\A$ is m-negligible in $\C$ we are now done.  Thus we can now assume that
$\A$ is m-well-represented in $\C$.

Consider a diagram $N \eDlto^{g_e} X \mDrto^{g_m} N'$ which we denote $g$.  We
define the functor $g_*:\sH_N \rto \sH_{N'}$ by
\[
  \begin{inline-diagram}
    { M & Y & N \\};
    \eDto{1-2}{1-1} \mDto{1-2}{1-3}
  \end{inline-diagram}
  \rgoesto
  \begin{inline-diagram}
    {  & Y\oslash_N X & X & N' \\
      M & Y & N \\};
    \eDto{2-2}{2-1} \mDto{2-2}{2-3} \eDto{1-3}{2-3}^{g_e} \mDto{1-3}{1-4}^{g_m}
    \eDto{1-2}{2-2}
    \mDto{1-2}{1-3}
    \comm{1-2}{2-3}
  \end{inline-diagram}
  \rgoesto
  \begin{inline-diagram}
    { M & Y\oslash_N X & N' \\};
    \eDto{1-2}{1-1} \mDto{1-2}{1-3}
  \end{inline-diagram}.
\]
This is functorial because pseudo-commutative squares compose.

\begin{lemma}
  There is a natural transformation $k_N \Rto k_{N'}g_*$.
\end{lemma}

\begin{proof}
  We have
  \[\makeshort{k_N(M \eDlto^{h_e} Y \mDrto N)} = Y^{k/M}.\]
  On the other hand,
  \[\makeshort{k_{N'}g_*(M \eDlto^{h_e} Y \mDrto N)} = (Y\oslash_N X)^{k/M}.\]
  The map $Y\oslash_N X \erto M$ factors through $Y \erto M$, so (by
  Lemma~\ref{lem:opensquare}) there is a functorially induced map
  $Y^{k/M} \mrto (Y\oslash_N X)^{k/M}$.  This map gives the natural
  transformation.  To check that this is actually natural, consider a map
  $(M \eDlto Y \mDrto N) \rto (M' \eDlto Y' \mDrto N)$ represented by
  $M \eDrto M_1 \mDrto M'$.  We must show that the square
  \begin{diagram}
    {      Y^{k/M} & Y^{k/M_1} & (Y')^{k/M'} \\
      (Y\oslash_N X)^{k/M} & (Y\oslash_N X)^{k/M_1} & (Y'\oslash_N X)^{k/M'}
      \\};
    \mto{1-1}{2-1} \mto{1-3}{2-3}
    \eto{1-1}{1-2} \mto{1-2}{1-3}
    \eto{2-1}{2-2} \mto{2-2}{2-3}
  \end{diagram}
  commutes in $Q\A$.  To show this it suffices to show that there exists a map
  $Y^{k/M_1} \mrto (Y\oslash_N X)^{k/M_1}$ such that the left-hand square is
  distinguished and the right-hand square commutes.  The map exists and makes
  the right-hand square commute by Lemma~\ref{lem:opensquare}.  To check that
  the left-hand square is distinguished it suffices to check that given any
  diagram
  \[A \erto B \erto C \erto D\]
  the square
  \begin{diagram}
    { B^{k/C} & A^{k/C} \\
      B^{k/D} & A^{k/D} \\};
    \mto{1-1}{1-2} \mto{2-1}{2-2}
    \eto{1-1}{2-1} \eto{1-2}{2-2}
  \end{diagram}
  is distinguished.  This follows directly from the definition of $c$ and $k$.
\end{proof}

Since $k_N$ and $k_{N'}$ are both homotopy equivalences, we get the following
corollary: 

\begin{corollary}
  $g_*$ is a homotopy equivalence.
\end{corollary}

Consider the functor $H:\I^m_V \rto \Cat$ sending $(N,\phi)$ to $\sH_N$ and
$g:(N,\phi) \rto (N,\phi')$ to $g_*$.  

\begin{lemma} \label{lem:Hiso}
  There is an isomorphism of categories
  \[\tilde H: \colim_{\I^m_V} H \rto \F_V\]
  induced by $P_{(N,\phi)}: \sH_N \rto \F_V$.
\end{lemma}

\begin{proof}
  We first check that $\tilde H$ is well-defined.
  To prove this it suffices to check that 
  for $g:(N,\phi) \rto (N',\phi')$, 
  \[P_{(N',\phi')} g_* = P_{(N,\phi)}.\] Since morphisms in $\sH_N$ are defined
  to be morphisms in $Q\C$ satisfying extra conditions, and since both
  $P_{(N,\phi)}$ and $g_*$ do not change any of the representation data in the
  morphism, if the two sides agree on objects they must also agree on morphisms.
  $P_{(N,\phi)}$ maps an object $(M \eDlto X \mDrto N)$ to the composition
  \[V \srto^{\phi^{-1}} sN \mDlto sX \eDrto sM,\]
  while $P_{(N',\phi')}g_*$ maps it to the composition
  \[V \srto^{{\phi'}^{-1}} sN' \mDlto sY \eDrto sN \mDlto sX \eDrto sM.\]
  However, since $\phi' s(g) = \phi$, these two compositions represent
  equivalent diagrams (since after being considered inside $\C\bs\A$, all $g_*$
  does is compose with $g$) and thus the left and right sides agree on objects.
  Therefore the functors $P_{(N,\phi)}$ produce a valid cone under $H$ and
  $\tilde H$ is well-defined.

  It now remains to show that it is, in fact, an isomorphism of categories.

  First we show that $\tilde H$ is surjective on objects; in other words, that
  for every $(M,u:V \rto^{\cong} sM)$ in $\F_V$ there exists an $(N,\phi)$ and
  an object $(M',h)$ in $\sH_N$ such that $P_{(N,\phi)}(M',h) = (M,u)$.  To do
  this, let $(N,\phi) = (M, u^{-1})$ and let $(M',h) = (M,M \req M \req M)$..
  Thus $\tilde H$ is surjective on objects.

  Now consider injectivity.  Since $\I^m_V$ is filtered, it suffices to check
  that each individual $P_{(N,\phi)}$ is injective on objects.  Suppose that
  \[P_{(N,\phi)}(M,h) = P_{(N,\phi)}(M',h').\] We must show that there exists
  $g:(N,\phi) \rto (N',\phi')$ in $\I_V^m$ such that $g_*(M,h) = g_*(M',h')$.
  Note, that by definition in order for this to hold we must have $M = M'$ and
  $s(h) = s(h')$.  The fact that such a $g$ exists is implied by condition (E);
  in fact, this $g$ will be represented by a morphism where the m-component is
  the identity.  Thus $\tilde H$ is injective on objects.

  We now consider morphisms.  As before, we consider surjectivity first.
  Consider a morphism $g:(M,u) \rto (M',u')$ in $\F_V$.  This is given by a
  morphism $g:M \rto M'$ in $Q\C$ such that $s(g) u = u'$ in $Q(\C\bs\A)$.  Since
  both $u$ and $u'$ are isomorphisms, $s(g)$ must be as well; thus it is
  represented by a diagram $M \eDrto X \mDrto M'$.  Consider the distinguished
  square
  \begin{diagram}
    { M & X' \\
      X & M \\};
    \mDto{1-1}{1-2}^{h_m} \mDto{2-1}{2-2}^{g_m}
    \eDto{1-1}{2-1}_{g_e} \eDto{1-2}{2-2}^{h_e} \dist{1-1}{2-2}
  \end{diagram}
  where the composition around the bottom is given by the components of $g$.
  Since all distinguished squares are pseudo-commutative, this defines a morphism
  \[\makeshort{(M, M \mDrto^h_m X')} \rto^f \makeshort{(M', M' \eDlto^{h_e} X')}\]
  in $\sH_{X'}$.  Note that $s(g_e) u = s(h_e^{-1}) u'$.  Thus
  $P_{(X,s(g_e)u)}(f) = g$, as desired.

  Now consider injectivity.  As before, it suffices to consider a single
  $P_{(N,\phi)}$ and show that it is faithful.  Suppose that
  $P_{(N,\phi)}(g) = P_{(N,\phi)}(g')$.  By definition, 
  \[g,g': \makeshort{(M \eDlto X \mDrto N)} \rto \makeshort{(M' \eDlto X' \mDrto
      N)}\] are given by morphisms $\tilde g,\tilde g': M \rto M'$ in $Q\C$
  satisfying the diagram in Definition~\ref{def:HN}.  For
  $P_{(N,\phi)}(g) = P_{(N,\phi)}(g')$ we must have $\tilde g = \tilde g'$;
  however, in this case we must have $g$ and $g'$ equal as well.  Thus
  $\tilde H$ is injective on morphisms, and we are done.
\end{proof}

We are now ready to finish:

\begin{proposition}
  If $\A$ is m-well-represented in $\C$ then $P_{(N,\phi)}$ is a homotopy
  equivalence.
\end{proposition}

\begin{proof}
  \cite[Proposition 3, Corollary 1]{quillen} states the following: given any
  filtered category $\C$ and a functor $F: \C \rto \Cat$ such that for all $f:A
  \rto B \in \C$, $F(f)$ is a homotopy equivalence.  Then the induced map $F(A)
  \rto \colim_\C F$ is a homotopy equivalence for all $A\in \C$.

  Applying this to the functor $H$, we get that the map $H(N,\phi) \rto
  \colim_{\I^m_V} H \cong \F_V$ is a homotopy equivalence for all $(N,\phi) \in
  \I^m_V$.  By definition this is exactly $P_{(N,\phi)}:\sH_N \rto \F_V$, and we
  are done.
\end{proof}

\appendix
\section{Checking that $\C\bs\A$ is a CGW-category} \label{app:C-A}

In this appendix we check as much as possible that the definition of $\C\bs\A$
gives a well-defined CGW-category.  More concretely, it is necessary to check
that the m-morphisms and e-morphisms give well-defined categories, that the
distinguished squares compose correctly, that $\phi$ exists, that $c$ and $k$
are equivalences of categories, and that axioms (Z), (I), (M), (K), and (A)
hold.

\begin{proposition} \label{prop:C-A-well-def}
  Let $\C$ be an ACGW-category, $\A$ a full ACGW-subcategory closed under
  subobjects, quotients, and extensions.  Then $\C\bs\A$ is a well-defined
  CGW-category assuming that the following condition holds:
  \begin{itemize}
  \item[(Ex)] The definitions of $c$ and $k$ in Definition~\ref{def:C-A} give
    equivalences of categories, in the sense that there exist equivalences of
    categories $k: \Ar_\square \E \rto \Ar_\triangle \M$ and $c: \Ar_\square \M
    \rto \Ar_\triangle \E$ which agree with the given definitions on objects.
  \end{itemize}
\end{proposition}

The rest of this appendix is a proof of this proposition.

As the definition of $\C\bs\A$ is symmetric with respect to e-morphisms and
m-morphisms it suffices to focus on proving only half of each statement; the
other half will follow by symmetry.

We first begin with a somewhat more explicit definition of the distinguished
squares in $\C\bs\A$.  These are generated by the following types of squares:
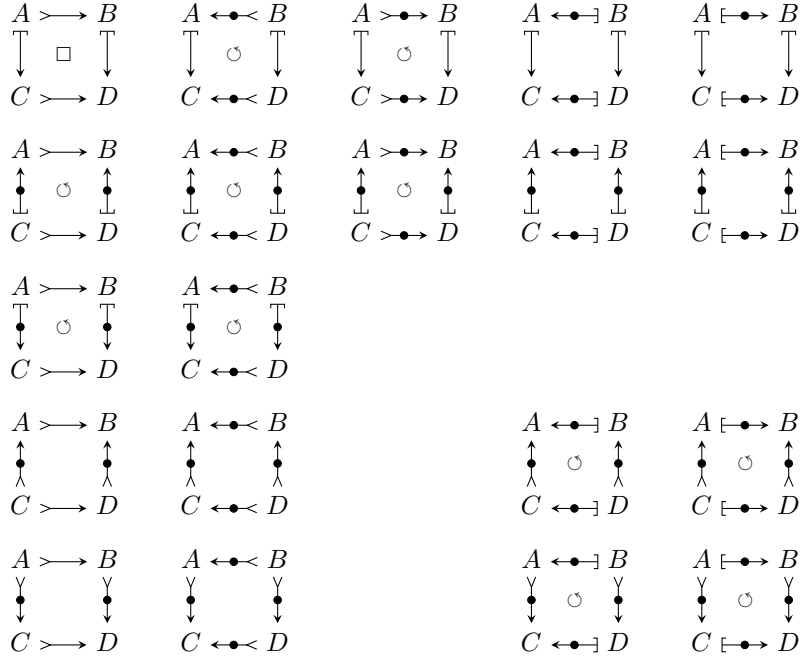
\begin{center}
  \begin{tabular}{ccccccc}
    $\dsq{A}{B}{C}{D}{}{}{}{}$
    & 
      \begin{inline-diagram}
        { A & B \\ C & D \\};
        \mDto{1-2}{1-1} \mDto{2-2}{2-1} \eto{1-1}{2-1} \eto{1-2}{2-2}
        \comm{1-1}{2-2}
        \end{inline-diagram}
    & 
      \begin{inline-diagram}
        { A & B \\ C & D \\};
        \mDto{1-1}{1-2} \mDto{2-1}{2-2} \eto{1-1}{2-1} \eto{1-2}{2-2}
        \comm{1-1}{2-2}
        \end{inline-diagram}
    &
      \begin{inline-diagram}
        { A & B \\ C & D \\};
        \eDto{1-2}{1-1} \eDto{2-2}{2-1} \eto{1-1}{2-1} \eto{1-2}{2-2}
        \end{inline-diagram}
    &
      \begin{inline-diagram}
        { A & B \\ C & D \\};
        \eDto{1-1}{1-2} \eDto{2-1}{2-2} \eto{1-1}{2-1} \eto{1-2}{2-2}
        \end{inline-diagram}
    \\
    \begin{inline-diagram}
      { A & B \\ C & D \\};
      \mto{1-1}{1-2} \mto{2-1}{2-2} 
      \eDto{2-1}{1-1} \eDto{2-2}{1-2}
      \comm{1-1}{2-2}
      \end{inline-diagram}
    &
      \begin{inline-diagram}
        { A & B \\ C & D \\};
        \mDto{1-2}{1-1} \mDto{2-2}{2-1} 
        \eDto{2-1}{1-1} \eDto{2-2}{1-2}
        \comm{1-1}{2-2}
        \end{inline-diagram}
    &
      \begin{inline-diagram}
        { A & B \\ C & D \\};
        \mDto{1-1}{1-2} \mDto{2-1}{2-2} 
        \eDto{2-1}{1-1} \eDto{2-2}{1-2}
        \comm{1-1}{2-2}
        \end{inline-diagram}
    &
      \begin{inline-diagram}
        { A & B \\ C & D \\};
        \eDto{1-2}{1-1} \eDto{2-2}{2-1} 
        \eDto{2-1}{1-1} \eDto{2-2}{1-2}
        \end{inline-diagram}
    &
      \begin{inline-diagram}
        { A & B \\ C & D \\};
        \eDto{1-1}{1-2} \eDto{2-1}{2-2} 
        \eDto{2-1}{1-1} \eDto{2-2}{1-2}
        \end{inline-diagram}
    \\
    \begin{inline-diagram}
      { A & B \\ C & D \\};
      \mto{1-1}{1-2} \mto{2-1}{2-2} 
      \eDto{1-1}{2-1} \eDto{1-2}{2-2} \comm{1-1}{2-2}
      \end{inline-diagram}
    &
      \begin{inline-diagram}
        { A & B \\ C & D \\};
        \mDto{1-2}{1-1} \mDto{2-2}{2-1} 
        \eDto{1-1}{2-1} \eDto{1-2}{2-2} \comm{1-1}{2-2}
        \end{inline-diagram}
    &
    &
    &
    \\
    \begin{inline-diagram}
      { A & B \\ C & D \\};
      \mto{1-1}{1-2} \mto{2-1}{2-2} 
      \mDto{2-1}{1-1} \mDto{2-2}{1-2}
      \end{inline-diagram}
    &
      \begin{inline-diagram}
        { A & B \\ C & D \\};
        \mDto{1-2}{1-1} \mDto{2-2}{2-1} 
        \mDto{2-1}{1-1} \mDto{2-2}{1-2}
        \end{inline-diagram}
    &
    &
      \begin{inline-diagram}
        { A & B \\ C & D \\};
        \eDto{1-2}{1-1} \eDto{2-2}{2-1} 
        \mDto{2-1}{1-1} \mDto{2-2}{1-2}
        \comm{1-1}{2-2}
        \end{inline-diagram}
    &
      \begin{inline-diagram}
        { A & B \\ C & D \\};
        \eDto{1-1}{1-2} \eDto{2-1}{2-2} 
        \mDto{2-1}{1-1} \mDto{2-2}{1-2}
        \comm{1-1}{2-2}
        \end{inline-diagram} 
    \\
    \begin{inline-diagram}
      { A & B \\ C & D \\};
      \mto{1-1}{1-2} \mto{2-1}{2-2} 
      \mDto{1-1}{2-1} \mDto{1-2}{2-2}
      \end{inline-diagram}
    &
      \begin{inline-diagram}
        { A & B \\ C & D \\};
        \mDto{1-2}{1-1} \mDto{2-2}{2-1} 
        \mDto{1-1}{2-1} \mDto{1-2}{2-2}
        \end{inline-diagram}
    &
    &
      \begin{inline-diagram}
        { A & B \\ C & D \\};
        \eDto{1-2}{1-1} \eDto{2-2}{2-1} 
        \mDto{1-1}{2-1} \mDto{1-2}{2-2}
        \comm{1-1}{2-2}
        \end{inline-diagram}
    &
      \begin{inline-diagram}
        { A & B \\ C & D \\};
        \eDto{1-1}{1-2} \eDto{2-1}{2-2} 
        \mDto{1-1}{2-1} \mDto{1-2}{2-2}
        \comm{1-1}{2-2}
        \end{inline-diagram}
  \end{tabular}
\end{center}

We now prove a series of lemmas about how different types of squares in $\C$
interact.  The common consequence of all of these lemmas is that the given
squares fit into a cube with opposite sides of the same ``type'' (be that
pseudo-commutative squares, distinguished squares, or simply squares that commute
inside $\E$ or $\M$).  We do not worry about which arrows have $c$ or $k$ in
$\A$; the properties of $\A$ ensure that whenever such an arrow is ``pulled
back'', the pullback also has $c$ or $k$ in $\A$.  

\begin{lemma} \label{lem:Ccube}
  Given two diagrams in $\C$
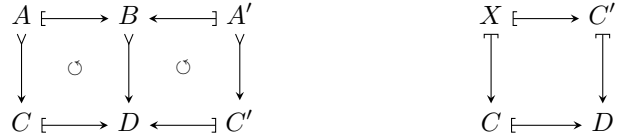
  \begin{diagram}
    { A & B & A' & \qquad & X & C'\\
      C & D & C' & & C & D \\};
    \mto{1-1}{2-1} \mto{1-2}{2-2} \mto{1-3}{2-3}
    \eto{1-5}{2-5} \eto{1-6}{2-6}
    \eto{1-1}{1-2} \eto{1-3}{1-2} \eto{2-1}{2-2} \eto{2-3}{2-2}
    \eto{1-5}{1-6} \eto{2-5}{2-6}
    \comm{1-1}{2-2} \comm{1-2}{2-3}
  \end{diagram}
  we can assemble these into a cube
  
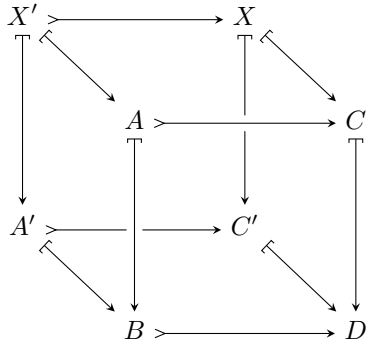
\begin{diagram}
    {X' & & X \\
      & A & & C \\
      A' & & C' \\
      & B & & D\\};
    \mto{1-1}{1-3} \eto{1-3}{2-4} \eto{2-4}{4-4}
    \mto{4-2}{4-4} \mto{3-1}{3-3}
    \eto{1-1}{2-2} \eto{3-1}{4-2} \eto{3-3}{4-4}
    \eto{1-1}{3-1} \eto{1-3}{3-3}
    \over{2-2}{2-4} \over{2-2}{4-2}
    \mto{2-2}{2-4} \eto{2-2}{4-2}
  \end{diagram}
  in which all faces with mixed morphisms are pseudo-commutative.   If $ABCD$ was
  originally distinguished, then $X'A'XC'$ will be, as well.
  
  An analogous statement with the roles of e-morphisms and m-morphisms swapped
  also holds.
\end{lemma}

\begin{proof}
  Apply $c$ to the left-hand diagram.  This turns both of the squares into
  pullback squares in $\E$ (by definition). We can then form the following
  diagram:
  
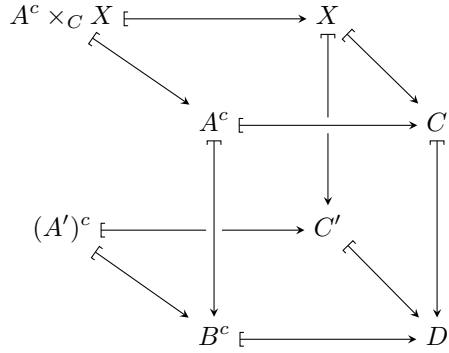
\begin{diagram}
    { A^c \times_C X &  & X \\
      & A^c & & C \\
      (A')^c & & C' \\
      & B^c & & D \\};
    \eto{1-3}{2-4} \eto{1-3}{3-3} \eto{2-4}{4-4} \eto{3-3}{4-4}
    \eto{3-1}{3-3} \eto{3-1}{4-2} \eto{4-2}{4-4}
    \over{2-2}{2-4} \over{2-2}{4-2} \eto{2-2}{2-4} \eto{2-2}{4-2}
    \eto{1-1}{1-3} \eto{1-1}{2-2}
  \end{diagram}
  To prove the main statement of the lemma it suffices to show that a morphism
  $A^c \times_C X \rto (A')^c$ exists and makes the back face into a pullback.
  To show the last stement it suffices to show that if $A^c \rto B^c$ is an
  isomorphism then this morphism is also an isomorphism.  This is a
  straightforward diagram chase using the fact that all solid faces in the above
  diagram are pullbacks and all morphisms in $\E$ are monic.
\end{proof}

As a corollary we can see that assembling distinguished squares and pullbacks
commutes:
\begin{corollary} \label{cor:comp}
  Suppose that we are given a diagram
  \[A \mrto B \erto C \elto D.\]
  The two diagrams
  \[\begin{inline-diagram}
    { A & B & B\times_C D\\
      X & C & D \\};
    \mto{1-1}{1-2} \mto{2-1}{2-2}
    \eto{1-1}{2-1} \eto{1-2}{2-2} \dist{1-1}{2-2}
    \eto{2-3}{2-2} \eto{1-3}{2-3} \eto{1-3}{1-2}
  \end{inline-diagram}
  \qquad\hbox{and}\qquad
  \begin{inline-diagram}
    { A \oslash_B (B\times_C D) & B\times_C D & D\\
      A & B & C \\};
    \eto{2-2}{2-3} \eto{1-3}{2-3}
    \mto{2-1}{2-2} \eto{1-2}{2-2} \eto{1-2}{1-3}
    \eto{1-1}{2-1} \mto{1-1}{1-2} \comm{1-1}{2-2}
  \end{inline-diagram}\]
  fit into a cube
  \begin{diagram}
    {A & & B \\
      & X & & C \\
      A\oslash_B (B\times_C D) & & B\times_C D \\
      & W & & D\\};
    \mto{1-1}{1-3} \eto{1-3}{2-4} \eto{4-4}{2-4}
    \mto{4-2}{4-4} \mto{3-1}{3-3}
    \eto{1-1}{2-2} \eto{3-1}{4-2} \eto{3-3}{4-4}
    \eto{3-1}{1-1} \eto{3-3}{1-3}
    \over{2-2}{2-4} \over{2-2}{4-2}
    \mto{2-2}{2-4} \eto{4-2}{2-2}
  \end{diagram}
  in which the top and bottom face are distinguished squares, the front and the
  back face are pseudo-commutative squares, and the right and left face are commutative
  in $\E$ with the right-hand face a pullback. 
\end{corollary}

We now prove a ``complement'' to Lemma~\ref{lem:(B)}: instead of assuming that a
commutative square in $\E$ is attached to the back of a pseudo-commutative square, we
assume that it is attached to the front:
\begin{lemma} \label{lem:lem**}
  Suppose that we are given a diagram
  \begin{diagram}
    { A & B & B' \\
      C & D & D' \\};
    \mto{1-1}{1-2} \mto{2-1}{2-2}
    \eto{1-1}{2-1} \eto{1-2}{2-2}
    \eto{1-2}{1-3} \eto{2-2}{2-3} \eto{1-3}{2-3}
    \comm{1-1}{2-2}
  \end{diagram}
  Then this diagram assembles into a cube
  
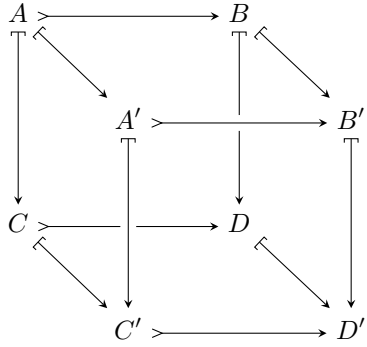
\begin{diagram}
    { A && B \\
      & A' && B'\\
      C && D \\
      & C' && D'\\};
    \mto{1-1}{1-3} \mto{3-1}{3-3} \eto{1-1}{3-1} \eto{1-3}{3-3}
    \over{2-2}{2-4} \mto{2-2}{2-4} \over{2-2}{4-2} \eto{2-2}{4-2}
    \mto{4-2}{4-4} \eto{2-4}{4-4}
    \eto{1-1}{2-2} \eto{1-3}{2-4}
    \eto{3-1}{4-2} \eto{3-3}{4-4}
  \end{diagram}
  where the front, back, and top faces are pseudo-commutative and the bottom face is
  distinguished.  If the right-hand square is a pullback then the top face will
  also be distinguished.

  The dual statement also holds. 
\end{lemma}

\begin{proof}
  Define $C'$ so that the bottom face of the cube is a distinguished square.
  Define $A' = C'\oslash_{D'} B'$.  By definition this produces a
  diagram where the front face is pseudo-commutative and the bottom face is
  distinguished.  It therefore suffices to check that there exists a morphism $A 
  \erto A'$ such that the left face commutes in $\E$ and the top face is
  pseudo-commutative. To prove this it suffices to check that there exists a morphism
  $A^{c/B} \erto B'\times_{D'} (C')^c$ such that in the diagram
  
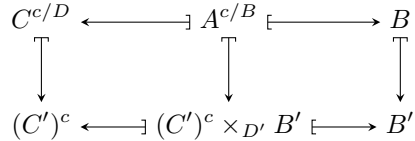
\begin{diagram}
    { C^{c/D} & A^{c/B} & B \\
      (C')^c & (C')^c\times_{D'} B' & B' \\};
    \eto{1-2}{1-1} \eto{1-2}{1-3}
    \eto{2-2}{2-1} \eto{2-2}{2-3}
    \eto{1-1}{2-1} \eto{1-2}{2-2} \eto{1-3}{2-3}
  \end{diagram}
  the left-hand square commutes and the right-hand square is a pullback.  This
  follows directly from the definitions.
\end{proof}

We are now ready to turn our attention to proving that $\C\bs\A$ is a CGW-category.

\noindent
\textbf{The m-morphisms form a well-defined category}

The m-morphisms in $\C\bs\A$ are defined to be equivalence classes of diagrams
\[ A \eDlto A' \mDlto X\eDrto B'\mrto B.\] The equivalence relation is generated
by the following types of diagrams (up to isomorphism), where the red diagram is
declared to be equivalent to the blue diagram:
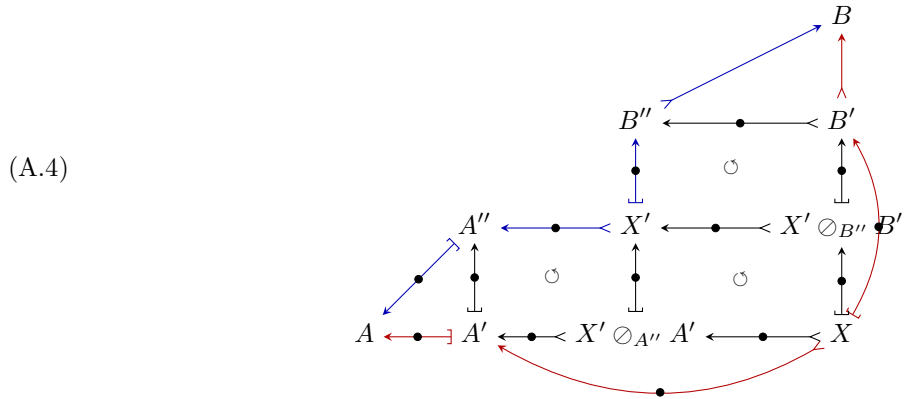
\begin{diagram-numbered}{diag:single-eq}
  { & & & B \\
    & & B'' & B' \\
    & A'' & X' & X' \oslash_{B''} B' \\
    A & A' & X'\oslash_{A''} A' & X \\};
  \node (m-100-100) at (m-4-4) {\phantom{X}};
  
  \diagArrow{emor,Dmor,red!70!black}{4-2}{4-1}
  \diagArrow{mmor,Dmor,red!70!black,bend left}{100-100}{4-2}
  \diagArrow{emor,Dmor,red!70!black,bend right}{100-100}{2-4}
  \diagArrow{mmor,red!70!black}{2-4}{1-4}
  \diagArrow{emor,Dmor,blue!70!black}{3-2}{4-1}
  \diagArrow{mmor,Dmor,blue!70!black}{3-3}{3-2}
  \diagArrow{emor,Dmor,blue!70!black}{3-3}{2-3}
  \diagArrow{mmor,blue!70!black}{2-3}{1-4}
  \eDto{4-2}{3-2} \eDto{4-3}{3-3}
  \mDto{4-3}{4-2} \mDto{4-4}{4-3}
  \eDto{4-4}{3-4} \mDto{3-4}{3-3}
  \eDto{3-4}{2-4} \mDto{2-4}{2-3}
  \comm{2-3}{3-4} \comm{3-2}{4-3} \comm{3-3}{4-4}
\end{diagram-numbered}
The relation defined between m-morphisms is a formal composition of
two such relations, one inverse to another.  Thus to show that the relation is
well-defined we must check that if we are given two such relations built on top
of one another, then either they compose to a single one, or that we can ``pull
back'' two such relations.

Let us consider the first such case.  Suppose that we are given two such
diagrams, one relating $A \eDlto A' \mDlto X \eDrto B' \mrto B$ to
$A \eDlto A'' \mDlto X' \eDrto B'' \mrto B$, and one relating
$A \eDlto A'' \mDlto X' \eDrto B'' \mrto B$ to
$A \eDlto A''' \mDlto X'' \eDrto B''' \mrto B$.    We can rearrange this data into the
following diagram, where the first formal composition is in red, the second is
in blue, and the third is in green:
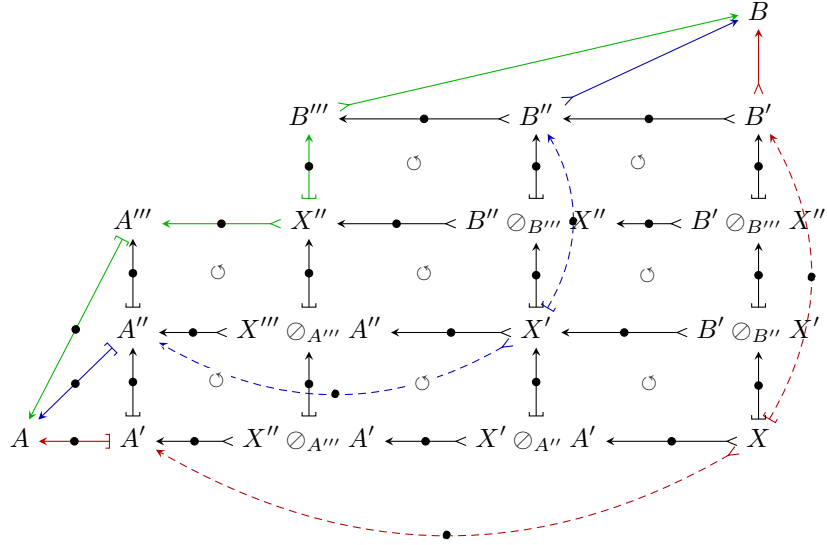
\begin{diagram}
  { & & & & B\\
    & & B''' & B'' & B' \\
    & A''' & X'' & B''\oslash_{B'''}X'' & B'\oslash_{B'''}X''\\
    & A'' & X'''\oslash_{A'''}A'' & X' & B'\oslash_{B''}X'\\
    A & A' & X''\oslash_{A'''}A' &  X'\oslash_{A''}A' & X\\};
  \node (m-100-100) at (m-5-5) {$\phantom{X}$};
  \diagArrow{red!70!black,emor,Dmor}{5-2}{5-1}
  \diagArrow{red!70!black,mmor,Dmor,bend left,densely dashed}{100-100}{5-2};
  \diagArrow{red!70!black,emor,Dmor,bend right,densely dashed}{100-100}{2-5};
  \diagArrow{red!70!black,mmor}{2-5}{1-5}

  \node (m-99-99) at (m-4-4) {$\phantom{X'}$};
  \diagArrow{blue!70!black,emor,Dmor}{4-2}{5-1}
  \diagArrow{blue!70!black,mmor,Dmor,bend left,densely dashed}{99-99}{4-2}
  \diagArrow{blue!70!black,emor,Dmor,bend right,densely dashed}{99-99}{2-4}
  \diagArrow{blue!70!black,mmor}{2-4}{1-5}

  \diagArrow{green!70!black,emor,Dmor}{3-2}{5-1}
  \diagArrow{green!70!black,mmor,Dmor}{3-3}{3-2}
  \diagArrow{green!70!black,emor,Dmor}{3-3}{2-3}
  \diagArrow{green!70!black,mmor}{2-3}{1-5}

  \eDto{4-2}{3-2} \eDto{5-2}{4-2} \eDto{4-3}{3-3} \eDto{5-3}{4-3}
  \mDto{4-3}{4-2} \mDto{5-3}{5-2}
  \eDto{3-4}{2-4} \eDto{4-4}{3-4} \eDto{5-4}{4-4}
  \eDto{5-5}{4-5} \eDto{4-5}{3-5} \eDto{3-5}{2-5}
  \mDto{2-4}{2-3} \mDto{2-5}{2-4}
  \mDto{3-5}{3-4} \mDto{3-4}{3-3}
  \mDto{4-5}{4-4} \mDto{4-4}{4-3}
  \mDto{5-5}{5-4} \mDto{5-4}{5-3}

  \comm{2-3}{3-4} \comm{2-4}{3-5}
  \comm{3-2}{4-3} \comm{3-3}{4-4} \comm{3-4}{4-5}
  \comm{4-2}{5-3} \comm{4-3}{5-4} \comm{4-4}{5-5}  
\end{diagram}
By regrouping the pseudo-commutative squares, we see that the red composition is
equivalent to the green composition, as desired.

To show the second case, consider the following diagram, which shows that red
and blue are both equivalent to green:
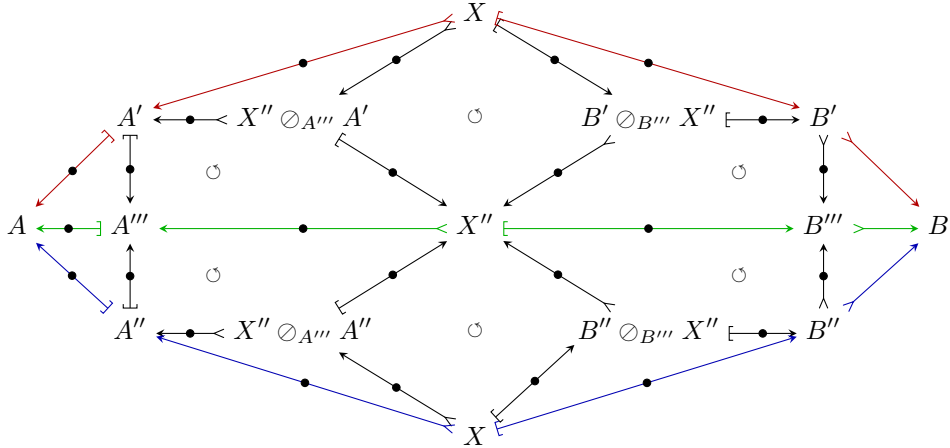
\begin{diagram}
  { & & & X \\
    & A' & X'' \oslash_{A'''} A' & & B' \oslash_{B'''} X'' & B'\\
    A & A''' & & X'' & & B''' & B \\
    & A'' & X'' \oslash_{A'''} A'' & & B'' \oslash_{B'''} X'' & B'' \\
    & & & X \\};
  \eDto{2-2}{3-2} \eDto{2-3}{3-4} \mDto{2-5}{3-4} \mDto{2-6}{3-6}
  \eDto{4-2}{3-2} \eDto{4-3}{3-4} \mDto{4-5}{3-4} \mDto{4-6}{3-6}
  \mDto{2-3}{2-2} \eDto{2-5}{2-6}
  \mDto{4-3}{4-2} \eDto{4-5}{4-6}
  \eDto{1-4}{2-5} \mDto{1-4}{2-3}
  \eDto{5-4}{4-4} \mDto{5-4}{4-3}
  \comm{1-4}{3-4} \comm{3-4}{5-4}
  \comm{2-3}{3-2} \comm{4-3}{3-2} \comm{2-5}{3-6} \comm{4-5}{3-6}

  \diagArrow{emor, Dmor, green!70!black}{3-2}{3-1}
  \diagArrow{emor, Dmor, green!70!black}{3-4}{3-6}
  \diagArrow{mmor, Dmor, green!70!black}{3-4}{3-2}
  \diagArrow{mmor, green!70!black}{3-6}{3-7}
  \diagArrow{emor,Dmor, red!70!black}{2-2}{3-1}
  \diagArrow{mmor,Dmor, red!70!black}{1-4}{2-2}
  \diagArrow{emor,Dmor, red!70!black}{1-4}{2-6}
  \diagArrow{mmor, red!70!black}{2-6}{3-7}
  \diagArrow{emor,Dmor, blue!70!black}{4-2}{3-1}
  \diagArrow{mmor,Dmor, blue!70!black}{5-4}{4-2}
  \diagArrow{emor,Dmor, blue!70!black}{5-4}{4-6}
  \diagArrow{mmor, blue!70!black}{4-6}{3-7}
\end{diagram}
Then the composition
\[A \eDlto A'\times_{A'''} A'' \mDlto  ( (B'\times_{B'''}
  B'')  \oslash_{B'''} X'' ) \oslash_{X''} (X'' \oslash_{A'''}  (A'\times_{A'''} A'') )
   \eDrto
  B'\times_{B'''} B'' \mrto B
\]
is equivalent to both the red and the blue, completing the desired picture.
Putting these together shows that the relation defined on m-morphisms is an
equivalence relation, as desired.

Now we can work with the definition of the m-morphisms directly.  Given two
morphisms $A \msrto B$ and $B \msrto C$ their composition is defined to be
represented by the diagonal in the following square:
\begin{diagram}
  {A & A' & X & B' & B \\
    & A'' & X\times(B'\oslash_B B'') & B'\oslash_B B'' & B'' \\
    & & Z & (B'\oslash_B B'')\times Y & Y \\
    & & & C'' & C' \\
    & & & & C\\};
  \eDto{1-2}{1-1} \mDto{1-3}{1-2} \eDto{1-3}{1-4} \mto{1-4}{1-5}
  \eDto{2-5}{1-5} \mDto{3-5}{2-5} \eDto{3-5}{4-5} \mto{4-5}{5-5}
  \mto{2-4}{2-5} \eDto{2-4}{1-4} \comm{1-4}{2-5}
  \eDto{2-3}{2-4} \eDto{2-3}{1-3}
  \mDto{2-3}{2-2} \eDto{2-2}{1-2} \dist{1-2}{2-3}
  \mDto{3-4}{2-4} \mto{3-4}{3-5}
  \mDto{3-3}{2-3} \eDto{3-3}{3-4} \comm{2-3}{3-4}
  \eDto{3-4}{4-4} \mto{4-4}{4-5} \dist{3-4}{4-5}
  \eDto{2-2}{1-1} \mDto{3-3}{2-2} \eDto{3-3}{4-4} \mto{4-4}{5-5}
\end{diagram}
Here, $Z = (X\times (B'\oslash_B B''))\oslash_{B'\oslash_BB''}
((B'\oslash_BB'')\times Y)$ and $A''$ and $C''$ are uniquely determined by the
distinguished squares they are in.

To check that this is well-defined, it suffices to check that given a diagram as
in (\ref{diag:single-eq}) and a morphism represented as one of $\eDlto$,
$\mDlto$, $\eDrto$ or $\mrto$ the composition (resp. precomposition) with the
red morphism and the composition (resp. precomposition) with the blue morphism
are equivalent.  We check the case of composing with a morphism represented by
$\eDlto$; all of the other cases are analogous.  This is a straightforward
diagram chase, using Lemma~\ref{lem:Ccube} to push the diagram showing the
equivalence of the two representations along the composition; the only
nontrivial part is ensured by Lemma~\ref{lem:(B)}.

We need to check that composition is associative.  As a morphism is a formal
composition of four arrows, it suffices to check that compositions of those
component arrows is associative.  It is not necessary to worry about which
morphisms have kernel/cokernel in $\A$, since that is preserved by the
definition of composition; all we are checking is associativity.  Thus our
definition of morphism is symmetric in e-morphism and m-morphism.  In addition,
since both $\E$ and $\M$ are closed under pullbacks, by standard arguments about
span categories we know that when all three morphisms are e-morphisms or all
three morphisms are m-morphisms composition is associative.  Thus it remains to
consider the case of 2 m-morphisms and 1 e-morphism or 1 m-morphism and 2
e-morphisms.  By symmetry again it suffices to consider this second case, and,
in fact, it suffices to consider the case when the m-morphism is directed
covariantly with the composition.

Now there are $12$ cases left (three positions for the m-morphism and four
directions in which the e-morphisms can point).  Most of these have only a
single composition, so associativity holds automatically for these.  The
remaining three cases are $\mrto \erto \erto$, $\mrto \elto \elto$ and
$\mrto \erto \elto$.  The first and second of these give associative
compositions because distinguished and pseudo-commutative squares work correctly with
respect to composition.  Thus the last case is the only one of interest, which
directly follows from Corollary~\ref{cor:comp}.  The fact that the two different
compositions assemble into a cube implies that they are equivalent in $\C\bs\A$.

\noindent\textbf{Distinguished squares compose correctly}  This is true by
definition.

\noindent\textbf{There exists a $\phi$}  We must show that the subcategory of
m-isomorphisms is isomorphic to the category of e-isomorphisms by a functor
which takes objects to themselves.  To construct this functor, use
Lemma~\ref{lem:axiomC} to change a representation of an m-isomorphism as
\[A \eDlto A' \mDlto X \eDrto B' \mDrto B\]
to
\[A \mDlto A'' \eDlto X \mDrto B'' \eDrto B,\] which gives a representation of
an e-isomorphism.  Since distinguished squares are unique up to unique
isomorphism, this is an isomorphism of categories.

\noindent\textbf{Axiom (Z)} We must check that $\initial$ is initial in $\M$.  

There exists a morphism $\initial \msrto B$ for any $B$ by simply taking the
representation where all but the last morphism are the identity. We must now
check that this morphism is unique.  Suppose that we are given any diagram
\[\initial \eDlto \initial \mDlto \initial \eDrto B' \mrto B.\]
We must have $B'\in \A$ for this diagram to be valid.  The diagram
\begin{diagram}
  { & & & B \\
    & & B' & \initial \\
    & \initial & \initial & \initial \\
    \initial & \initial & \initial & \initial \\};
  \node (m-100-100) at (m-4-4) {\phantom{X}};
  
  \eDto{4-2}{3-2} \eDto{4-3}{3-3}
  \mDto{4-3}{4-2} \mDto{4-4}{4-3}
  \eDto{4-4}{3-4} \mDto{3-4}{3-3}
  \eDto{3-4}{2-4} \mDto{2-4}{2-3}
  \comm{2-3}{3-4} \comm{3-2}{4-3} \comm{3-3}{4-4}
  \eDto{4-2}{4-1} \eDto{3-2}{4-1} \mDto{3-3}{3-2}
  \eDto{3-3}{2-3} \mDto{2-3}{1-4} \mDto{2-4}{1-4}
\end{diagram}
shows that the two are equivalent.  Thus $\initial$ is horizontally initial.

\noindent\textbf{Axiom (I)} The m-morphisms which are isomorphisms are
exactly those morphisms of the form
\[A \eDlto \mDlto \eDrto \mDrto B.\]
Using this description and the listing of different kinds of distinguished
squares we can construct each of the required squares by hand.

\noindent\textbf{Axiom (M)} It suffices to check this for the m-morphisms of
$\C\bs\A$; the statement for the e-moprhisms will follow by symmetry.  Thus we
want to check that if we are given two morphisms $f,g: A \msrto B$ and a
morphism $h: B \msrto C$ in $\C\bs\A$ then if $h f = h g$ then $f = g$.  All
morphisms in $\M$ are equal, up to isomorphism, to ones represented by diagrams
$\dotp \mrto \dotp$.  Thus it suffices to assume that $h$ is of this form.  This
means that the compositions $hf$ and $hg$ are computed simply by composing the
last m-moprhism components.

The fact that $hf = hg$ implies that for any choice of representatives for $f$
and $g$, the following diagram exists:
\begin{diagram}
  {& & \dotp && B'  \\
    & \dotp && \dotp && C' \\
    A && \dotp && \dotp && C && B \\
    & \dotp && \dotp && C'' \\
    & & \dotp && B'' \\};
  \eDto{2-2}{3-1} \mDto{1-3}{2-2} \eDto{1-3}{1-5} \diagArrow{mmor,bend left}{1-5}{3-7}
  \eDto{4-2}{3-1} \mDto{5-3}{4-2} \eDto{5-3}{5-5} \diagArrow{mmor,bend right}{5-5}{3-7}
  \eDto{3-3}{2-2} \eDto{3-3}{4-2}
  \mDto{3-5}{2-4} \mDto{3-5}{4-4} \mDto{4-4}{3-3} \mDto{2-4}{3-3}
  \eDto{3-5}{2-6} \mto{2-6}{3-7}
  \eDto{3-5}{4-6} \mto{4-6}{3-7}
  \eDto{2-4}{1-3} \comm{2-2}{2-4} \eDto{2-4}{1-5}
  \mDto{2-6}{1-5} \comm{2-4}{2-6}
  \eDto{4-4}{5-3} \comm{4-2}{4-4} \eDto{4-4}{5-5} \mDto{4-6}{5-5}
  \comm{4-4}{4-6}
  \to{2-6}{4-6}^{\cong}
  \mto{3-9}{3-7}_h \diagArrow{mmor,bend left}{1-5}{3-9} \diagArrow{mmor,bend right}{5-5}{3-9}

  \node (m-100-100) at (m-3-9) {$\phantom{B}$};
  \diagArrow{densely dashed,->,red!70!black,out=90,in=90}{3-1}{100-100}^f
  \diagArrow{densely dashed,->,blue!70!black,out=270,in=270}{3-1}{100-100}_g
\end{diagram}
To show that $f = g$ it suffices to check that there exist maps $C' \mrto B$ and
$C'' \mrto B$ such that the triangle
\begin{squisheddiagram}
  { C' \\
    & B \\
    C'' \\};
  \to{1-1}{3-1}_\cong \mto{1-1}{2-2} \mto{3-1}{2-2}
\end{squisheddiagram}
commutes.  Setting these maps to be the evident ones generated by the above
diagram, we see that the given triangle must commute, as it commutes after
postcomposition with $h$ and $h$ is monic.

\noindent\textbf{Axiom (K)} As before, we prove this only for $c$; the result for $k$
follows by symmetry.

Let $f: A \msrto B$ be a morphism.  Given a representative
\[A \eDlto A' \mDlto X \eDrto B' \mrto B\]
of $f$, we can conclude that $c(f) \cong (B')^c \erto B$.  Thus if we can show
that a distinguished square as desired exists for this representative, we will
be done.  The following diagram shows that this is the case
\begin{diagram}
  { \initial & \initial & \initial & \initial & (B')^c \\
    A & A' & X & B' & B \\};
  \eDto{1-2}{1-1} \mDto{1-3}{1-2} \eDto{1-3}{1-4} \mto{1-4}{1-5}
  \eDto{2-2}{2-1} \mDto{2-3}{2-2} \eDto{2-3}{2-4} \mto{2-4}{2-5}
  \eto{1-1}{2-1} \eto{1-2}{2-2} \eto{1-3}{2-3} \eto{1-4}{2-4} \eto{1-5}{2-5}
  \dist{1-4}{2-5} \comm{1-2}{2-3}
\end{diagram}
as it is a composition of squares which are distinguished in $\C\bs\A$.

\noindent\textbf{Axiom (A)}  This holds because it holds inside $\C$ and all
distinguished squares in $\C$ are also distinguished in $\C\bs\A$. \qed

\bibliographystyle{amsalpha}
\bibliography{CZ}

\end{document}